\documentclass{amsart}
\usepackage{amsmath,amssymb,enumerate,amscd,tikz,url, color}
 
\usepackage[skip=0pt]{caption} 

\theoremstyle{plain}
\newtheorem{theo}[equation]{Theorem}
\newtheorem{crit}[equation]{Criterion}
\newtheorem{lem}[equation]{Lemma}
\newtheorem{cor}[equation]{Corollary}
\newtheorem{prop}[equation]{Proposition}
\theoremstyle{definition}
\newtheorem{Def}[equation]{Definition}

\newtheorem{Rem}[equation]{Remark} 
\newtheorem{Not}[equation]{Notation}

\newcommand{\F}{{\mathbb F}}

\newcommand{\Q}{{\mathbb Q}}

\newcommand{\W}{{\mathbb W}}
\newcommand{\Z}{{\mathbb Z}}

\newcommand{\bfa}{{\mathbf a}}

\newcommand{\bfr}{{\bf r}}
\newcommand{\bfs}{{\bf s}}
\newcommand{\bft}{{\bf t}}

\newcommand{\Kappa}{\boldsymbol \kappa}
\newcommand{\Mu}{\boldsymbol \mu}

\newcommand{\ga}{{\mathfrak a}}
\newcommand{\gb}{{\mathfrak b}}

\newcommand{\gE}{{\mathfrak E}}

\newcommand{\gf}{{\mathfrak f}}
\newcommand{\gh}{{\mathfrak h}}

\newcommand{\gm}{{\mathfrak m}}

\newcommand{\gp}{{\mathfrak p}}
\newcommand{\gP}{{\mathfrak P}}
\newcommand{\gq}{{\mathfrak q}}
\newcommand{\gR}{{\mathfrak R}}

\newcommand{\cC}{{\mathcal C}}
\newcommand{\cD}{{\mathcal D}}
\newcommand{\cE}{{\mathcal E}}

\newcommand{\cI}{{\mathcal I}}

\newcommand{\cO}{{\mathcal O}}
\newcommand{\cP}{{\mathcal P}}
\newcommand{\cR}{{\mathcal R}}
\newcommand{\cS}{{\mathcal S}}

\newcommand{\cU}{{\mathcal U}}
\newcommand{\cV}{{\mathcal V}}
\newcommand{\cW}{{\mathcal W}}

\newcommand{\Frob}{{\rm Frob}}
\newcommand{\rL}{{\rm L}}
\newcommand{\rM}{{\rm M}}
\newcommand{\rF}{{\rm F}}
\newcommand{\rR}{{\rm R}}
\newcommand{\rO}{{\rm O}}
\newcommand{\rV}{{\rm V}}

\newcommand{\wtk}{\widetilde{k}}
\newcommand{\wtF}{\widetilde{F}}

\newcommand{\wtL}{\widetilde{L}}

\newcommand{\un}[1]{\underline{#1}}

\newcommand{\Aut}{\operatorname{Aut}}

\newcommand{\CW}{\operatorname{CW}}

\newcommand{\disc}{\operatorname{disc}}

\newcommand{\End}{\operatorname{End}}
\newcommand{\Ext}{\operatorname{Ext}}

\newcommand{\rk}{\operatorname{rank}}
\newcommand{\Gal}{\operatorname{Gal}}
\newcommand{\Hom}{\operatorname{Hom}}

\newcommand{\SP}{{\rm Sp}}

\newcommand{\res}{\operatorname{res}}

\newcommand{\ord}{\operatorname{ord}}

\newcommand{\0}{\vec{0}}

\newcommand{\CWplus}{\, \dot{+} \, }

\newcommand{\ov}[1]{\overline{#1}}
\newcommand{\lr}[1]{\langle{#1}\rangle}
\newcommand{\vv}[1]{\vert {#1} \vert}

\newcommand{\fdeg}[2]{[{#1}\!:\!{#2}]}

\makeatletter
\renewcommand*\l@subsection{\@tocline{2}{0pt}{30pt}{0pt}{}}
\makeatother

\author[A. Brumer]{Armand Brumer}
\address{Department of Mathematics, Fordham University, Bronx, NY 10458, USA}
\email{brumer@fordham.edu}
\author[K. Kramer]{Kenneth Kramer*}
\address{Department of Mathematics, Queens College (CUNY), Flushing, NY 11367, USA;  Department of Mathematics, The Graduate Center of CUNY, New York, NY 10016, USA}
\email{kkramer@qc.cuny.edu}

\begin{document}
\title{Hyperelliptic $\mathcal{S}_7$-Curves of Prime Conductor}


\begin{abstract}
An abelian threefold $A_{/\Q}$ of prime conductor $N$ is {\em favorable} if its 2-division field $F$ is an $\cS_7$-extension over $\Q$ with ramification index 7 over $\Q_2$.   Let $A$ be favorable and let $B$ be a semistable abelian variety of conductor $N^d$ with $B[2]$ filtered by $d$ copies of $A[2]$.  We obtain a class field theoretic criterion on $F$ to guarantee that $B$ is isogenous to $A^d$ and {\em a fortiori}, $A$ is unique up to isogeny.
\end{abstract}

\maketitle 
\tableofcontents

\numberwithin{equation}{section}

\section{Introduction} \label{Intro}

Schoof \cite{Sch1,Sch2,Sch3} describes the semistable abelian varieties $A$ over $\Q$ with
good reduction outside the set $S$ of primes dividing a small odd integer $N$. This is done by studying the group schemes over $\Z[1/N]$ that could occur in $A[2]$ and their field of points. Grothendieck and Fontaine give upper bounds for the root discriminants of those fields and when $N\le 37$, Odlyzko provides a bound for their degree.  Then class field theory and group theory allow the determination of all such simple group schemes and their extensions by one another.

To treat larger conductors $N$, one must restrict the simple group schemes that are permitted.  Let $N$ be squarefree, let $p$ be a prime not dividing $N$ and let $\cE$ be an absolutely  simple group scheme of exponent $p$ over $\Z[1/N]$. When is there at most one isogeny class of  abelian varieties $A$ of conductor $N$ such that $A[p]\simeq \cE$? In \cite{BK5}, using algebraic number theoretic criteria, we  proved such  a uniqueness result for many isogeny classes of abelian surfaces of prime conductor.  Here we treat the  uniqueness  question for the following threefolds.
 
\begin{Def}
An abelian threefold  $A$ of prime conductor $N$ is {\em favorable} if $F = \Q(A[2])$ is the Galois closure of a field $F_0$ of degree 7 such that 2 is totally ramified in $F_0/\Q$ and the discriminant of $F_0$ is $\pm 2^{6} N$.  
\end{Def}

Assume that $A$ is favorable and let $F = \Q(A[2])$. Up to isomorphism, there is a unique subfield $K$ of $F$ of degree 21 over $\Q$.  Let $\Omega_K$ be the maximal elementary 2-extension of $K$ unramified outside $\{2,\infty\}$  with ray class conductor exponent $\gf_\gp(\Omega_K/K) \le 6$ for the unique prime $\gp$ over 2 in $K$. 
   Suppose that $L \ne F$ is the Galois closure over $\Q$ of a quadratic extension of $K$ inside $\Omega_K$.   We show in \S \ref{3Folds} that $\Gal(L/\Q)$ is isomorphic to one of three groups $G_a$, with $\vv{G_a} = 7! \cdot 2^a$ and $a = 6, 14, 20$. We say that $F$ is {\em amiable} if one of the following conditions holds: 
\begin{enumerate} [{\rm i)}] 
\item $\Omega_K = K$,   \vspace{2 pt}
\item $\fdeg{\Omega_K}{K}=2$ and $\gf_\gp(\Omega_K/K) =6$,   \vspace{2 pt}
\item $\fdeg{\Omega_K}{K}=2$, $\gf_\gp(\Omega_K/K) =4$ and the Galois closure of $\Omega_K/\Q$ has group $G_6$.
\end{enumerate}

\begin{theo} \label{MainThm}
Let $A$ be a favorable abelian threefold of prime conductor $N$ such that $F=\Q(A[2])$ is amiable. If $B$ is a semistable abelian variety of  conductor $N^d$, with $B[2]$  filtered by $d$  copies of $A[2]$, then $B$ is isogenous to $A^d$.  In particular, if $B$ is an abelian threefold of conductor $N$ with $B[2]$  isomorphic to $A[2]$, then $B$ is isogenous to $A$. 
\end{theo} 

Despite our title, the threefold $A$ need not even be principally polarized. However, Jacobians do provide the simplest examples.  Let $f$ be a monic polynomial of degree 7 in $\Z[x]$ and let $J$ be the Jacobian of the hyperelliptic curve $C\!: \, y^2 + y = f(x)$.  Then $C$ has good reduction at 2 and so $J[2]$ is a finite flat group scheme over $\Z_2$.  If $\theta$ is a root of $h(x) = 1 + 4f(x)$, then $F = \Q(J[2])$ is the Galois closure of $F_0 = \Q[\theta]$.  When the discriminant of $h(x)$ is $\pm 2^{6} N$, with $N$ a prime, $J$ is a favorable threefold.

In \S \ref{data}, we report on a range of examples which suggest  that there might be infinitely many abelian threefolds with $\Q(A[2])$ amiable. Magma \cite{MAG} was used for those computations and to verify some of the group theoretical assertions in \S \ref{3Folds}.

\section{Overview}  \label{Overview}

Let $\ov{K}$ denote an algebraic closure of the field $K$.  When $R$ be a Dedekind domain with quotient field $K$, we use calligraphic letters for finite flat group schemes $\cV$ over $R$ and the corresponding Roman letter for the Galois module $V=\cV(\ov{K})$. The order of $\cV$ is the rank over $R$ of its affine algebra, or equivalently the order of the finite abelian group $V$.   As usual, $\Mu_n$ is the group scheme of $n$-th roots of unity. 

The next two sections contain the local theory of the relevant group schemes and comprise the most delicate part of our paper.   We treat simple 3-dimensional biconnected group schemes $\cE_i$ over the ring of integers of an unramified extension $K$ of $\Q_p$, to exhibit the dependence on $p$.  In \S \ref{Honda}, we describe the Honda system of  $\cE_i$.  Then we classify the Honda systems of the extensions $\cV$ of $\cE_i$ by $\cE_j$ of exponent $p$ for $1 \le i,j \le 2$.  

 In \S \ref{LocalFieldofPoints}, we carefully determine the field of points for each of the group schemes above.  We obtain the bound $p^2+p$ for the conductor exponent of the elementary abelian $p$-extension $K(V)/K(E_i,E_j)$ using a new conductor evaluation in Appendix \ref{ArtinSchreier}, perhaps of independent interest.  This improvement on Fontaine's general bound for group schemes of exponent $p$ relies on the biconnectedness of $\cE_i$ and is critical for our applications.  
 
The more precise local analysis needed to compare with the global results, is restricted to  $p=2$ and $K = \Q_2$.  Then $\Q_2(V)$ is contained in the maximal elementary 2-extension $T$ of $F = \Q_2(E_1,E_2)$ of conductor exponent $6=2^2+2$.  The extension class $[V]$ of $E_i$ by $E_j$ as local Galois modules gives rise to a cohomology class in 
$$
H^1(\Gal(T/\Q),\Hom_{\F_2}(E_i,E_j)).
$$
Let $\Delta = \Gal(F/\Q_2)$ and $\rR = \F_2[\Delta]$.   Using local reciprocity, we determine the $\rR$-module structure of $\Gal(T/F)$ and thereby control the cohomology classes associated to extensions $V$.

\vspace{5 pt}

We next introduce our generalization of favorability for arbitrary $g$.  The arithmetic of favorable fields, with applications to abelian varieties, is developed in  \S \ref{FavFieldsSection}.  
 
\begin{Def}  \label{FavField} 
Let $N$ be odd and square-free.  
\begin{enumerate}[\, i)]
\item  A field $F_0$ of degree $2g+1$ over $\Q$ is {\em favorable} if its discriminant has the form $\pm 2^{2g} N$ and the prime over 2 in $F_0$ is totally ramified.  \vspace{2 pt}
\item An abelian variety of dimension $g$ will also be called {\em favorable} if $\Q(A[2])$ is the Galois closure of a favorable field of degree $2g+1$. 
\end{enumerate}
\end{Def}

\noindent For the proof of Theorem \ref{MainThm}, we need the following concepts.

\begin{Def} \label{semistable}
Let $N$ be odd.  A 2-primary finite flat group scheme $\cV$ over $\Z[1/N]$ is {\em semistable} if, for each prime $q$ over $N$ in $\Q(V)$, the inertia group $\lr{\sigma_\gq}$ in $\Gal(\Q(V)/\Q)$ is cyclic and $(\sigma_\gq-1)^2$ annihilates the Galois module $V=\cV(\ov{\Q})$.  \vspace{2 pt}
\end{Def}
 
Fix an odd prime  $N$ and a simple semistable finite flat group scheme $\cE$ over $R =  \Z[1/N]$.  In \S \ref{3Folds}, as a special case of \cite[\S 3]{BK5}, we consider the category $\un{E}$ of finite flat group schemes $\cV$ over $R$ satisfying: 
\begin{enumerate} [{\bf \hspace{3 pt} E1.}]
\item Each composition factor of  $\cV$ is isomorphic to $\cE$.  \vspace{2 pt}
\item The group scheme $\cV$ is semistable over $N$.  \vspace{2 pt}
\item If $d$ is the multiplicity of $E$ in the semi-simplification $V^{ss}$ of $V$, then the two Artin conductor exponents at $N$ satisfy
$
\gf_N(V) = \gf_N(V^{ss}) = d.
$
\end{enumerate}

\vspace{5 pt}

The proof of Theorem \ref{MainThm} now reduces to showing that $\Ext^1_{[2], \un{E}}(\cE,\cE) = 0$, i.e., in the category $\un{E}$, all extensions $\cV$ of $\cE$ by $\cE$ such that $2\cV = 0$ are split.  By the local analysis above, the conductor exponent of the elementary 2-extension $\Q_2(V)/\Q_2(E)$ is at most 6.  Semistability at $N$ is used to deduce that, for some  $a$ in $\{0,6,14,20,29,35 \}$, $\Gal(\Q(V)/\Q)$ is isomorphic to the group $G_a$ of order $2^a \cdot 7!$ defined in Notation \ref{Ca}.  Proposition \ref{Cond4} gives better conductor bounds over 2 for certain $a$'s, based on \S\ref{LocalCorners}.  The improved conductor bounds and Proposition \ref{MakeGaFields} motivate the definition of amiability.  By Proposition \ref{ext2}, amiability implies the vanishing of $\Ext^1_{[2],\un{E}}(\cE,\cE)$ and Theorem \ref{MainThm} follows. 

Finally, we review some standard notation and facts from representation theory needed in both the local and global setting. 

\subsection*{Parabolics}  \label{Parabolics}

Let $G$ be a group and let $\rho_{E}$ be the representation afforded by a $G$-module $E$ over a field $k$.  The exact sequence of $k[G]$-modules 
$$
0 \to E_1 \xrightarrow{\iota} V \xrightarrow{\pi} E_2 \to 0
$$ 
gives rise to a cocycle $\psi_V\!: G \to \gh$ such that 
\begin{equation} \label{RepFromPsi}
\rho_V(g) = \left[\begin{matrix} \rho_{E_1}(g)  \, & \,  \psi_V(g) \, \rho_{E_2}(g) \\ 0 \, & \, \rho_{E_2}(g) \end{matrix} \right]
\end{equation}
where  $g$ in $G$ acts on $m$ in $\gh = \Hom_k(E_2,E_1)$ by $g(m) = \delta_1 m \delta_2^{-1}$, with $\delta_i = \rho_{E_i}(g)$.  The class $[V]$ in $\Ext^1_{k[G]}(E_2,E_1)$ corresponds to that of $[\psi_V]$ in $H^1(G,\gh)$.  Indeed, if $x$ is in $E_2$ and $s $ in $\Hom_k(E_2,V)$ is any section of $\pi$, then $g(s(g^{-1}x))-s(x) = \iota(y)$ for some $y$ in $E_1$ and one can take $\psi_V(g)(x) = y$.  Thus,  in the category of $k[G]$-modules, the extension classes of $E_2$ by $E_1$ under Baer sum form a group isomorphic to $H^1(G,\gh)$.   If $N$ is a normal subgroup of $G$ contained in $\ker \rho_V$, then $[\psi_V]$ comes by inflation from a unique class in $H^1(G/N,\gh)$ which we also denote by $[\psi_V]$. 

The image $\rho_V(G)$ lies in a {\em parabolic} matrix group
\begin{equation} \label{defP}
\cP = \cP_{E_1, E_2} = \left\{g = \left[\begin{smallmatrix} \delta_1 & m \\ 0 & \delta_2 \end{smallmatrix} \right] \, \vert \, \delta_i = \rho_{E_i}(g), \, m \in M_{n_1,n_2}(k) \right\}
\end{equation}
with $n_i = \dim_k E_i$.  If $H_i = \{ g \in G \, \vert \, g_{\vert E_i} = 1 \}$ and $\Delta_i = G/H_i$, then $E_i$ is a faithful $k[\Delta_i]$-module.  Any normal subgroup $H$ of $G$ acting trivially on both $E_1$ and $E_2$ satisfies
$$ 
\rho_V(H) \, \subseteq \,  \left\{\left[\begin{smallmatrix} 1 & m \\ 0 & 1 \end{smallmatrix} \right] \in \cP \, \vert \, m \in M_{n_1,n_2}(k) \right\}.
$$
The inflation-restriction sequence gives:
\begin{equation} \label{InfRes}
0 \to H^1(G/H, \gh) \xrightarrow{\rm inf} H^1(G,\gh) \xrightarrow{\rm res} H^1(H,\gh)^{G/H}.
\end{equation}

Suppose that $E_1 = E_2 = E$ and let $V$ represent a class in $\Ext^1_{k[G]}(E,E)$. 
Let $\Delta = \rho_E(G)$ and let $\pi\!: \cP \to \Delta$ be the projection $\left[\begin{smallmatrix} \delta & h \\ 0 & \delta \end{smallmatrix} \right] \mapsto \delta$.    Note that $\pi \rho_V = \rho_E$.

\begin{lem} \label{CocycleRep}
Let $\psi_V$ and $\psi_W$ correspond to representatives $V$ and $W$ for classes in $\Ext^1_{k[G]}(E,E)$ as in \eqref{RepFromPsi}.  Then $[W] = [V]$ if and only if
\begin{equation} \label{conj}
\rho_W(g) = \left[\begin{smallmatrix} 1  &  m  \\ 0 &1 \end{smallmatrix} \right] \rho_V(g) \left[\begin{smallmatrix} 1 & -m  \\ 0 & \hspace{5 pt} 1 \end{smallmatrix} \right].
\end{equation}
for some $m$ in $\gh$ and all $g$ in $G$.  If $\rho_W(g) = \alpha \, \rho_V(g) \, \alpha^{-1}$ for some $\alpha$ in $\cP$ and all $g$ in $G$ and the center of $\Delta$ is trivial, then $\alpha = \left[\begin{smallmatrix} 1  &  m  \\ 0 &1 \end{smallmatrix} \right]$ for some $m$ in $\gh$ and $[W] = [V]$.\end{lem}

\begin{proof}
Formula \eqref{conj} is equivalent to $\psi_W(g) = \psi_V(g) + m - g(m)$, so  $\psi_W-\psi_V$ is a coboundary for $H^1(G,\gh)$.   Suppose, instead that $\rho_W(g) = \alpha \, \rho_V(g) \, \alpha^{-1}$ for all $g$ in $G$ and set $\alpha =  \left[\begin{smallmatrix}  a  &  m  \\ 0 & a \end{smallmatrix} \right]$.  By multiplying out the matrices, we find that $a$ is in the center $Z(\Delta)$, so $a = 1$ and then \eqref{conj} implies that $[W] = [V]$.
\end{proof}

\section{Some Honda systems} \label{Honda}  \numberwithin{equation}{section}

In this section, we construct the Honda systems associated to group schemes over $\Z_2$ occurring in $A[2]$ when $A$ is a favorable abelian threefold.  To begin, we review some basic material on Honda systems found in \cite{BrCo,Con2,Fon3}.   For a perfect field $k$  of characteristic $p$, let $\W = \W(k)$ be the ring of Witt vectors and $K$ its field of fractions.   Let $\sigma:\W \to \W$  be the Frobenius automorphism characterized by $\sigma(x)\equiv x^p  \pmod{p}$ for $x$ in $\W$. The Dieudonn\'e ring $D_k=\W[\rF,\rV]$ is generated by the Frobenius operator $\rF$ and Verschiebung operator $\rV$. We have $\rF\rV=\rV\rF=p$, $\rF a=\sigma(a)\rF$ and $\rV a=\sigma^{-1}(a)\rV$ for all $a$ in $\W$. 

 A {\em finite Honda system} over $\W$ is a pair $(\rM,\rL)$ consisting of a left $D_k$-module $\rM$ of finite  $\W$-length and a $\W$-submodule $\rL$ with $\rV\!: \rL\to \rM$ injective and the induced map $\rL/p\rL\to \rM/\rF\rM$ an isomorphism. If $\rF$ is nilpotent on $\rM$, then $(\rM,\rL)$ is said to be {\em connected}.  Morphisms are defined in the obvious manner. 
 
\begin{lem} \label{pM}
Let $(\rM,\rL)$ be a Honda system of exponent $p$.  Then $\rM=\rL + \rF\rM$ is a direct sum, $\ker \rF=\rV\rL=\rV\rM$, $\dim \ker \rF = \dim \rL$ and $\ker \rV=\rF\rM$.
\end{lem}

Let $\widehat{\CW}_k$ denote the formal $k$-group scheme associated to the {\em Witt covector} group functor $\CW_k$.  When $k'$ is a finite extension of $k$ and $K'$ is the field of fractions of $\W(k')$, we have $\CW_k(k') \simeq K'/\W(k')$.  For any $k$-algebra $R$, let $D_k$ act on elements $\bfa = ( \dots, a_{-n}, \dots, a_{-1},a_{0})$ of $\CW_k(R)$ by 
\begin{eqnarray*}
\rF \bfa &=& ( \dots,a_{-n}^p, \dots, a_{-1}^p,a_{0}^p),  \hspace{10 pt} 
\rV \bfa = ( \dots, a_{-(n+1)}, \dots, a_{-2},a_{-1}), \\ 
\text{and} \quad \dot{c} \, \bfa &=& ( \dots, c^{p^{-n}}a_{-n}, \dots, c^{p^{-1}}a_{-1},ca_{0}), 
\end{eqnarray*}
where $\dot{c}$ in $\W$ is the Teichm\"uller lift of $c$.

The Hasse-Witt exponential map is a homomorphism of additive groups
\begin{equation}  \label{HW}
\xi: \, \widehat{CW}_k(\cO_{\ov{K}}/p\cO_{\ov{K}}) \to \ov{K}/p\cO_{\ov{K}} 
\end{equation}
defined by $\xi(\dots,a_{-n},\dots,a_{-1},a_0) = \sum \, p^{-n} \, \tilde{a}_{-n}^{p^n}$, independent of the choice of lifts $\tilde{a}_{-n}$ in $\cO_{\ov{K}}$.  If $\cU$ is the group scheme of a Honda system $(\rM,\rL)$, the points of the Galois module $U$ correspond to $D_k$-homomorphisms $\varphi\!: \, \rM \to  \widehat{CW}_k(\cO_{\ov{K}}/p\cO_{\ov{K}})$ such that $\xi(\varphi(\rL)) = 0$ and the action of $G_K$ on $U(\ov{K})$ is induced from its action on $\widehat{CW}_k(\cO_{\ov{K}}/p\cO_{\ov{K}})$.  

 We write $\CWplus$ for the usual Witt covector addition \cite[p.\! 242]{Con2} and state some related elementary facts.  For $q$ a power of $p$ and $x, y$ in $\cO_{\ov{K}}/p\cO_{\ov{K}}$, the congruence $\Phi_q(x,y) \equiv ((\tilde{x}+\tilde{y})^q-\tilde{x}^q-\tilde{y}^q)/q \pmod{p\cO_{\ov{K}}}$ defines a unique, possibly non-integral element of $\ov{K}/p\cO_{\ov{K}}$, independent of the choices of lifts $\tilde{x}, \tilde{y}$ in $\cO_{\ov{K}}$.   The binomial theorem yields the following estimates:

\begin{lem} \label{ppower}
{\rm (i)} $
\ord_p((\tilde{x} + \tilde{y})^q - \tilde{x}^q - \tilde{y}^q) \ge 1 + q \min\{\ord_p(\tilde{x}), \ord_p(\tilde{y}) \}.
$ 

\vspace{2 pt}

\hspace{46 pt} {\rm (ii)} If $\alpha, \beta \in \cO_{\ov{K}}$ and $\alpha \equiv \beta \pmod{p}$, then $\textstyle{\frac{1}{p} \alpha^p \equiv \frac{1}{p} \beta^p \pmod{p}}$.  \qed  
\end{lem}

Write $(\0,x_{-n},\dots,x_0)$ for $(\dots,0,0,x_{-n},\dots,x_0)$ in $\widehat{CW}_k(\cO_{\ov{K}}/p\cO_{\ov{K}})$.   A routine calculation using the formulas in \cite{Abr,Con2} gives:

\begin{lem} \label{Wittadd} 
Addition in $\widehat{CW}_k(\cO_{\ov{K}}/p\cO_{\ov{K}})$ specializes to:
$$
(\0,u_4,u_3,u_2,u_1,u_0) \CWplus (\0,v_2,v_1,v_0)  = (\0,u_4,u_3,u_2+v_2,w_1,w_0)
$$
where $w_1 = u_1+v_1-\Phi_p(u_2,v_2)$ and 
$$
w_0=u_0+v_0+ \frac{1}{p} (u_1^p+v_1^p)-\Phi_{p^2}(u_2,v_2)- \frac{1}{p} (u_1+v_1-\Phi_p(u_2,v_2))^p. \hspace{20 pt} \qed
$$
\end{lem}

For any subset $S$ of a $k$-vector space, let $\lr{S}$ denote the subspace spanned by $S$. 

\begin{prop}  \label{LocalSimple}
The biconnected finite flat group schemes of order $p^3$ over $\W$ are described by the Honda systems $(\rM_1,\rL_1)$ or $(\rM_2,\rL_2)$ of the following form:
\begin{enumerate}[\rm i)]  
\item  For some  $\lambda$ in $k^\times$, $\rM_1$ has a $k$-basis $x_1,x_2,x_3$ such that $\rL_1 = \lr{x_1}$,  
\begin{equation} \label{VFE1}
\rV=\left[\begin{smallmatrix}0&0&0\\ 1&0&0\\ 0&0&0\end{smallmatrix}\right] \quad \text{and} \hspace{12 pt} \rF=\left[\begin{smallmatrix}0&0&0\\ 0&0&\lambda \\ 1&0&0 \end{smallmatrix}\right].
\end{equation}
\item For some  $\lambda'$ in $k^\times$, $\rM_2$ has a $k$-basis $y_1,y_2,y_3$  such that $\rL_2 = \lr{y_1,y_2}$, 
\begin{equation} \label{VFE2}
\rV=\left[\begin{smallmatrix}0&0&0\\ 1&0&0\\ 0&\lambda' &0\end{smallmatrix}\right] \quad \text{and} \hspace{12 pt} \rF=\left[\begin{smallmatrix}0&0&0\\ 0&0&0 \\ 1&0&0 \end{smallmatrix}\right].
\end{equation}
\end{enumerate}
\end{prop}

\begin{proof}
Refer to Lemma \ref{pM} as needed.   Since $(\rM,\rL)$ is biconnected, $\dim \rL = 1$ or 2.   In the first case, let $\rL_1 = \lr{x_1}$ and $x_2 = \rV x_1$.   Then $x_1$ and $x_2$ are linearly independent over $k$ because $\rV$ is injective on $\rL_1$ and nilpotent.   By Lemma \ref{pM}, $\dim \ker \rV = 2$, so $\rV ^2 = 0$.  Then $\rV x_2 = 0$, so there is a $k$-basis $x_2,z$ for $\ker \rV$ and $\rM_1 = \lr{x_1,x_2,z}$.  Since $\rF \rM_1 = \ker \rV$, we can write $\rF z = \lambda x_2 + \mu z$ with $\lambda, \mu$ in $k$.  But $\mu = 0$ because $\rF$ is nilpotent and $\rF x_2 = 0$.  Let $\rF x_1 = a x_1 + b x_2 + c z$ and use $\rV \rF = 0$ to find that $a = 0$.    Then $\lambda$ and $c$ are units because $\dim \rF \rM_1 = 2$.  Conclude that $x_1,x_2$ and $x_3 = \rF x_1$ form a basis for $\rM_1$ affording the matrix representations of $\rV$ and $\rF$ in \eqref{VFE1}.

In the second case, let $\dim \rL_2 = 2$. Then $\rV \rM_2 = \rV \rL_2 = \ker \rF$ is 2-dimensional and $\rF \rM_2 = \ker \rV$ is 1-dimensional.  In particular $\rF^2 = 0$, so that $\rM_2 = \rL_2 + \rF \rM_2$ gives $\rF \rM_2 = \rF \rL_2$.    Because $\rV$ is nilpotent and injective on $\rL_2$, we cannot have $\rV \rL_2 = \rL_2$ and so $\rV \rL_2 \cap \rL_2$ is 1-dimensional, say generated by $y_2 = \rV y_1$ with $y_1$ in $\rL_2$.   Thus $\rF y_2 = 0$.   In addition, $y_1$ and $y_2$ are linearly independent, so $\rL_2 = \lr{y_1,y_2}$.   Let $y_3 = \rF y_1$.  Then $\rM_2 = \rL + \rF \rL_2 = \lr{y_1,y_2, y_3}$ and $\rF y_3 = \rF^2 y_1 = 0$.  Write $\rV y_2 = a y_1 + b y_2 + c y_3$, use $\rF \rV = 0$ to get $a = 0$ and use nilpotence of $\rV$ to get $b = 0$.  This proves \eqref{VFE2} with $\lambda' = c$.
\end{proof}

\begin{cor} \label{basicdual}
The Cartier dual $(\rM_2^*,\rL_2^*)$ of $(\rM_2,\rL_2)$ is a Honda system of the form {\rm \eqref{VFE1}} with parameter $(\lambda')^{p ^2}$.  It is isomorphic to $(\rM_1,\rL_1)$ with parameter $\lambda$ precisely when there is some $r$ in $k^\times$ such that $\lambda = r^{1-p^3} (\lambda')^{p^2}$ in $k^\times$.
\end{cor}

\begin{proof}
Let $\rV$ and $\rF$ be the Verschiebung and Frobenius matrices in \eqref{VFE2} with respect to the basis  $y_1, y_2, y_3$ for $\rM_2$.  With respect to the dual basis $y_1^*,y_2^*,y_3^*$ of $\rM_2^*$, the Verschiebung and Frobenius for $(\rM_2^*,\rL_2^*)$ are given by
\begin{equation*} 
\rV^*= \sigma^{-1}(\rF^{\rm tr}) = \left[\begin{smallmatrix}0&0&1\\ 0 &0&0\\ 0&0&0\end{smallmatrix}\right] \quad \text{and} \hspace{12 pt} \rF^*= \sigma (\rV^{\rm tr}) = \left[\begin{smallmatrix}0&1&0\\ 0&0& \sigma(\lambda') \\ 0 &0&0 \end{smallmatrix}\right],
\end{equation*}
where $\rV^{\rm tr}$ and $\rF^{\rm tr}$ are the transposes of $\rV$ and $\rF$.  We have $\rL_2^* = \lr{y_3^*}$, since it is the annihilator of $\rL_2 = \lr{y_1,y_2}$.  Set 
$$
X_1 = y_3^*,  \quad X_2 = \rV^* X_1 =  y_1^*, \quad X_3 = \rF^* X_1 = \sigma(\lambda') y_2^*.
$$
Then $\rF^* X_3 = \sigma^2(\lambda') X_2$, so $X_1, X_2, X_3$ is the basis for a Honda system of the form \eqref{VFE1} with parameter $\lambda^* = \sigma^2(\lambda')$.  

The only freedom in the standard basis $x_1,x_2,x_3$ for $\rM_1$ in \eqref{VFE1} is to replace $x_1$ by $\tilde{x}_1 = \sigma(r)x_1$ for some $r$ in $k^\times$.  Then $\tilde{x}_2 = \rV \tilde{x}_1 = r x_2$ and $\tilde{x}_3 = \rF \tilde{x}_1 = \sigma^2(r)x_3$.  Since $\rF \tilde{x}_3 = \frac{\sigma^3(r)}{r} \lambda \tilde{x}_2$, the parameter for the new basis is $\tilde{\lambda} = \frac{\sigma^3(r)}{r}\lambda$.  Hence $(\rM_1,\rL_1)$ is isomorphic $(\rM_2^*,\rL_2^*)$ if and only if $\tilde{\lambda} = \lambda^*$ for some $r$.  

As a check, with respect to the dual basis $x_1^*,x_2^*,x_3^*$ for $(\rM_1^*,\rL_1^*)$,   Verschiebung and Frobenius are given by
\begin{equation*} 
\rV^*=  \sigma^{-1}(\rF^{\rm tr}) = \left[\begin{smallmatrix}0&0&1\\ 0 &0&0\\ 0&\sigma^{-1}(\lambda)&0\end{smallmatrix}\right] \quad \text{and} \hspace{12 pt} \rF^*= \sigma (\rV^{\rm tr}) = \left[\begin{smallmatrix}0&1&0\\ 0&0&0\\ 0&0&0 \end{smallmatrix}\right].
\end{equation*}
Then $Y_1 = x_2^*$, $Y_2 = \rV^*Y_1 =  \sigma^{-1}(\lambda) x_3^*$, $Y_3 = \rF^*Y_1 = x_1^*$ is a standard basis for a Honda system of the form \eqref{VFE2}  and $\rL_1^* = \lr{Y_1,Y_2}$ is the annihilator of $x_1^*$.  Since $V^* Y_2 = \sigma^{-2}(\lambda) Y_3$, the parameter for the basis $Y_1,Y_2,Y_3$ is $\sigma^{-2}(\lambda)$.  Taking the double dual gives the expected isomorphism.  
\end{proof}

Let $(\rM_1,\rL_1)$ be given by {\rm \eqref{VFE1}} and $(\rM_2,\rL_2)$ by {\rm \eqref{VFE2}}, with respective parameters $\lambda$ and $\lambda'$ in $k^\times$, fixed for the rest of this section.   For each pair $(i,j)$, $1 \le i, j \le 2$, we parametrize extensions of the form:
\begin{equation} \label{MExt}
0 \to (\rM_i,\rL_i) \xrightarrow{\iota} (\rM,\rL) \xrightarrow{\pi} ( \rM_j,\rL_j) \to 0, \text{ with } p\rM = 0.
\end{equation}

\begin{prop} \label{E1E1}
Let $(i,j) = (1,1)$.  For each triple $\bfs_{11} = (s_1,s_2,s_3)$ in $k^3$, there is a Honda system $(\rM,\rL)$ as in {\rm \eqref{MExt}} with $k$-basis $e_1, \dots, e_6$, such that $\rL = \lr{e_1,e_4}$, $\iota(x_1) = e_1$, $\pi(e_4) = x_1$, 
\vspace{2 pt}
$$
\rV=\left[\begin{array}{ccc|ccc}
              0&0&0&0& 0 &0 \\
              1&0&0&0& -\sigma^{-1} (s_1) &0  \\
              0&0&0&0&0&0  \\
                \hline
              0&0&0&0&0&0   \\ 
              0&0&0&1&0&0    \\ 
              0&0&0&0&0&0                                                                                
                                             \end{array}\right] \quad \text{and} \quad \,
\rF= \left[\begin{array}{ccc|ccc}
              0&0&0&0&0& \lambda s_1
           \\ 0&0&\lambda&0&0&  \lambda s_2
           \\ 1&0&0&0&0& \lambda s_3
           \\ \hline    0&0&0&0&0&0
           \\ 0&0&0&0&0&\lambda
           \\ 0&0&0&1&0&0  \end{array}\right]. 
$$
Conversely, every Honda system $(\rM,\rL)$ as in \eqref{MExt} has this form for suitable choice of $\bfs_{11}$.
\end{prop}

\begin{proof}
For a given {\em parameter vector} $\bfs_{11} = (s_1,s_2,s_3)$, this construction satisfies the conditions for a Honda system.  Conversely, given $(\rM,\rL)$ as in \eqref{MExt}, let $e_t = \iota(x_t)$ for $t =1,2,3$.   Exactness of  
$
0 \to \rL_1  \xrightarrow{\iota} \rL  \xrightarrow{\pi} \rL_1 \to 0
$ 
implies that $e_1 \ne 0$ is in $\rL$ and we can extend  to a basis for $\rL$ by adjoining $e_4$ such that $\pi(e_4) = x_1$.  We verify that $e_5 = \rV e_4$ and $e_6 = \rF e_4$ are lifts of $x_2$ and $x_3$ respectively:
$$
\pi(e_5) = \pi(\rV e_4) = \rV(\pi(e_4))= \rV x_1 = x_2,  \quad \pi(e_6) = \rF(\pi(e_4)) = \rF x_1 = x_3.
$$
Since $\pi(\rF e_6) = \rF x_3 = \lambda x_2 = \pi(\lambda e_5)$, we can write
$$
\rF e_6 = \lambda e_5 + \lambda( s_1 e_1 + s_2 e_2 + s_3 e_3)
$$
for some scalars $s_1$, $s_2$, $s_3$ in $k$.  It follows from $\rV \rF e_6 = 0$ that $\rV e_5 = -\sigma^{-1} (s_1) e_2$.  The matrix representations of $\rV$ and $\rF$ are completed by using $\rV e_6 = \rV \rF e_4 = 0$ and $\rF e_5 = \rF \rV e_4 = 0$.  
\end{proof}

\begin{Rem} \label{BasisChange11}
We examine possible changes of the above basis, if the matrices for $\rV$ and $\rF$ in the new basis have the same shape while keeping $\lambda$ and $\lambda'$ fixed.  Any other lift of $x_1$ to $\rL$ has the form $\tilde{e}_4 = e_4 + r e_1$.  Then, to obtain matrices for $\rV$ and $\rF$ with the above shape, $\tilde{e}_5 = e_5 + \sigma^{-1}(r) e_2$ and $\tilde{e}_6 = e_6+\sigma(r)e_3$.  Hence 
\begin{eqnarray*}
\rF \tilde{e}_6 &=& \lambda (e_5 + s_1 e_1 + s_2 e_2 + s_3 e_3) + \sigma^2(r) \rF e_1 \ \\
& = &  \lambda (\tilde{e}_5 - \sigma^{-1}(r) e_2+ s_1 e_1 + s_2 e_2 + s_3 e_3) + \sigma^2(r) \lambda e_2. 
\end{eqnarray*}
Thus $\sigma(\tilde{s}_2) = s_2 + \sigma^3(r) - r$, but the other parameters are preserved, suggesting that $s_2$ is related to the unramified degree $p$ extension of $F$.
\end{Rem}

\begin{prop} \label{E2E1} 
Let $(\rM,\rL)$ represent an extension of the form {\rm \eqref{MExt} with $(i,j) = (2,1)$.  There is a $k$-basis $e_1, \dots, e_6$ for $\rM$} and a parameter $\bfs_{21} = s$ in $k$ such that $\iota(y_1) = e_1$, $\pi(e_4) = x_1$, $\rL = \lr{ e_1,e_2,e_4}$,   
\vspace{2 pt}
$$
\rV=\left[\begin{array}{ccc|ccc}
              0&0&0&0& 0 &0 \\
              1&0&0&0& s &0  \\
              0&\lambda'&0&0&0&0  \\
                \hline
              0&0&0&0&0&0   \\ 
              0&0&0&1&0&0    \\ 
              0&0&0&0&0&0                                                                                
                                             \end{array}\right] \quad \text{and} \quad \,
\rF= \left[\begin{array}{ccc|ccc}
              0&0&0&0&0& -\lambda \sigma(s)
           \\ 0&0&0&0&0&0
           \\ 1&0&0&0&0&0
           \\ \hline    0&0&0&0&0&0
           \\ 0&0&0&0&0&\lambda
           \\ 0&0&0&1&0&0  \end{array}\right]. 
$$
\end{prop}

\begin{proof}
For $j =1,2,3$, let $e_j = \iota(y_j)$.   Since $0 \to \rL_2  \xrightarrow{\iota} \rL  \xrightarrow{\pi} \rL_1 \to 0$ is exact, $e_1$ and $e_2$ are in $\rL$ and we can extend  to a basis for $\rL$ by adjoining $\tilde{e}_4$ such that $\pi(\tilde{e}_4) = x_1$.   Define $\tilde{e}_5 = \rV \tilde{e}_4$ and $\tilde{e}_6 = \rF \tilde{e}_4$, so that $\pi(\tilde{e}_5) = \rV( \pi(\tilde{e}_4)) = \rV x_1 = x_2$ and $\pi(\tilde{e}_6) = \rF (\pi(\tilde{e}_4)) = \rF x_1 = x_3$.  
Since $\pi(\rF \tilde{e}_6) = \rF x_3 = \lambda x_2 = \pi(\lambda \tilde{e}_5)$, we have
$$
\rF \tilde{e}_6 = \lambda (\tilde{e}_5 + s_1 e_1 + s_2 e_2 + s_3 e_3) 
$$
for some scalars $s_1, s_2, s_3$ in $k$.

We can choose a different lift $e_4 = \tilde{e}_4 + \alpha_1 e_1 + \alpha_2 e_2$ of $x_1$ in $\rL$ with $\alpha_1, \alpha_2$ in $k$ to be determined.  Then we have the lift
$$
e_5 = \rV e_4 = \tilde{e}_5 + \sigma^{-1}(\alpha_1) e_2 + \sigma^{-1}(\alpha_2)\lambda'e_3
$$
of $x_2$ and the lift $e_6 = \rF e_4 = \tilde{e}_6 + \sigma(\alpha_1) e_3$
lof $x_3$.  We find that 
\begin{eqnarray*}
\rF e_6 &=& \lambda (\tilde{e}_5 + s_1 e_1 + s_2 e_2 + s_3 e_3) \\
            &=&  \lambda (e_5 -  \sigma^{-1}(\alpha_1) e_2 - \sigma^{-1}(\alpha_2)\lambda' e_3 + s_1 e_1 + s_2 e_2 + s_3 e_3)
\end{eqnarray*} 
Choose $\alpha_1 = \sigma(s_2)$ and $\alpha_2 = \sigma((\lambda')^{-1} s_3)$, so that $\rF e_6 = \lambda( e_5 + s_1e_1)$.  It follows from $\rV \rF e_6 = 0$ that $\rV e_5 = -\sigma^{-1}(s_1) e_2$.  Finally, let $s = -\sigma(s_1)$.
\end{proof}

\begin{prop} \label{E1E2} 
Let $(\rM,\rL)$ represent an extension of the form {\rm \eqref{MExt} with $(i,j) = (1,2)$.  There is a $k$-basis $e_1, \dots, e_6$ for $\rM$} and a parameter vector $\bfs_{12} = (s_1, s_2, s_3, s_4)$ in $k^4$ such that $\iota(x_1) = e_1$, $\pi(e_4) = y_1$, $\rL = \lr{ e_1,e_4,e_5}$,   
\vspace{2 pt}
$$
\rV=\left[\begin{array}{ccc|ccc}
              0&0&0&0& \lambda' s_1 &0 \\
              1&0&0&0& \lambda' s_2 &0  \\
              0&0&0&s_4& \lambda' s_3&0  \\
                \hline
              0&0&0&0&0&0   \\ 
              0&0&0&1&0&0    \\ 
              0&0&0&0&\lambda'&0                                                                                
                                             \end{array}\right] \hspace{2 pt} \text{and} \hspace{6 pt} 
\rF= \left[\begin{array}{ccc|ccc}
              0&0&0&0&0&0
           \\ 0&0&\lambda&0&- \lambda \sigma(s_4)& - \lambda  \sigma(s_3) 
           \\ 1&0&0&0&0& - \sigma(s_1)
           \\ \hline    0&0&0&0&0&0
           \\ 0&0&0&0&0& 0
           \\ 0&0&0&1&0& 0  \end{array}\right]. 
$$
\end{prop}

\begin{proof}
For $j =1,2,3$, let $e_j = \iota(x_j)$.   Since $0 \to \rL_1  \xrightarrow{\iota} \rL  \xrightarrow{\pi} \rL_2 \to 0$ is exact, $e_1$ is in $\rL$ and we can extend  to a basis for $\rL$ by adjoining $\tilde{e}_4$ and $\tilde{e}_5$ such that $\pi(\tilde{e}_4) = y_1$ and $\pi(\tilde{e}_5) = y_2$.  Then the most general lifts of $y_1$ and $y_2$ to $\rL$ are $e_4 = \tilde{e}_4 + \alpha_1 e_1$ and $e_5 = \tilde{e}_5+ \alpha_2 e_1$ respectively, with $\alpha_1, \alpha_2$ in $k$ to be determined.  Since 
$$
\pi(\rV \tilde{e}_4) = \rV y_1 = y_2 = \pi(\tilde{e}_5)
$$
we can write $\rV(\tilde{e}_4) = \tilde{e}_5 + t_1 e_1+ t_2 e_2 + t_3 e_3$ for some $t_1,t_2,t_3$ in $k$.   Then
\begin{eqnarray*}
\rV e_4 &=& \rV \tilde{e}_4 + \sigma^{-1}(\alpha_1) e_2 =  \tilde{e}_5 + t_1 e_1+ t_2 e_2 + t_3 e_3  + \sigma^{-1}(\alpha_1) e_2 \\
           &=& e_5 + (t_1-\alpha_2) e_1+ (t_2 + \sigma^{-1}(\alpha_1))e_2 + t_3 e_3.
\end{eqnarray*}
Choose $\alpha_2 = t_1$ and $\alpha_1 = - \sigma(t_2)$ to arrange that $\rV e_4 = e_5 + s_4 e_3$ with $s_4 = t_3$.   Define $e_6 = \rF e_4$, so that $\pi(e_6) = y_3$.  From
$$
0 =\rF \rV e_4 = \rF (e_5 + s_4 e_3) = \rF e_5 + \sigma(s_4) \lambda e_2,
$$ 
we find that $\rF e_5 = -\sigma(s_4) \lambda e_2$.  We have
$$
\pi(\rV e_5) = \rV (\pi(e_5)) = \rV y_2 = \lambda' y_3 = \lambda' \pi(e_6), 
$$
so we can write  $\rV e_5 = \lambda' (e_6 + s_1  e_1 + s_2 e_2 + s_3 e_3)$.  Then 
$$
0 =\rF \rV e_5 = \sigma(\lambda') (\rF e_6 + \sigma(s_1) e_3 + \sigma(s_3) \lambda e_2) \quad \Rightarrow \quad \rF e_6 = -\sigma(s_1) e_3 - \sigma(s_3) \lambda e_2.
$$
Finally, $\rV e_6 = \rV \rF e_4 =0$, to complete the matrix for $\rV$.
\end{proof}

\begin{prop} \label{E2E2} 
Let $(\rM,\rL)$ represent an extension of the form {\rm \eqref{MExt} with $(i,j) = (2,2)$.  There is a $k$-basis $e_1, \dots, e_6$ for $\rM$} and a parameter vector $\bfs_{22} = (s_1,s_2,s_3)$ in $k^3$ such that $\iota(y_1) = e_1$, $\pi(e_4) = y_1$, $\rL = \lr{ e_1,e_2,e_4,e_5}$,   
\vspace{2 pt}
$$
\rV=\left[\begin{array}{ccc|ccc}
              0&0&0&0& \lambda' s_1 &0 \\
              1&0&0&0& \lambda' s_2 &0  \\
              0&\lambda' &0&0& \lambda' s_3&0  \\
                \hline
              0&0&0&0&0&0   \\ 
              0&0&0&1&0&0    \\ 
              0&0&0&0&\lambda'&0                                                                                
                                             \end{array}\right] \quad \text{and} \quad \,
\rF= \left[\begin{array}{ccc|ccc}
              0&0&0&0&0&0
           \\ 0&0&0&0& 0 & 0
           \\ 1&0&0&0&0& - \sigma(s_1)
           \\ \hline    0&0&0&0&0&0
           \\ 0&0&0&0& 0 & 0
           \\ 0&0&0&1&0&0  \end{array}\right]. 
$$
\end{prop}

\begin{proof}
For $j =1,2,3$, let $e_j = \iota(y_j)$.   Since $0 \to \rL_2  \xrightarrow{\iota} \rL  \xrightarrow{\pi} \rL_2 \to 0$ is exact, $e_1$ and $e_2$ are in $\rL$.  We can extend  to a basis for $\rL$ by adjoining $\tilde{e}_4$ and $\tilde{e}_5$ such that $\pi(\tilde{e}_4) = y_1$ and $\pi(\tilde{e}_5) = y_2$.  Since $\pi(\rV\tilde{e}_4) = \rV y_1 = y_2$, there are scalars $\alpha_1, \alpha_2,\alpha_3$ in $k$ such that $\rV \tilde{e}_4 = \tilde{e}_5 + \alpha_1 e_1 + \alpha_2 e_2 + \alpha_3 e_3$.

We may correct $\tilde{e}_4$ by an element of $\iota(\rL_2) = \lr{e_1,e_2}$.  First make the change $e_4 = \tilde{e}_4 + u_2 e_2$ with $u_2$ in $k$ to be determined.  Then 
$$
\rV e_4 = \rV \tilde{e}_4 + \lambda' \sigma^{-1}(u_2) e_3 
             =  \tilde{e}_5 + \alpha_1 e_1 + \alpha_2 e_2 + (\alpha_3 + \lambda \sigma^{-1}(u_2)) e_3.
$$
Since $\tilde{e}_5$ also can be corrected by an element of $\lr{e_1,e_2}$, let 
$
e_5 = \tilde{e}_5+ \alpha_1e_1+ \alpha_2 e_2
$
and choose $u_2 = -\sigma((\lambda')^{-1} \alpha_3)$.  Then $e_5$ is another lift of $y_2$ in $L$ and $\rV e_4 = e_5$.  Define $e_6 = \rF e_4$, so that $\pi(e_6) = \rF y_1 = y_3$.  Since $\pi (\rV e_5) = \rV y_2 = \lambda' y_3$, we have 
$$
\rV e_5 =  \lambda' (e_6 + s_1 e_1 + s_2 e_2 + s_3 e_3)
$$
for some $s_1, s_2, s_3$ in $k$.  From $\rF \rV e_5 = 0$, we find that $\rF e_6 = - \sigma(s_1) e_3$.  Then use $\rV e_6 = \rV \rF e_4 = 0$ and $\rF e_5 = \rF \rV e_4 = 0$ to complete the matrices for $\rV$ and $\rF$.

We could still take  $e_4' = e_4 + u_1 e_1$. Then  $e_5' = \rV e_4' = e_5 + \sigma^{-1}(u_1) \lambda' e_2$ and $e_6' = \rF e_4' = e_6 + \sigma(u_1) e_3$ and so 
\begin{eqnarray*}
\rV e_5' &=& \rV e_5 + \sigma^{-2}(u_1) \lambda' e_3 =  \lambda' (e_6 + s_1 e_1 + s_2 e_2 + (s_3 + \sigma^{-2}(u_1)) e_3) \\ 
             &=&  \lambda' (e_6' + s_1 e_1 + s_2 e_2 + (s_3 + \sigma^{-2}(u_1) - \sigma^{-1}(u_1)) e_3).
\end{eqnarray*}
Thus   $s_3' = s_3 + \sigma^{-2}(u_1) - \sigma^{-1}(u_1)$. 
\end{proof}

Fix the parameters $\lambda$ for $(\rM_1,\rL_1)$ and $\lambda'$ for $(\rM_2,\rL_2)$.  Each of the Propositions \ref{E1E1}, \ref{E2E1}, \ref{E1E2}, \ref{E2E2} treats a specific pair $(i,j)$, $1 \le i,j \le 2$ and assigns to a parameter vector $\bfs_{ij}$, an extension $(\rM,\rL)$ as in \eqref{MExt}.  

\begin{prop} \label{BS}
Given $(i,j)$, let the parameter vectors $\bfs_{ij}$, $\bfs_{ij}'$ and $\bfs_{ij}'' = \bfs_{ij} + \bfs_{ij}'$ determine the Honda systems $(\rM,\rL)$,  $(\rM',\rL')$ and $(\rM'',\rL'')$, respectively.  Then $(\rM'',\rL'')$ is the Baer sum of $(\rM,\rL)$ and $(\rM',\rL')$.  
\end{prop}

\begin{proof} 
The underlying argument applies to additive groups $M_1$ and $M_2$.  Consider two extensions of  $M_2$ by $M_1$, say:  
$$
0 \to M_1 \xrightarrow{\iota} M \xrightarrow{\pi} M_2 \to 0 \quad \text{ and } \quad
0 \to M_1 \xrightarrow{\iota'} M' \xrightarrow{\pi'} M_2 \to 0.
$$  
Their Baer sum $M''$ is the fiber product of $M$ and $M'$ modulo the subgroup
$$
\Delta:=\left\langle \, (\iota(x),0)-(0,\iota'(x)) \hspace{4pt} | \hspace{4pt} x \in M_1\right\rangle.
$$ 
Thus, 
$
M''= \{(m,m') \in M \times M' \hspace{4pt} | \hspace{4pt} \pi(m)=\pi'(m') \}/\Delta
$ 
is an extension 
$$
0 \to M_1 \xrightarrow{\iota''} M'' \xrightarrow{\pi''} M_2 \to 0,
$$ 
with $\iota''$ induced by $x \mapsto (\iota(x), 0)$ and $\pi''$ induced by $(m,m') \mapsto \pi(m)$.

Fix sections $s\!: \, M_2 \to M$ and $s'\!: \, M_2 \to M'$ describing the two extensions, i.e., $\pi \circ s = {\rm Id}_{M_2}$ and $\pi' \circ s' = {\rm Id}_{M_2}$.  Then $s'' = s \oplus s'$ is a section of $\pi''$.  In fact, for $y \in M_2$, the map $s''$ is induced by $y \mapsto (s(y),s'(y)) \in M \oplus M'$ and 
$$
\pi''(s''(y)) = \pi''\left( \, (s(y),s'(y)) + \Delta \right) = \pi(s(y)) = y \, \text{ for all } y \in M_2.
$$

Fix endomorphisms $f_1$ and $f_2$ on $M_1$ and $M_2$ and let $g$ be an endomorphism of $M$ such that $g \circ \iota = \iota f_1$ and $\pi \circ g =  f_2 \circ \pi$.  For $y$ in $M_2$, we find that $g(s(y)) - s(f_2(y))$ is in $\ker \pi$, so $y$ determines a unique element $x$ in $M_1$ such that $g(s(y)) - s(f_2(y)) = \iota(x)$.  Define the homomorphism $\delta_g: M_2 \to M_1$ by $\delta_g(y) = x$.

Similarly, let $g'$ be an endomorphism of $M'$ respecting $\iota'$ and $\pi'$ and let $\delta'_{g'}$ be the corresponding homomorphism in $\Hom(M_2,M_1)$.  One easily checks that $g'' =g \oplus g'$ is an endomorphism of $M''$ respecting $\iota''$ and $\pi''$ and that $\delta''_{g''} = \delta_g + \delta_g'$.

For Honda systems, we treat $i=j=1$, since the other cases are similar.  Let $\bfs_{11} = (s_1,s_2,s_3)$ and $\bfs_{11}' = (s_1',s_2',s_3')$ determine the Honda systems $(\rM,\rL)$ and $(\rM',\rL')$, respectively.  Let $\rF$ and $\rV$ denote the Frobenius and Verschiebung endomorphisms on $\rM$.   Then $\delta_{\rV}$ and $\delta_{\rF}$ dictate the upper right $3 \times 3$ blocks in the matrices for $\rV$ and $\rF$ in Proposition \ref{E1E1} and thereby determine $(s_1,s_2,s_3)$.   The upper $3\times3$ block for the Verschiebung of the Baer sum is the sum of the corresponding $3 \times 3$ blocks for $(\rM,\rL)$ and $(\rM',\rL')$ and similarly for the Frobenius.  It follows that $\bfs_{11}+\bfs_{11}'$ gives the Baer sum.
\end{proof}

\section{The local field of points} \label{LocalFieldofPoints}

Let $\cE_i$ be the simple group scheme associated to the Honda system $\gE_i =  (\rM_i,\rL_i)$ of Proposition \ref{LocalSimple} for $i = 1,2$.  We preserve the notation of \S \ref{Honda}, except that we now write $\lambda$ and $\lambda'$ for lifts to $\W$ of what previously were the parameters in $k$ of $\gE_1$ and $\gE_2$.  The choices of $\lambda$ and $\lambda'$ are fixed throughout and will be suppressed from the notation.  In this section, we study the field of points of $\cE_i$ and of extensions $\cV$ of $\cE_i$ by $\cE_j$ such that $p\cV = 0$.  

\numberwithin{equation}{subsection}

\subsection{The simple group schemes} 
Let $x_1, x_2, x_3$ and $y_1, y_2, y_3$ be standard bases for $\gE_1$ and $\gE_2$ respectively.  Then $x_1$ generates $\rM_1$ and $y_1$ generates $\rM_2$ as a $D_k$-module.  Write $E_i$ for the Galois module of $\cE_i$.  The points of $E_i$ are $D_k$-homomorphisms 
\begin{equation} \label{DkHom}
\psi\!: \, \rM_i \to  \widehat{CW}_k(\cO_{\ov{K}}/p\cO_{\ov{K}}) 
\end{equation}
such that the Hasse-Witt exponential $\xi(\psi(\rL_i)) = 0$, as in \eqref{HW}. 

More generally, whenever a congruence involving elements of $\ov{K}/w \cO_{\ov{K}}$ is stated, it is meant to be valid, independent of the choice of lifts to $\ov{K}$.

We first  show that as Galois modules, $E_1$ and 
$$
\gR_\lambda = \{ a \in \cO_{\ov{K}}/p \cO_{\ov{K}}  \hspace{2 pt} \vert \hspace{2 pt} a^{p^3} +\lambda^p p a \equiv 0 \! \pmod{p^2 \cO_{\ov{K}}} \}
$$
are isomorphic and we describe the field of points $F = K(E_1)$.  For $a$ in $\gR_\lambda$, define $b \in \cO_{\ov{K}}/p \cO_{\ov{K}}$ by  
$
 b \equiv \lambda^{-1} a^{p^2} \! \pmod{p\cO_{\ov{K}}}  
$
and denote by $\psi_a$ the unique $D_k$-homomorphism $\rM_1 \to  \widehat{CW}_k(\cO_{\ov{K}}/p\cO_{\ov{K}})$ satisfying $\psi_a(x_1) =  (\0,b,a)$.

\begin{prop}  \label{E1Field} 
The field of points $F = K(E_1)$ is the splitting field over $K$ of $f_\lambda(x) = x^{{p^3}-1} + \lambda^p p$.  Moreover:  
\begin{enumerate}[{\rm i)}]
\item A $D_k$-map $\psi$ is in $E_1$ if and only if $\psi = \psi_a$ for some $a$ in $\gR_\lambda$.  \vspace{2 pt}

\item The maximal subfield of $F$ unramified over $K$ is $F_0 = K(\Mu_{p^3-1})$ and $F/F_0$ is totally ramified of degree $p^3-1$.  \vspace{2 pt}

\item   $\gR_\lambda$ is an $\F_{p^3}$-vector space under the usual operations on $\cO_{\ov{K}}/p \cO_{\ov{K}}$ and $a \mapsto \psi_a$ defines an $\F_p[G_K]$-isomorphism $\gR_\lambda \xrightarrow{\sim} E_1$.     
\end{enumerate}
\end{prop}

\begin{proof}   
i) Let the $D_k$-homomorphism $\psi$ be a non-zero point in $E_1$. Since $V^2 = 0$, we have $\psi(x_1) = (\0,b,a)$ for some $a$ and $b$ in $\cO_{\ov{K}}/p \cO_{\ov{K}}$ with $a \ne 0$.   Acting by $\rV$ and $\rF$ gives 
$$
\psi(x_2)  = (\0,b) \quad \text{ and } \quad \psi(x_3) = (\0,b^p,a^p).
$$
  Then $0 = {\rV \rF}(\psi(x_1)) = \rV(\psi (x_3))$ forces $b^p \equiv 0 \pmod{p}$, so $\ord_p(b) \ge \frac{1}{p}$.  Also,
$$
\lambda (\0,b) = \psi(\lambda x_2) = \psi(\rF x_3) = \rF(\psi(x_3)) = (\0,a^{p^2}).
$$
Hence $\lambda b \equiv a^{p^2} \pmod{p}$ and so $\ord_p(a) = \frac{1}{p^2} \ord_p(b) \ge \frac{1}{p^3}$.   From $\xi(\psi(x_1)) = 0$, we find that $a + \frac{1}{p} b^p \equiv 0 \pmod{p}$.  Then Lemma \ref{ppower} gives
$$
\textstyle{a + \frac{1}{p}(\lambda^{-1} a^{p^2})^p \equiv 0 \pmod{p}}.
$$  
Thus $a^{p^3} +\lambda^p p a \equiv 0 \pmod{p^2}$, from which we obtain the following valuations:
$$
\textstyle{\ord_p(a) = \frac{1}{p^3-1} \quad \text{ and } \quad \ord_p(b) = \frac{p^2}{p^3-1}} \, .
$$ 
All points of $E_1$ are accounted for as $a$ ranges over $\gR_\lambda$
\vspace{2 pt}

ii) If $f_\lambda(\theta) = 0$ and $\zeta$ generates $\Mu_{p^3-1}$, then the roots of $f_\lambda$ have the form $\theta_j = \zeta^j \theta$, while their reductions modulo $p$ give all non-zero elements of $\gR_\lambda$.  For the converse, lift $a \in \gR_\lambda$ to $\tilde{a}$ in $\cO_{\ov{K}}$.  If $g(x) = x^{p^3}-x$, then $g(\tilde{a}/\theta) \equiv 0 \pmod{\frac{p}{\theta} \, \cO_{\ov{K}}}$, so $\tilde{a} \equiv 0$ or $\tilde{a} \equiv \theta_j \pmod{p \cO_{\ov{K}}}$ for some $j$ by Hensel's Lemma.    Thus $F = K(\Mu_{p^3-1},\theta)$ is the splitting field of $f_\lambda$.   Since $f_\lambda$ is an Eisenstein polynomial over $F_0 = K(\Mu_{p^3-1})$, the extension $F/F_0$ is totally ramified of degree $p^3-1$ and $F_0$ is the maximal subfield of $F$ unramified over $K$.   \vspace{2 pt}

iii)  Scalar multiplication by $\F_{p^3}$ on $\gR_\lambda$ is defined via the embedding 
$$
\F_{p^3} = \W(\F_{p^3})/p \hspace{.5 pt} \W(\F_{p^3}) \hookrightarrow \cO_{\ov{K}}/p \cO_{\ov{K}}.
$$  
Closure of $\gR_\lambda$ under this operation and under the usual addition in $ \cO_{\ov{K}}/p \cO_{\ov{K}}$  is clear.  If $a_1$ and $a_2$ are in $\gR_\lambda$, addition of points in $E_1$ implies that 
$$
\psi_{a_1}(x_1) \CWplus \psi_{a_2}(x_1) = \psi_{a}(x_1)
$$
for some $a$ in $\gR_\lambda$.  Denote this equation of Witt covectors by
$$
(\0,b_1,a_1) \CWplus (\0,b_2,a_2) = (\0,b,a).
$$
By applying the Verschiebung $\rV$, we find that  $b = b_1+b_2$, so $a^{p^2} = a_1^{p^2} + a_2^{p^2}$ in $\cO_{\ov{K}}/p\cO_{\ov{K}}$.  Without changing notation, lift $a$, $a_1$ and $a_2$ to $\cO_{\ov{K}}$ and let $a = \omega_0 \theta$, $a_1 = \omega_1 \theta$ and $a_2 = \omega_2 \theta$, where $\theta$ is a root of $f$ and each $\omega_j$ is in $\Mu_{p^3-1} \cup \{0\}$.  Then  
$$\omega_0^{p^2} \equiv \omega_1^{p^2} + \omega_2^{p^2} \equiv (\omega_1+\omega_2)^{p^2} \pmod{\frac{p}{\theta^{p^2}} \cO_{\ov{K}}}.
$$  
Since the $\omega$'s lie in the unramified extension $\Q_p(\Mu_{p^3-1})/\Q_p$ and $\ord_p(p/\theta^{p^2}) > 0$, we obtain $\omega_0 \equiv \omega_1 + \omega_2 \pmod{p}$ and thus $a = a_1+a_2$ in $\cO_{\ov{K}}/p\cO_{\ov{K}}$.  Alternatively, $\ord_p(a-a_1-a_2) \ge 1$ by the covector addition formulas in Lemma \ref{Wittadd}.
\end{proof}
 
Next we show that there is a Galois module isomorphism of $E_2$ with
\begin{equation} 
\gR'_{\lambda'} = \{ \alpha \in \cO_{\ov{K}}/p \cO_{\ov{K}}  \hspace{2 pt} \vert \hspace{2 pt} (-\lambda')^{p^2} \alpha^{p^3} +  p^{p+1} \alpha \equiv 0 \! \pmod{p^{p+2} \cO_{\ov{K}}} \}
\end{equation}
and we describe $F' = K(E_2)$.  For $\alpha$ in $\gR'_{\lambda'}$, define $\beta,\gamma \in \cO_{\ov{K}}/p \cO_{\ov{K}}$ by  
\begin{equation*}
\textstyle{\gamma \equiv \lambda' \alpha^p \! \pmod{p\cO_{\ov{K}}} \quad \text{and} \quad \beta \equiv - \frac{1}{p} \gamma^p \equiv -\frac{1}{p} (\lambda')^p \alpha^{p^2} \! \pmod{p\cO_{\ov{K}}}}.    
\end{equation*}
Denote by $\psi'_\alpha$ the unique $D_k$-homomorphism $\rM_2 \to  \widehat{CW}_k(\cO_{\ov{K}}/p\cO_{\ov{K}})$ satisfying $\psi'_\alpha(y_1) =  (\0,\gamma,\beta,\alpha)$.

\begin{prop}  \label{E2Field} 
The field of points $F' = K(E_2)$ is the splitting field over $K$ of $g_{\lambda'}(x) = (-\lambda')^{p^2} x^{{p^3}-1} + p^{p+1}$.  \begin{enumerate}[{\rm i)}]
\item A $D_k$-map $\psi$ is in $E_2$ if and only if $\psi = \psi'_\alpha$ for some $\alpha$ in $\gR'_{\lambda'}$.  \vspace{2 pt}

\item  Let $F'_0$ be the maximal subfield of $F'$ unramified over $K$.  If $p$ is odd, then $F'_0 = K(\Mu_{p^3-1},\sqrt{\lambda'})$ and $\fdeg{F'}{F'_0} = \frac{1}{2} \left( p^3-1 \right)$.  If $p = 2$, then $F'_0 = K(\Mu_7)$ and $\fdeg{F'}{F'_0} = 7$.   \vspace{2 pt}

\item   $\gR'_{\lambda'}$ is an $\F_{p^3}$-vector space under the usual operations on $\cO_{\ov{K}}/p \cO_{\ov{K}}$ and $\alpha \mapsto \psi'_\alpha$ defines an $\F_p[G_K]$-isomorphism $\gR'_{\lambda'} \xrightarrow{\sim} E_2$.
\end{enumerate}
\end{prop}

\begin{proof}   
i)  Suppose that the $D_k$-homomorphism $\psi$ is a non-zero point in $E_2$. Since $V^3 = 0$, we have $\psi(y_1) = (\0,\gamma,\beta,\alpha)$ for some $\alpha$, $\beta$ and $\gamma$ in $\cO_{\ov{K}}/p \cO_{\ov{K}}$.  By applying $\rV$ and $\rF$ we find that 
$$ 
\psi(y_2)  = (\0,\gamma,\beta) \quad \text{and} \quad \psi(y_3) = (\0,\gamma^p,\beta^p,\alpha^p).
$$  
Also, $0 = {\rV \rF}(\psi(y_1)) = \rV(\psi (y_3))$ gives $\gamma^p \equiv \beta^p \equiv 0 \pmod{p}$, so $\psi(y_3) = (\0,\alpha^p)$.   In addition,
$$
\lambda' (\0,\alpha^p) = \psi(\lambda' y_3) = \psi(\rV y_2) = \rV(\psi(y_2)) = (\0,\gamma) \quad \Rightarrow \quad \gamma \equiv \lambda' \alpha^p \pmod{p}.
$$
Vanishing of the Hasse-Witt exponential $\xi(\psi(y_2))$ gives $\beta + \frac{1}{p} \gamma^p \equiv 0 \pmod{p}$ and so $\beta \equiv -\frac{1}{p} (\lambda')^p \alpha^{p^2} \pmod{p}$ by Lemma \ref{ppower}.  In follows that $\ord_p(\gamma) = p \ord_p(\alpha)$ and $\ord_p(\beta) = p^2 \ord_p(\alpha)-1$.  Vanishing of $\xi(\psi(y_1))$ gives 
\begin{equation} \label{HWy1}
\textstyle{\alpha + \frac{1}{p} \beta^p + \frac{1}{p^2} \gamma^{p^2} \equiv 0  \pmod{p}}.  
\end{equation}
By comparing the terms, we obtain the following valuations:
$$
\textstyle{\ord_p(\alpha) = \frac{p+1}{p^3-1},  \quad  \ord_p(\beta) = \frac{p^2+1}{p^3-1}, \quad \ord_p(\gamma) = \frac{p^2+p}{p^3-1}.}  
$$  
Then $\frac{1}{p^2} \gamma^{p^2} \equiv 0  \pmod{p}$ in \eqref{HWy1}.  Substitute for $\beta$  in that equation to find that
$
\alpha + \frac{1}{p^{p+1}} (-\lambda')^{p^2}  \alpha^{p^3} \equiv 0 \pmod{p}   
$
and so  $g_{\lambda'}(\alpha) \equiv 0 \pmod{p^{p+2}}$. 

Now, let $\alpha$ be a root of $g_{\lambda'}$.   If $p = 2$, then $2\alpha^{-2}$ is a root of the Eisenstein polynomial $y^7 - 2 (\lambda')^8$ and $F' = K(\sqrt[7]{2\lambda'}, \Mu_7)$.  If $p$ is odd, we have
$$
(\lambda')^{p^2} =\left( p^{(p+1)/2} \, \alpha^{-\frac{1}{2}(p^3-1)} \right)^2
$$
and so $\lambda'$ is a square in $F'$.  Choose the sign of $\eta = \sqrt{\lambda'}$ so that $\alpha$ a root of $\eta^{p^2}y^{\frac{1}{2}(p^3-1)} - p^{\frac{1}{2}(p+1)}$.  Since $\gcd\{\frac{1}{2}(p^3-1), \frac{1}{2}(p+1)\} = 1$, the field $F'$ is totally ramified of degree $\frac{1}{2}(p^3-1)$ over $F_0 = K(\Mu_{p^3-1},\sqrt{\lambda'})$.

The proof of (iii) parallels that of Proposition \ref{E1Field}
\end{proof}

\begin{Rem} \label{Pa}  
According to Proposition \ref{E1Field}(ii), if $\theta$ is a root of $f_\lambda(x)$, then the lifts of all $a \ne 0$ in $\gR_\lambda$ to $\cO_{\ov{K}}$ comprise the cosets $\zeta^j \theta + p\cO_{\ov{K}}$ as $\zeta^j$ ranges over $\Mu_{p^3-1}$.  Thus $\gR_\lambda$ descends to an $\F_{p^3}$-linear subspace of $\cO_F/p \cO_F$ and we write 
$$
\gR_\lambda(F) = \{ a \in \cO_F/p \cO_F \hspace{3 pt} \vert \hspace{3 pt} a^{p^3} +\lambda^p p a \equiv 0 \text{ (mod }p^2 \cO_F) \}.
$$
Let the point $\psi_a$ in $E_1$ correspond to $a$ in $\gR_\lambda(F)$.  We define an action of elements $\epsilon$ of $\F_{p^3}$ on $E_1$ by $\epsilon \psi_a = \psi_{\epsilon a}$, in agreement with multiplication on Witt covectors.  In fact, if $b_a = \lambda^{-1} a^{p^2}$, then evaluating on $x_1$ gives
$$
[\epsilon] (\0,b_a, a) = (\0,  \epsilon^{\frac{1}{p}} b_a, \epsilon a)= (\0,b_{\epsilon a}, \epsilon a).
$$
Similarly, $\gR'_{\lambda'}(F') = \{ \alpha \in \cO_{F'}/p \cO_{F'}  \hspace{2 pt} \vert \hspace{2 pt} (-\lambda')^{p^2} \alpha^{p^3} +  p^{p+1} \alpha \equiv 0 \! \pmod{p^{p+2} \cO_{F'}} \}$.  If the point $\psi'_\alpha$ in $E_2$ corresponds to $\alpha$ in $\gR'_{\lambda'}(F')$, let $\epsilon \psi'_{\alpha} = \psi'_{\epsilon \alpha}$.   
\end{Rem}

For the convenience of the reader, we summarize some congruences and valuations from Propositions \ref{E1Field} and \ref{E2Field} in the following Table.

\begin{table}[h] 
\begin{center}
{\begin{caption}{Convenient formulae for the Honda systems $\rM_1$ and $\rM_2$} \label{Bank} \end{caption} }
\end{center}
\renewcommand{\arraystretch}{1.4}
\begin{tabular}{ | l |}
\hline
For $\rM_1$: \,  $a^{p^3} + \lambda^p pa \equiv 0 \pmod{p^2}$ \\ 
\hspace*{3 pt} $\bullet$ mod $p$:  \hspace{4 pt}  $b \equiv \lambda^{-1} a^{p^2}$, \hspace{10 pt} $\frac{1}{p}b^p \equiv -a$   \hspace{20 pt}  $\bullet$ mod $p^p$:  \hspace{4 pt} $\frac{1}{p^2}b^{p^2} \equiv (-1)^p p^{p-2} a^p$ \\ 
\hspace*{3 pt}  $\bullet$ $\ord_p(a) = \frac{1}{p^3-1}$, \quad$\ord_p(b) = \frac{p^2}{p^3-1}$ \\ [3 pt]
\hline  
For $\rM_2$:  \, $(-\lambda')^{p^2} \alpha^{p^3} + p^{p+1} \alpha \equiv 0 \pmod{p^{p+2}}$\\ 
\hspace*{3pt} $\bullet$ mod $p$:  \hspace{4 pt} $\beta \equiv -\frac{1}{p} (\lambda')^p \alpha^{p^2}$, \hspace{18 pt}  $\frac{1}{p}\beta^p \equiv -\alpha$, \hspace{18 pt}  $\gamma \equiv \lambda' \alpha^p$, \hspace{18 pt} $\frac{1}{p} \gamma^p \equiv -\beta$ \\  
\hspace*{3pt} $\bullet$ mod $p^p$:  \hspace{10 pt} $\frac{1}{p^2}\beta^{p^2} \equiv (-1)^p p^{p-2} \alpha^p$, \hspace{25 pt}  $\frac{1}{p^2} \gamma^{p^2} \equiv (-1)^{p+1} p^{p-1} \alpha$ \\ 
\hspace*{3 pt}  $\bullet$ $\ord_p(\alpha) = \frac{p+1}{p^3-1},  \quad  \ord_p(\beta) = \frac{p^2+1}{p^3-1}, \quad \ord_p(\gamma) = \frac{p^2+p}{p^3-1}$   \\ [3 pt]
\hline
\end{tabular}
\end{table}
\renewcommand{\arraystretch}{1}

\begin{Rem} \label{AlphaA}
We have
$
(\alpha/a)^{p^3-1} = \left( (-\lambda')^{-p} \lambda^{-1} p \right)^p, 
$
from the defining equations, so $\alpha/a = w^p$ is a $p$-th power for some $w$ in $F'' = FF' = K(E_1,E_2)$.
\end{Rem}

\subsection{Extensions of exponent $p$} \label{ptsEEp}

For $j$ in $\{1,2\}$, let $\cE_j$ be the group scheme associated to the Honda system $(\rM_j, \rL_j)$ of Proposition \ref{LocalSimple}.  For a given pair $(i,j)$, let  
$$
\Ext^1_{[p]}(\cE_i,\cE_j) = \{ \, [\cV] \in \Ext^1(\cE_i,\cE_j) \, \vert \, p \cV = 0 \}
$$
be the subgroup of $\Ext^1(\cE_i,\cE_j)$ comprising the classes represented by group schemes $\cV$ such that $p \cV = 0$.  Thus, there is an exact sequence of Honda systems 
\begin{equation} \label{ExtSeq}
0 \to (\rM_i, \rL_i) \xrightarrow{\iota} (\rM,\rL) \xrightarrow{\pi}  (\rM_j,\rL_j) \to 0
\end{equation}
 corresponding to the exact sequence of group schemes $0 \to \cE_j \xrightarrow{\pi^*} \cV \xrightarrow{\iota^*} \cE_i \to 0$.

Let $\bfs_{ij}$ be a parameter vector for $(\rM,\rL)$ and preserve the notation in \S \ref{Honda} for ordered bases:   $x_1, x_2, x_3$ of $\rM_1$,  $y_1,y_2,y_3$ of $\rM_2$ and $e_1, \dots, e_6$ of $\rM$. The points of $\cV$ are homomorphisms
$$
\varphi \in \Hom_{D_k}({\rM},\widehat{CW}_k(\cO_{\ov{K}}/p\cO_{\ov{K}}))
$$
determined by $\varphi(e_1)$ and $\varphi(e_4)$, since $e_1$ and $e_4$ generates $\rM$ as a $D_k$-module.

\begin{theo}  \label{CondBd}
Let $\cV$ represent a class in $\Ext^1_{[p]}(\cE_i,\cE_j)$.  Then $K(V)$ is an elementary abelian $p$-extension of $K(E_i,E_j)$ with ray class conductor exponent $\gf \le p^2+p$.  
\end{theo}

The proof is done in the next four propositions, depending on the pair $(i,j)$.  In some cases, we determine conductor bounds for certain parameters and obtain the others from Baer sums, using Proposition \ref{BS} and the following general Lemma.

\begin{lem} \label{BC}
Fix an algebraic closure $\ov{F}$ of a local field $F$.  Inside $\ov{F}$, let $L_i/F$ be an abelian extension with ray class conductor exponent $\gf_i = \gf(L_i/F)$ for $i=1,2,3$.  If $\gf_1 < \gf_2$ and the inclusions $L_2 \subseteq L_1L_3$ and $L_3 \subseteq L_1L_2$ hold, then $\gf_2 = \gf_3$.
\end{lem}

\begin{proof}
The inequalities $\gf_3 \le \gf(L_1L_2/F) = \gf_2$ and $\gf_2 \le \gf(L_1L_3/F) \le \max\{\gf_1,\gf_3\}$ follow from the given inclusions and imply that $f_2 = f_3$.
\end{proof}

By construction, $K(V)/K(E_i,E_j)$ is an elementary abelian $p$-extension.  See Propositions \ref{E1Field} and \ref{E2Field} for properties of the fields $F = K(E_1)$ and $F' = K(E_2)$.  Let $F'' = K(E_1,E_2)$ and $L = K(V)$.   Points of $E_1$ have the form $\psi_a$ with $a$ in $\gR_\lambda(F)$ and points of $E_2$ have the form $\psi'_\alpha$, with $\alpha$ in $\gR'_\lambda(F')$, as in Remark \ref{Pa}.  See Table \ref{Bank} for a summary of the relations on $a$, $\alpha$ and the other parameters of $\cE_1$ and $\cE_2$.

If $i=1$, let $L_a$ be the field generated by the points $\varphi$ of $V$ in the fiber over the point $\psi_a \ne 0$ of $E_1$, treating $L_a$ as an extension of $F$ if $j=1$, or $F''$ if $j=2$.  Similarly, if $i = 2$, let $L_\alpha$ be the field generated by the points in the fiber of $V$ over the point $\psi'_\alpha \ne 0$ of $E_2$.   By controlling the conductor for each of these fibers, we obtain the bound on $\gf(L/F)$,

To determine the field of points, it is convenient to make a change of variables using congruences on elements of $\ov{K}$ that are not necessarily integral.  Thus, we consider $\cO_{\ov{K}}$-modules of the form $\rR_d = \ov{K}/\frac{p}{d} \cO_{\ov{K}}$, with $d$ in $\cO_{\ov{K}}$ and $\ord_p(d) < 1$.  Let $r = x + \frac{p}{d} \cO_{\ov{K}}$ be an element of  $\rR_d$ with $x$ in $\ov{K}$ and also write $\ov{x} = r$ for the image of $x$ in $\rR_d$.  Let $q = p^n$ with $n \ge 1$.  For the reader's convenience, we summarize some aspects of the arithmetic of $\rR_d$ needed here and easily checked.

\vspace{2 pt}

\setlength{\hangindent}{25 pt}
\noindent \hspace{3 pt} {\bf C1} If $r$ is 0 in $\rR_d$, define $\ord_p(r) = \infty$.  Otherwise, define $\ord_p(r) = \ord_p(x)$, which is independent of the choice of representative $x$.   

\vspace{2 pt}

\noindent \hspace{3 pt} {\bf C2} Assume that $px^{q-1}$ is in $\cO_K$ and define $r^q = x^q + \frac{p}{d} \cO_{\ov{K}}$, also independent of the choice of representative $x$.   Thus, by definition $\ov{x}^{\, q} = \ov{x^q}$. 

\vspace{2 pt}

\noindent \hspace{3 pt} {\bf C3} If $dx^q$ and $dy^q$ are in $\cO_{\ov{K}}$,  then $\ov{x}^q$ and $\ov{y}^q$ are well-defined in $\rR_d$ and we have $(\ov{x}+\ov{y})^q = \ov{x}^{\,q} + \ov{y}^{\,q}$ in $\rR_d$. 

\vspace{2 pt}

\noindent \hspace{3 pt} {\bf C4} If $\ord_p(d) \le \ord_p(d')$, then the natural projection $\rR_{d} \to \rR_{d'}$ is well-defined.  

\begin{prop} \label{E1E1Field}
If $\cV$ represents a class in $\Ext^1_{[p]}(\cE_1,\cE_1)$ and $L = K(V)$, then $\gf(L/F) \le p^2$.
\end{prop}

\begin{proof}
Here $L_a$ is the extension of $F$ generated by the points $\varphi$ in the fiber of $V$ over $\psi_a$, as described above.  In the proof, we provide more details about the conductor of $L_a$ in terms of the parameter vector $\bfs_{11} = (s_1,s_2,s_3)$ of the associated Honda system $\rM$.   Since $\cV$ is a non-split extension, some $s_j \ne 0$. 

By Proposition \ref{E1Field}, $\varphi(e_1) = (\0,b,a)$, $\varphi(e_2) = (\0,b)$ and $\varphi(e_3) = (\0,a^p)$.  Set $\varphi(e_4)=(\0,d_2,d_1,d_0)$, since $\rV^3 = 0$ and obtain
$
\varphi(e_5) = \varphi(\rV e_4) = (\0,d_2,d_1).
$  
Then $d_2^p \equiv d_1^p \equiv 0 \pmod{p}$ because $\varphi(\rF \rV e_4) = 0$, so $\varphi(e_6) = \varphi(\rF e_4) = (\0,d_0^p)$.     We have
$$
-\sigma^{-1}(s_1)(\0,b) = \varphi(-\sigma^{-1}(s_1) e_2) = \varphi(\rV e_5) = (\0, d_2), 
$$
so $d_2 \equiv -\sigma^{-1}(s_1) b \pmod{p}$ and by Lemma \ref{ppower}, $\frac{1}{p^2} d_2^{p^2} \equiv p^{p-2} s_1^p a^p \pmod{p}$. 

By Proposition \ref{E1E1}, 
$
\rF e_6 = \lambda (e_5 + s_1 e_1 + s_2e_2 + s_3 e_3).  
$
Applying $\varphi$ gives
$$
(\0,d_0^{p^2}) = \lambda ((\0,d_2,d_1)  \CWplus  s_1(\0,b,a)  \CWplus  s_2(\0,b) \CWplus s_3(\0,a^p)).
$$
Combine the Witt covectors above via \eqref{Wittadd}. Then the rightmost coordinate gives
$$
d_0^{p^2} \equiv \lambda (d_1 + s_1 a + s_2 b + s_3 a^p) \pmod{p}.
$$ 

The change of variables $d_0 = ax$, $d_1 = by$, $d_2 =cz$ yields 
\begin{equation}  \label{E1E1y}
y = x^{p^2} - A \text{ in } \rR_b, \hspace{3 pt} \text{ with } \, A = s_2 + \lambda (s_1a+s_3a^p)\,{a^{-p^2}}. 
\end{equation} 

From the vanishing of the Hasse-Witt exponential on $\varphi(\rL)$, we have
$$
0 = \xi(\varphi(e_4)) = \frac{1}{p^2} d_2^{p^2} + \frac{1}{p} d_1^p + d_0 \equiv p^{p-2} s_1^p a^p  + \frac{1}{p} d_1^p + d_0   \pmod{p}
$$
and so 
\begin{equation}  \label{E1E1x}
x = y^p - B \text{ in } \rR_a, \quad \text{ where } B = p^{p-2} s_1^p a^{p-1}.
\end{equation}  
Then $B$ is in the maximal ideal $\gm_F$ and we can take $B = 0$ if $p \ge 3$ or if $s_1 = 0$ in $k$.   By \eqref{E1E1x}, $\ord_p(x) = \ord_p(y^p)$, leading to $\ord_p(x^{p^2}) = \ord_p(A)$ by  \eqref{E1E1y}.   

Since $\ord_p(a) < \ord_p(b)$, we also have  $x = y^p - B$ in $\rR_b$ by {\bf C4}.  Then {\bf C3} implies that $x^{p^2} = (y^{p}-B)^{p^2} = y^{p^3}$ in $\rR_b$, since
$$
\ord_p(by^{p^3}) = \ord_p(bA) = \ord_p(a^{p^2} A) > 0 \quad \text{and} \quad B^{p^2} \equiv 0 \pmod{p}.
$$ 
Substituting in \eqref{E1E1y} gives $y^{p^3}-y-A = 0$ in $\rR_b$.  Conversely, if $y$ is a root of $f_a(Y) = Y^{p^3} - Y -A$ over $F$ and $x$ is defined by \eqref{E1E1x}, we find that the congruence \eqref{E1E1y} holds.  Hence the field $L_a$ generated by the points of the fiber in $V$ over $\psi_a$ is the splitting field of $f_a$.

If $s_1 = s_3 = 0$, then the reduction of $f_a$ is separable over the residue field $k_F$, so $L_a/F$ is unramified.  In the remaining cases, Proposition \ref{AS} implies that $L_a/F$ is an elementary abelian extension of degree $p^3$ and gives the conductor exponent.  More precisely, if $s_1 \ne 0$, then $\gf(L_a/F) = p^2$.  If $s_1 = s_2 = 0$ and $s_3 \ne 0$, then $\gf(L_a/F) = p$.  Finally, the conductor exponent also is $p$ for parameters of the form $\bfs_{11} = (0,s_2,s_3)$ by using Proposition \ref{BS} and Lemma \ref{BC} to combine the result for $\bfs_{11}= (\0,0,s_3)$ with the unramified case $\bfs_{11} = (0,s_2,0)$. 
\end{proof}

\begin{prop} \label{E2E1Field}
If $\cV$ represents a class in $\Ext^1_{[p]}(\cE_1,\cE_2)$ and $L = K(V)$, then $\gf(L/F'') = p^2-p$.
\end{prop}

\begin{proof}
Let $s$ be the parameter for the corresponding Honda system $\rM$, with $s \ne 0$ in $k$ since $\cV$ is non-split.  Here $L_\alpha$ is the extension of $F''$ generated by the points $\varphi$ in the fiber over $\psi'_\alpha$.   By Proposiiton \ref{E2Field}, $\varphi(e_1) = (\0,\gamma,\beta,\alpha)$, $\varphi(e_2) = (\0,\gamma,\beta)$ and $\varphi(e_3) = (\0,\alpha^p)$.  Since $\rV^4 = 0$, we can set 
$$
\varphi(e_4)=(\0,d_3,d_2,d_1,d_0) \quad \text{and so} \quad \varphi(e_5) = \varphi(\rV e_4) = (\0,d_3,d_2,d_1).
$$
Then $\varphi(\rF \rV e_4) = 0$ implies that $d_3^p \equiv d_2^p \equiv d_1^p \equiv 0 \pmod{p}$
and so $\varphi(e_6) = \varphi(\rF e_4) = (\0,d_0^p)$.  Since $\rV e_5 = s e_2$, we have
$
(\0,d_3,d_2) = s (\0,\gamma,\beta).
$
Thus
$$
d_3 \equiv s^{\frac{1}{p}} \gamma \pmod{p} \quad \text{and} \quad d_2 = s \beta \pmod{p}.
$$
By applying $\varphi$ to  
$
\rF e_6  + \lambda \sigma(s)e_1 = \lambda e_5,
$
we find that
$$
(\0,d_0^{p^2}) \CWplus  \lambda \sigma(s) (\0,\gamma,\beta,\alpha) = \lambda  (\0,d_3,d_2,d_1).
$$
By the Witt covector addition in Lemma \ref{Wittadd}, we recover the above congruences on $d_3$ and $d_2$ and we find that 
$
d_0^{p^2} + \lambda \sigma(s) \alpha \equiv  \lambda d_1 \pmod{p}.
$

For $a$ and $b$ as in Table \ref{Bank}, the change of variables $d_0 = ax$, $d_1 = by$ gives 
\begin{equation} \label{M2M1y}
y = x^{p^2} + A \text{ in } \rR_b, \quad \text{ where } A = \sigma(s) \alpha/b \, \text{ in } F''.  
\end{equation}
Since $e_4$ is in $\rL$, the Hasse-Witt exponential is 0 on $\varphi(e_4)$, so 
\begin{equation} \label{HWE1E2}
\frac{1}{p^3} d_3^{p^3} + \frac{1}{p^2} d_2^{p^2}+ \frac{1}{p}d_1^p + d_0 \equiv 0 \pmod{p}.
\end{equation}
The leftmost term above vanishes because $\ord_p(d_3) \ge \ord_p(\gamma) = \frac{p^2+p}{p^3-1}$.  In addition, $\frac{1}{p} \beta^p \equiv - \alpha \pmod{p}$ and so $\frac{1}{p^2} \beta^{p^2} \equiv (-1)^{p} p^{p-2} \alpha^p \pmod{p^{p-1}}$ by Lemma \ref{ppower}(ii).   Hence \eqref{HWE1E2} yields
\begin{equation} \label{M2M1x}
x= y^p - B  \text{ in } \rR_a, \quad \text{ where }  B = (-s)^{p^2} p^{p-2} \frac{\alpha^p}{a}.
\end{equation}
By assumption, $s \ne 0$ in $k$, so
$$
\ord_p(A) = -\frac{p^2-p-1}{p^3-1} \hspace{4 pt} \text{ and}  \hspace{4 pt} \ord_p(B) = p-2+\frac{p^2+p-1}{p^3-1}.
$$   
We can take $B = 0$ if $p \ge 3$.  In general, since $\ord_p(B) > 0$ and $\ord_p(A) < 0$, it follows from \eqref{M2M1x} and \eqref{M2M1y} that  
$$
\ord_p(y^p) = \ord_p(x) = \frac{1}{p^2} \ord_p(A).
$$ 
Since $\ord_p(a) < \ord_p(b)$, we also have  $x = y^p - B$ in $\rR_b$ by {\bf C4}.  Then {\bf C3} implies that $x^{p^2} = (y^{p}-B)^{p^2} = y^{p^3}$ in $\rR_b$, since
$$
\ord_p(by^{p^3}) = \ord_p(bA) = \ord_p(\alpha) > 0 \quad \text{and} \quad B^{p^2} \equiv 0 \pmod{p}.
$$ 
Putting this into \eqref{M2M1y} gives $y^{p^3}-y+A = 0$ in $\rR_b$.  Conversely, if $y$ is a root of $f_a(Y) = Y^{p^3} - Y + A$ over $F''$ and $x$ is defined by \eqref{M2M1x}, we find that the congruence \eqref{M2M1y} holds.  Hence the field $L_a$ generated by the points of the fiber in $V$ over $\psi'_\alpha$ is the splitting field of $f_a$.  

By Proposition \ref{AS} with $m = 0$ and $w = A^{-1}$, the extension $L_\alpha/F''$ is elementary abelian and totally ramified of degre $p^3$.  The conductor exponent of every intermediate field is $\gf = p^2-p$.  
\end{proof}

\begin{prop} \label{E1E2Field}
If $\cV$ represents a class in $\Ext^1_{[p]}(\cE_2,\cE_1)$ and $L = K(V)$, then $\gf(L/F'') \le p^2+p$.
\end{prop}

\begin{proof}
Let $\bfs_{12} = (;s_1,s_2,s_3,s_4)$ be the parameters for the corresponding Honda system $\rM$.  Here $L_a$ is the extension of $F''$ generated by the points $\varphi$ in the fiber over $\psi_a \ne 0$.  By Proposition \ref{E1Field} with $b \equiv \lambda^{-1}a^{p^2} \pmod{p}$, we have
$$
\varphi(e_1) = (\0,b,a), \quad \varphi(e_2) = (\0,b), \quad \varphi(e_3) = (\0,a^p).
$$ 
Set $\varphi(e_4)=(\0,d_3,d_2,d_1,d_0)$, since $\rV^4 = 0$.   Applying $\rV$ gives
$$
 (\0,d_3,d_2,d_1) = \rV(\varphi(e_4)) =  \varphi(\rV e_4) = \varphi(e_5) \CWplus s_4 \varphi(e_3) = \varphi(e_5) \CWplus (\0,s_4 a^p).
$$
Hence $\varphi(e_5) =  (\0,d_3,d_2,d_1-s_4 a^p)$, using Lemma \ref{Wittadd}.  From $0 = \varphi(\rF \rV e_4)$ we find that $d_3^p \equiv d_2^p \equiv d_1^p \equiv 0 \pmod{p}$, so $\varphi(e_6) = \varphi(\rF e_4) = (\0,d_0^p)$.   
By Proposition \ref{E1E2}, 
$
\rF e_6  + \sigma(s_1)e_3 + \sigma(s_3) \lambda e_2 = 0.  
$
Applying $\varphi$ gives
$$
(\0,d_0^{p^2}) \CWplus  \sigma(s_1) (\0,a^p) \CWplus \lambda  \sigma(s_3) (\0,b) = 0.
$$
Since $\lambda b \equiv a^{p^2} \pmod{p}$, we find that 
\begin{equation} \label{redundant}
d_0^{p^2} + \sigma(s_1) a^p + \sigma(s_3) a^{p^2} \equiv 0 \pmod{p}.
\end{equation}
Also, we have  
\begin{eqnarray*}
(\0,d_3,d_2) = \varphi(\rV e_5) &=& \lambda' \varphi(e_6+ s_1e_1+s_2e_2+s_3e_3) \\
 &=& \lambda' \left((\0,d_0^p) \CWplus s_1(\0,b,a) \CWplus s_2(\0,b)  \CWplus s_3(\0,a^p) \right) \\
&=&  (\0,(\lambda's_1)^{\frac{1}{p}}b, \lambda'(d_0^p +  s_1a + s_2 b + s_3 a^p)). 
\end{eqnarray*}
If we set $A = s_1a + s_2 b + s_3 a^p$, then 
\begin{equation} \label{12d3d2}
d_3 \equiv (\lambda's_1)^{\frac{1}{p}}b \pmod{p} \quad \text{and} \quad d_2 \equiv \lambda' (d_0^p + A) \pmod{p}.
\end{equation}

The last formula for $d_2$ already implies \eqref{redundant}.  Indeed, since $d_2^p \equiv 0 \pmod{p}$ and $b^p \equiv 0 \pmod{p}$, we find that 
$$
0 \equiv d_0^{p^2} + A^p \equiv d_0^{p^2} + (s_1a)^p + (s_2b)^p + (s_3a^p)^p \equiv d_0^{p^2} + (s_1a)^p + s_3^p a^{p^2} \! \pmod{p}.
$$

Vanishing of the Hasse-Witt exponential on $\varphi(\rL)$ gives: 
\begin{eqnarray} 
\label{xi5} \frac{1}{p^2} d_3^{p^2} + \frac{1}{p} d_2^p + d_1-ta^p = \xi(\varphi(e_5)) &\equiv& 0 \, (\bmod \, p) \quad \text{and} \\
\label{xi4} \frac{1}{p^3} d_3^{p^3} + \frac{1}{p^2} d_2^{p^2 }+ \frac{1}{p}d_1^p + d_0 =  \xi(\varphi(e_4)) &\equiv& 0 \, (\bmod \, p) .
\end{eqnarray}
Thanks to Lemma \ref{ppower}(ii) and Table \ref{Bank}, equation \eqref{12d3d2} implies that
$$
\frac{1}{p} d_3^{p} \equiv \lambda' s_1 \frac{1}{p} b^p \equiv -\lambda' s_1a  \, (\bmod \, p) \quad \text{and so} \quad  \frac{1}{p^{p+1}}{d_3^{p^2}} \equiv (-\lambda' s_1)^p \frac{1}{p} a^p \, (\bmod \, p).
$$
Hence, we can rewrite \eqref{xi5} as 
\begin{equation} \label{eqd1}
d_1 \equiv  - \frac{1}{p} d_2^p + ra^p \, (\bmod \, p), \quad \text{with } r = s_4 - (-\lambda's_1)^p p^{p-2}.
\end{equation}
Since vanishing of $d_1^p \pmod{p}$ implies that $\ord_p(d_1) \ge \frac{1}{p}$, we have
$$
\ord_p(\textstyle{\frac{1}{p}} d_2^p) \ge \min\{\ord_p(d_1),\ord_p(a^p)\} = \ord_p(a^p),
$$
with equality when $r \ne 0$.  Hence $\ord_p(d_2) \ge \frac{1}{p}+ \ord_p(a)$ and so
\begin{equation} \label{12vz}
\ord_p(\frac{1}{p^{2p}} d_2^{p^3}) \ge p^2-2 p+ p^3 \ord_p(a) = (p-1)^2 + \frac{1}{p^3-1} > 1.
\end{equation}

The congruence \eqref{12d3d2} implies that $\ord_p(d_3) \ge \ord_p(b) = p^2/(p^3-1)$, so the term $\frac{1}{p^3} d_3^{p^3} \equiv 0 \, (\bmod \, p)$ drops out of \eqref{xi4}, to give 
\begin{equation} \label{12d0eq}
\frac{1}{p^2} d_2^{p^2 }+ \frac{1}{p}d_1^p + d_0 \equiv 0 \pmod{p}.
\end{equation}
Then $d_0 \equiv -\frac{1}{p}d_1^p - \frac{1}{p^2} d_2^{p^2} \pmod{p}$.  Each term on the right of the last congruence is integral, so
\begin{equation} \label{d0p}
d_0^{p} \equiv (-1)^p \left( \frac{1}{p^p} d_1^{p^2} + \frac{1}{p^{2p}} d_2^{p^3} \right) \equiv  (-1)^p  \frac{1}{p^p} d_1^{p^2} \pmod{p}.
\end{equation}
thanks to \eqref{12vz}.

Fix $\alpha$, $\beta$, $\gamma$ as in Table \ref{Bank} and change variables via $d_2 = \gamma z$, $d_1 = \beta y$, $d_0 = \alpha x$ in \eqref{12d3d2}, \eqref{eqd1} and \eqref{12d0eq}, to obtain
\begin{equation}  \label{12zyx}
\begin{array}{r l l l l l}
{\rm (i)} & \hspace{5 pt}  &z = x^p + A \alpha^{-p}  &\text{ in } \rR_\gamma   \vspace{4 pt} \\
{\rm (ii)} & &y = z^p + r a^p \beta^{-1} &\text{ in } \rR_\beta, \vspace{4 pt}  \\
{\rm (iii)} & &x = y^p + (-1)^pp^{p-1} z^{p^2}  &\text{ in } \rR_\alpha.
\end{array}
\end{equation}
These congruences are necessary and sufficient for the construction of the map $\varphi$.

Making the same change of variables in \eqref{d0p} yields $x^p = y^{p^2}$ in $\rR_\gamma$.   Since $ \ord_p(\beta) < \ord_p(\gamma)$, equation (\ref{12zyx}(ii)) also holds in $\rR_\gamma$.  Using $d_1 = \beta y$ and $\ord_p(d_1) \ge \frac{1}{p}$ one can check that $py^{p^2-1}$ is integral and so $y^{p^2} = (z^p + r a^p \beta^{-1})^{p^2}$ is well-defined in $\rR_\gamma$ by {\bf C2}.  Combining this information with (\ref{12zyx}(i)) gives 
\begin{equation} \label{12Fieldz}
\left(z^p + ra^p \beta^{-1} \right)^{p^2} - z + A \alpha^{-p} = 0 \quad \text{in } \rR_\gamma. 
\end{equation}

To treat the case $A = 0$, an equation in terms of $y$ is preferred.  For that, we use $x^p = y^{p^2}$ in $\rR_\gamma$, already derived above, to obtain $z = y^{p^2} +A \alpha^{-p}$ in $\rR_\gamma$ from (\ref{12zyx}(i)).  One can check that $pz^{p-1}$ is integral, so $z^p = (y^{p^2} + A \alpha^{-p})^p$ is well-defined in $\rR_\gamma$ by {\bf C2}.  Then (\ref{12zyx}(ii)) implies that
\begin{equation} \label{12Fieldy}
(y^{p^2} + A \alpha^{-p})^p - y +   ra^p \beta^{-1} = 0 \quad \text{in } \rR_\gamma. 
\end{equation}

Conversely, given a root $z$ of \eqref{12Fieldz} in $\ov{K}$, define $y$ and $x$ by (\ref{12zyx}(ii)) and (iii) respectively, treating each as an equation in $\ov{K}$.  Then the congruence (\ref{12zyx}(i)) holds in $\rR_\alpha$.  Similar considerations apply if we begin with a root $y$ of \eqref{12Fieldy}.

Next, we find the conductor exponent $\gf(L_a/F'')$ in various cases where Proposition \ref{AS} applies.  It follows that in these cases, $L_a/F''$ is a totally ramified elementary $p$-extension of degree $p^3$.  

See that Proposition for notation used here.  In particular, $\ord_p(C) > p^3/(1-p^3)$ holds, as required.  Recall that $e_{F''} = p^3-1$, $v_{F''}(a) = 1$ and $v_{F''}(\alpha) = p+1$.

\vspace{2 pt}

\noindent {\bf Case 1}.  $r = 0$ and $z^{p^3}-z + C = 0$ over $F''$, with $C = A \alpha^{-p}$ by \eqref{12Fieldz}.
\begin{enumerate}[\hspace{2 pt} $\bullet$]
\item $s_1 \ne 0$.  Then $m = 0$, $w = C^{-1} $, $v_{F''}(w) = v_{F''}(\alpha^p) - v_{F''}(a) = p^2+p-1$ and so $\gf(L_a/F'') = p^2+p$.  \vspace{2 pt}

\item $s_1 = s_2 = 0$, $s_3 \ne 0$.  Then $C = s_3 (a/\alpha)^p$.  But $\alpha/a = w^p$ is a $p$-th power in $F''$ by Remark \ref{AlphaA}, so $C = s_3 w^{-p^2}$.  Furthermore, $m=2$ and  $v_{F''}(w) = 1$, so $\gf(L_a/F'') = 2$.  
\vspace{2 pt}

\item $s_1 = s_3 = 0$, $s_2 \ne 0$.   Then $C = s_2 b/\alpha^p = s_2 \lambda^{-1} (a^p/\alpha)^p$.  Set $w = \alpha/a^p$, so $m = 1$, $v_{F''}(w) = 1$ and $\gf(L_a/F'') = 2$.  
\end{enumerate}

\vspace{2 pt}

\noindent {\bf Case 2}.  $A = 0$, $r \ne 0$ and $y^{p^3}-y + C = 0$ over $F''$, with $C = r a^p/\beta$ by \eqref{12Fieldy}.  Then  $w = C^{-1}$, $m = 0$ and $v_{F''}(w) = p^2-p+1$, so $\gf(L_a/F'') = p^2-p+2$.

\vspace{5 pt}

Thanks to Proposition \ref{BS} and Lemma \ref{BC}, judicious use of Baer sums shows that in Case 1, the conductor exponent satisfies $\gf \le p^2+p$ for all $\bfs_{12} = (s_1,s_2,s_3,s_4)$ with $r=0$, or equivalently, $s_4 = (-\lambda' s_1)^p p^{p-2}$.  Case 2 covers all parameters of the form $\bfs_{12} = (0,0,0,s_4)$ with $s_4 \ne 0$.  By combining Cases 1 and 2, we have $\gf \le p^2+p$ for all possible parameters.
\end{proof}

\begin{prop} \label{E2E2Field}
If $\cV$ represents a class in $\Ext^1_{[p]}(\cE_2,\cE_2)$ and $L = K(V)$, then 
$$
\gf(L/F') \le \begin{cases} \frac{1}{2}(p^2+1) &\text{if } p \text{ is odd}, \vspace{2 pt} \\ 4 &\text{if } p = 2. \end{cases}
$$
\end{prop}

\proof
A more precise result, in terms of the parameter vector $\bfs_{22} = s_1,s_2,s_3$ for the associated Honda system $\rM$ can be found at the end of the proof.  Here $L_\alpha$ is the extension of $F''$ generated by the points $\varphi$ in the fiber over $\psi'_\alpha \ne 0$.

We have $\varphi(e_1) = (\0,\gamma,\beta,\alpha)$, $\varphi(e_2) = (\0,\gamma,\beta)$ and $\varphi(e_3) = (\0,\alpha^p)$.  Since $\rV^5 = 0$, we can set   
$$
\varphi(e_4)=(\0,d_4,d_3,d_2,d_1,d_0),     
$$
so $\varphi(e_5) = \varphi(\rV e_4) = (\0,d_4,d_3,d_2,d_1)$.  Then $\varphi(\rF \rV e_4) = 0$ gives 
$$
d_4^p \equiv d_3^p \equiv d_2^p \equiv d_1^p \equiv 0 \pmod{p},
$$ 
and thus $\varphi(e_6) = \varphi(\rF e_4) = (\0,d_0^p)$.    By applying $\varphi$ to $\rF e_6 +\sigma(s_1)e_3 = 0$, we find that
$
(\0,d_0^{p^2}) + \sigma(s_1) (\0,\alpha^p) = 0,
$
so
$$
d_0^{p^2} + \sigma(s_1) \alpha^p \equiv 0 \pmod{p}.
$$
By Proposition \ref{E2E2}, 
$
\rV e_5 = \lambda' (e_6 + s_1 e_1 + s_2 e_2 + s_3 e_3).  
$
Applying $\varphi$ gives
\begin{eqnarray*}
(\0,d_4,d_3,d_2) &=&  \lambda' \left( (\0,d_0^p) \CWplus  s_1(\0,\gamma,\beta,\alpha)  \CWplus s_2(\0,\gamma,\beta) \CWplus s_3(\0,\alpha^p) \right) \\
&=&\lambda'  (\0, s_1^{\frac{1}{p^2}} \gamma, s_1^{\frac{1}{p}}\beta + s_2^{\frac{1}{p}} \gamma,  d_0^p + s_1 \alpha + s_2 \beta + s_3 \alpha^p - \delta),
\end{eqnarray*}
where $\delta = \frac{1}{p} \lambda' \sum_{j=1}^{p-1} \, \binom{p}{j} (s_1^{\frac{1}{p}} \beta)^j (s_2^{\frac{1}{p}} \gamma)^{p-j}$, with $\ord_p(\delta) \ge \ord_p(\beta^{p-1}\gamma) = 1+ \frac{2p}{p^3-1}$.  Hence $\delta \equiv 0 \pmod{p}$ and we have the mod $p$ congruences:
\begin{equation} \label{E2E2d2d3eq}
d_4 \equiv (\lambda' s_1)^{\frac{1}{p^2}} \gamma, \quad d_3 \equiv (\lambda')^{\frac{1}{p}}(s_1^{\frac{1}{p}} \beta + s_2^{\frac{1}{p}} \gamma), \quad d_2 \equiv \lambda' (d_0^p +A), 
\end{equation}
with $A = s_1\alpha + s_2 \beta + s_3 \alpha^p$.  Vanishing of the Hasse-Witt exponential gives:
$$
\begin{array}{r l l}
\label{E2xi4} \xi(\varphi(e_4)) &= \, \frac{1}{p^4} d_4^{p^4} + \frac{1}{p^3} d_3^{p^3} + \frac{1}{p^2} d_2^{p^2} + \frac{1}{p} d_1^p + d_0 \equiv 0 &\pmod{p}, \vspace{4 pt} \\
\label{E2xi5} 
\xi(\varphi(e_5)) &=  \, \frac{1}{p^3} d_4^{p^3} + \frac{1}{p^2} d_3^{p^2} + \frac{1}{p} d_2^p + d_1 \equiv 0 &\pmod{p}.  
\end{array}
$$

By \eqref{E2E2d2d3eq}, $\ord_p(d_4^{p^3}) \ge \frac{p^5+p^4}{p^3-1} > 5 $ and $\ord_p(d_3^{p^3}) \ge \frac{p^5+p^3}{p^3-1} > 4$.  The ordinals of the interior terms in the binomial expansion of $(s_1^{\frac{1}{p}}\beta  + s_2^{\frac{1}{p}} \gamma)^p$ are bounded below by
$
\ord_p(p \beta^{p-1} \gamma) = 2 + \frac{2p}{p^3-1}.
$
A lower bound for ordinals of the interior terms of $(s_1\alpha  + s_2 \beta)^p$ is given by
$
\ord_p(p \alpha^{p-1} \beta) = 1 + \frac{2p^2}{p^3-1}.
$
From the relations in Table \ref{Bank} and Lemma \ref{ppower}(ii), we have mod $p$ congruences:
$$
\begin{array}{c l l}
\frac{1}{p} d_3^p &\equiv& \frac{1}{p} \lambda'(s_1 \beta^p  + s_2 \gamma^p) \equiv -\lambda' (s_1 \alpha + s_2 \beta)  \hspace{15 pt}  \text{and}  \vspace{3 pt} \\
\frac{1}{p^{p+1}} d_3^{p^2} &\equiv& \frac{1}{p} (-\lambda')^p (s_1 \alpha + s_2 \beta)^p.
\end{array}
$$
Since the interior terms in the expansion of $\frac{1}{p} (s_1 \alpha + s_2 \beta)^p$ are integral, as is $\frac{1}{p} \beta^p$, we find that
$$
\frac{1}{p^2} d_3^{p^2} \equiv p^{p-2} (-\lambda' s_1 \alpha)^p \pmod{p^{p-1}}.
$$
The above equations for vanishing of $\xi(\varphi(e_4))$ and $\xi(\varphi(e_5))$ now reduce to 
\begin{eqnarray}
\label{22xi4}   \frac{1}{p^2} d_2^{p^2} + \frac{1}{p} d_1^p + d_0 &\equiv& 0 \pmod{p}  
, \\
\label{22xi5}   \frac{1}{p} d_2^p + d_1 &\equiv& (\lambda' s_1)^p p^{p-2}  \alpha^p  \pmod{p}.  
\end{eqnarray}

Set $d_0 = \alpha x$, $d_1 = \beta y$, $d_2 = \gamma z$.  Using the relations in Table \ref{Bank}, the last equation in \eqref{E2E2d2d3eq} and equations \eqref{22xi5} and \eqref{22xi4} simplify to
\begin{equation} \label{22xyz}
z = x^p + A \alpha^{-p} \text{ in } \rR_\gamma,  \hspace{10 pt}  y = z^p + B \text{ in } \rR_\beta, \hspace{10 pt}    x = y^p - p^{p-1} z^{p^2} \text{ in } \rR_\alpha,
\end{equation}
where $A = s_1\alpha + s_2 \beta+ s_3\alpha^{p}$ as above, $B = (\lambda' s_1)^2 \, \alpha^2 \beta^{-1}$ if $p = 2$ and $B = 0$ otherwise.  If $p = 2$ and $s_1 \ne 0$, then $\ord_2(B) = \frac{1}{7}$.  

We may assume that $\cV$ is a non-split extension, so some parameter $s_j$ is not zero in $k$.  The ordinals of $x,y,z$ can be estimated from \eqref{22xyz}, as follows.  First check that $\ord_p(z) \le 0$ and then arrive at
$$
\ord_p(x^p) = \ord_p(y^{p^2}) = \ord_p(z^{p^3}) = \ord_p\left(A \alpha^{-p} \right) \ge -\frac{p^2-1}{p^3-1}.	
$$

We use the rules for arithmetic on fractional elements from Appendix \ref{ArtinSchreier} to simplify \eqref{22xyz}.  Since $\ord_p(\gamma) > \ord_p(\beta)$, the second equation there is valid in $\rR_\gamma$ and $B=0$ in $\rR_\gamma$ even when $p = 2$.  Similarly, the third equation passes ot $R_\gamma$ and $p^{p-1} z^{p^2} = 0$ in $\rR_\gamma$.  In $\rR_\gamma$, we therefore have
\begin{equation} \label{new22xyz}
{\rm (i) } \hspace{4 pt} z = x^p + A \alpha^{-p}, \hspace{30 pt}  {\rm (ii) } \hspace{4 pt}  y = z^p, \hspace{30 pt}  {\rm (iii) }  \hspace{4 pt}  x = y^p.
\end{equation}
Since $py^{p^2-1}$ is integral, \eqref{new22xyz}(ii) implies that $y^{p^2} = z^{p^3}$ in $\rR_\gamma$ by {\bf C2}.  Similarly, \eqref{new22xyz}(iii) leads to $x^p = y^{p^2}$ in $\rR_\gamma$.  Thus, $z$  satisfies 
$
f_\alpha(Z) = Z^{p^3} - Z + A \alpha^{-p} = 0
$
in $\rR_\gamma$.

For the converse, let $z$ be a root of $f_\alpha$ in its splitting field over $F'$ and use the second and third congruences in \eqref{22xyz} as equations to defining $y$ and then $x$.  The congruences established above imply that $z = x^p + A$ in $\rR_\gamma$.  Thus, all non-zero points in the fiber of $V$ over $\psi'_\alpha$ occur when $z$ ranges over the roots of $f_\alpha$.   

If $s_2 = s_3 = 0$, then $A \alpha^{-p} = s_3$, so $f_\alpha$ is separable over the residue field of $F'$ and $L_\alpha/F'$ is unramified and possibly split.  In all other cases, we use Proposition \ref{AS} to verify that $L_\alpha/F'$ is a totally ramified elementary abelian $p$-extension of degree $p^3$.  By Proposition \ref{E2Field}(ii) the ramification index of $F'/\Q_p$ is given by
\begin{equation} \label{eF'}
e_{F'} = \begin{cases} \frac{1}{2}(p^3-1) &\text{if } p \text{ is odd},  \vspace{2 pt}
 \\ 7 &\text{if } p = 2.
 \end{cases}
\end{equation}
Proposition \ref{AS} holds with $m = 0$ and $w = C^{-1}$ and so
$$
\gf(L_\alpha/F') = v_{F'}(w) + 1 = e_{F'} \, \ord_p(w) + 1.
$$   
We have the following cases:
\begin{enumerate}[\hspace{3 pt} $\bullet$]
\item If $s_1 \ne 0$,  then $\ord_p(w) = \ord_p(\alpha^{p-1}) = \frac{p^2-1}{p^3-1}$.  Hence
$$
\gf(L_\alpha/F') = \begin{cases} \frac{1}{2}(p^2+1) &\text{if } p \text{ is odd}, \vspace{2 pt} \\ 4 &\text{if } p = 2. \end{cases}
$$

\vspace{2 pt}
 
\item If $s_1=0$ and $s_2 \ne 0$, then $\ord_p(w) = \ord_p(\alpha^p \beta^{-1}) = \frac{p-1}{p^3-1}$.  Hence
$$
\gf(L_\alpha/F')  = \begin{cases} \frac{1}{2}(p+1) &\text{if } p \text{ is odd}, \vspace{2 pt} \\ 2 &\text{if } p = 2.   \quad \qed \end{cases}   
$$
\end{enumerate}

\subsection{Local corners}  \label{LocalCorners}
In this subsection, $p = 2$ and $K = \Q_2$.  For $i = 1$ or $2$, denote by $\cE_i$ the simple group scheme corresponding to the Honda system $(\rM_i,\rL_i)$ of Proposition \ref{LocalSimple}, with parameter necessarily equal to 1 in $\F_2$.  By Propositions \ref{E1Field} and \ref{E2Field}, the corresponding Galois modules $E_1$ and $E_2$ have the same field of points, $F= \Q_2(\zeta,\varpi)$, where $\zeta \in \Mu_7$ satisfies $\zeta^3+\zeta+1 \equiv 0 \pmod{2}$ and $\varpi^7 = 2$.  Let $\rR = \F_2[\Delta]$, where $\Delta = \Gal(F/\Q_2) = \lr{\sigma, \tau}$, with the following action:
\begin{equation} \label{staction}
\tau(\zeta) = \zeta^2, \quad \tau(\varpi) = \varpi, \hspace{30 pt} \sigma(\zeta) = \zeta, \quad \sigma(\varpi) = \zeta \varpi
\end{equation}
Thus $\tau$ is a Frobenius, $\tau \sigma \tau^{-1} = \sigma^2$ and $\Delta$ is the non-abelian group of order 21.  

Recall the $\F_2[\Delta]$-homomorphism $\psi\!:  \gR_\lambda(F) \to E_1$ and the twisted action of $\zeta$ on $E_1$ defined in Remark \ref{Pa}.  We consistently use $P$ to denote the unique non-zero element of $E_1$ fixed by $\tau$, so $P = \psi_{-\varpi}$.  Then $P$, $\sigma(P) = \zeta P$, $\sigma^2(P) = \zeta^2 P$ is an $\F_2$-basis for $E_1$ affording the matrix representations 
\begin{equation} \label{stEnd1}
s_1 = \rho_{E_1}(\sigma) = \left[\begin{smallmatrix}0&0&1\\ 1&0&1\\0&1&0 \end{smallmatrix}\right] \quad \text{and} \quad t = \rho_{E_1}(\tau) = \left[\begin{smallmatrix} 1&0&0\\ 0 &0&1\\0&1&1\end{smallmatrix}\right].
\end{equation}
Thus, $c_1 = \sigma^3+\sigma+1$ annihilates $E_1$.  In general, $\sigma(\epsilon P) =  \epsilon \zeta P$ and $\tau(\epsilon P) = \epsilon^2 P$ for all $\epsilon$ in $\F_8$.  

Similarly, $P' = \psi'_{-\varpi^3}$ is the unique non-zero element of $E_2$ fixed by $\tau$.  Then $P'$, $\zeta P'$, $\zeta^2 P'$ is an $\F_2$-basis for $E_2$ affording the matrix representations 
\begin{equation} \label{stEnd2}
s_2 = \rho_{E_2}(\sigma) = \left[\begin{smallmatrix}1&0&1\\ 1&1&1\\0&1&1 \end{smallmatrix}\right] = s_1^3 \quad \text{and} \quad  \rho_{E_2}(\tau) = t.
\end{equation}
Here, $c_2 = \sigma^3+\sigma^2+1$ annihilates $E_2$.  Also, $\sigma(\epsilon P') = \epsilon \zeta^3 P'$ and $\tau(\epsilon P') = \epsilon^2 P'$ for all $\epsilon$ in $\F_8$.

Recall that the irreducible representations of $\Delta$ over $\F_2$ are given by $E_1$ and $E_2$, as well as the trivial representation and a 2-dimensional representation $D$ that splits over $\F_4$.  The following $\rR$-module invariants, which we call {\em corners}, are used to impose the information about extensions of $\cE_i$ by $\cE_j$ over $\Z_2$ obtained in this section on the global extension problem in \S\ref{3Folds}.

\begin{Def}  \label{Cors}
Let $M$ be a finitely generated $\rR$-module, $c_1 = \sigma^3+\sigma+1$ and $c_2 = \sigma^3+\sigma^2+1$.  The {\em corners} are the following $\F_2$-vector subspaces of $M$:  
\begin{equation*} 
C_j(M) = \{ m \in M \, \vert \, \tau(m) = m \text{ and }c_j(m) = 0 \} \quad \text{for } j = 1,2.
\end{equation*}
\end{Def}

Because $\sigma$ has odd order, $M = M_0 \oplus M'$, where $M_0$ is the submodule on which $\sigma$ acts trivially and $M' = (\sigma-1)(M)$.  Since $\tau$  normalizes $\langle \sigma\rangle$, we have $\rR m_j  \simeq E_j$  for any non-zero $m_j$ in $C_j(M)$.  Indeed, the annihilator of $m_j$ is the left ideal generated by $\tau-1$ and $c_j$.  If $\rR \, C_j(M)$ denotes the $\rR$-module generated by $C_j(M)$, then
$$
M' =  \rR \, C_1(M)  \oplus  \rR \, C_2(M)  \simeq E_1^{n_1} \oplus E_2^{n_2},
$$
where $n_j = \dim_{\F_2} (C_j(M))$. 

\vspace{5pt}

By Theorem \ref{CondBd}, if $\cV$ represents a class in $\Ext^1_{[2],\Z_2}(\cE_i,\cE_j)$,  then the field generated by the points of $V$ is contained in the maximal elementary 2-extension $T$ of $F$ with ray class conductor exponent at most 6.   We write $\sigma$ for an element of order 7 in $\Gal(T/\Q_2)$ projecting to the element $\sigma$ in $\Delta$ chosen above.   

For $n \ge 1$, let $U_F^{(n)} = 1+ \varpi^n \cO_\F$ be the standard filtration of local units, with $U_F^{(0)} = U_F$.  The Artin symbol of  local class field theory gives an isomorphism   
$$
\Theta_{T/F}\!: \,  F^\times/U_F^{(6)} F^{\times 2} \, \xrightarrow{\sim} \, \Gamma=\Gal(T/F), 
$$
with $\delta \in \Delta$ acting by $\delta\Theta_{T/F}(x)\delta^{-1} = \Theta_{T/F}(\delta(x))$. The $\rR$-module structure of the inertia group $\Gamma_0$ will be described in terms of the following corners:
$$
g_1 =\Theta_{T/F}(1+\varpi+\varpi^3), \quad g_2 = \Theta_{T/F}(1+\varpi^3), \quad  g_3 = \Theta_{T/F}(1+\varpi^5).
$$

\begin{prop} \label{Max2}
Let $T_n$ be the maximal subfield of $T$ with ray class conductor exponent at most $n$ over $F$.  We have the diagram of fields and Galois groups:
\begin{center}
\begin{tikzpicture}[scale=1, thin, baseline=(current  bounding  box.center)]	
    \draw (0,0) -- (0,3.7);
    \draw (0,0) node {$\bullet$};   \draw (.35,0) node {$\Q_2$};
    \draw (-.3,.45) node {$\Delta$};
    \draw (0,.75) node {$\bullet$};  \draw (1.95,.7) node {$F = \Q_2(E_1) = \Q_2(E_2)$};  
    \draw(-.6,1.05) node {\rm unram};
    \draw (0,1.4) node {$\bullet$};      \draw (-.33,1.45) node {$T_0$};
    \draw (.7,2.15) -- (1,2.15) -- (1,3.75) -- (.7,3.75); \draw (1.95,2.95) node {$\rR g_2 \oplus \rR g_3$};
    \draw (0,2.15) node {$\bullet$};    \draw (-.33,2.2) node {$T_2$};
    \draw (0,2.9) node {$\bullet$};    \draw (-.33,2.95) node {$T_4$};
    \draw (.4,3.3) node {$\rR g_3$};  
    \draw (0,3.7) node {$\bullet$};   \draw (-1.88,3.75) node {$L = \Q_2(W) \, \subseteq  \, T = T_6$};  
\end{tikzpicture}
\end{center}
Then $g_0 =\Theta_{T/F}(\varpi)$ is a Frobenius element and 
$$
\Gamma = \Gal(T/F) = \rR g_0 \oplus \rR g_1 \oplus \rR g_2 \oplus \rR g_3.
$$  
There are $\rR$-module isomorphisms $\rR g_0 \simeq \F_2$, $\rR g_1 \simeq E_1$ and $\rR g_2 \simeq \rR g_3 \simeq E_2$.   
\end{prop}

\begin{proof}
If $L/F$ is an abelian extension of conductor $n$, then $U_F^{(n)}$ is contained in the group of norms $N_{L/K}(L^\times)$ by class field theory \cite[\!XV \S2]{Ser1}. For $m\le 3$, the Artin map $\Theta_{T/F}$ therefore  induces an isomorphism $F^\times/U_F^{(2m)}F^{\times 2} \xrightarrow{\sim} \Gal(T_{2m}/F)$.  Furthermore,
$$
 U_F^{(2m+1)} F^{\times 2}/U_F^{(6)}F^{\times 2} \simeq\Gal(T/T_{2m}),
$$ 
since $U_F^{(2m)}U_F^2 = U_F^{(2m+1)} U_F^2$.  The action of $\rR$ on $g_0$ is trivial because $\sigma(\varpi) = \zeta \varpi$ is in $\varpi F^{\times 2}$ and $\tau(\varpi) = \varpi$.  As an $\rR$-module, the inertia group 
$$
\Gamma_0 = \Gal(T/T_0) \simeq U_F/U_F^{(6)}U_F^2
$$ 
is a direct sum of copies of $E_1$ and $E_2$, since $\sigma$ has no non-trivial fixed points in it.  In terms of the corner invariants of $\rR$-modules in Definition \ref{Cors}, one checks that $C_1(\Gamma_0) = \lr{g_1}$ and $C_2(\Gamma_0) = \lr{g_2,g_3}$. 
\end{proof} 

As needed for our applications, let $\cE = \cE_1 \oplus \cE_2$ and $\gh = \Hom_{\F_2}(E,E)$.   Suppose that $\cW$ represents a class in $\Ext^1_{[2],\Z_2}(\cE,\cE) $ and let $L = \Q_2(W)$.   Then 
$$
\Q_2(E_1) = \Q_2(E_2) = \Q_2(E) = F \subseteq L \subseteq T.
$$
The extension class of the Galois module $W$ corresponds to a cohomology class $[\widetilde{\chi}_\cW]$ in $H^1(\Gal(L/\Q_2),\gh)$.  Since $\Gamma = \Gal(T/F)$ acts trivially on $\gh$ and $\Delta = \Gal(F/\Q_2)$ has odd order, $H^1(\Delta,\gh) = 0$ and so \eqref{InfRes} gives 
\begin{equation} \label{ToToverF}
0 \to H^1(\Gal(T/\Q_2),\gh) \xrightarrow{{\rm res}_\Gamma} H^1(\Gamma,\gh)^\Delta = \Hom_\rR(\Gamma,\gh),
\end{equation} 
with $\rR = \F_2[\Delta]$.  Also, the inflation map $H^1(\Gal(L/\Q_2),\gh) \xrightarrow{{\rm inf}^T_L} H^1(\Gal(T/\Q_2),\gh)$ is injective.  

\begin{Def} \label{HondaChar}
The {\em Honda character} of $\cW$, denoted $\chi_{\cW}\!: \, \Gamma \to \gh$, is given by $\chi_\cW = \res_\Gamma \, {\rm inf}^T_L([\widetilde{\chi}_\cW])$.
\end{Def}

\noindent Note that $[\widetilde{\chi}_\cW]$ is trivial if and only if $\chi_\cW = 0$.   For elements $g$ in $\Gamma$, the representation afforded by $W$ has the form  
$
\rho_\cW(g) = \left[\begin{smallmatrix} I & \chi_\cW(g) \\  0 & I \end{smallmatrix}\right]
$
in the parabolic matrix group $\cP = \cP_{E,E}$ of \eqref{defP}. 

Let $\gh_{ij} = \Hom_{\F_2}(E_i,E_j) \simeq M_3(\F_2)$, via the isomorphism afforded by the special bases for $E_1$ and $E_2$ in \eqref{stEnd1} and \eqref{stEnd2}, respectively.   The block decomposition
$$
\chi_\cW(g) = \left[\begin{array}{c | c} \chi_{11}(g) & \chi_{21}(g) \\ \hline \chi_{12}(g)  & \chi_{22}(g) \end{array}\right] \text{ in } M_6(\F_2) \simeq \Hom_{\F_2}(E,E),  \vspace{3 pt}
$$
defines $\rR$-homomorphisms $\chi_{ij}\!: \, \Gamma \to \gh_{ij}$.

Table \ref{hStructure} gives the corner spaces of each $\gh_{ij}$ in terms of the $3 \times 3$ identity $I$ and the matrix $t$ of \eqref{stEnd1}.  The $\F_2$-span of a subset $S$ of $\gh_{ij}$ is denoted by $\lr{S}$ and $D$ is the unique 2-dimensional irreducible $\Delta$-module over $\F_2$. 

\begin{table}[h] 
\begin{center}
{\begin{caption}{$\rR$-module decomposition of $\gh_{ij}$ \label{hStructure}} \end{caption} }
\end{center}
$$
\begin{array}{| c | c | c | c | }
\hline
                              &    \rR\text{-module decomposition} & C_1 & C_2 \\
\hline
\gh_{11} & \F_2 \oplus D \oplus E_1 \oplus E_2 & \lr{t^2} & \lr{t} \\
\hline 
\gh_{12} & E_1 \oplus E_1 \oplus E_2 & \lr{I,t} & \lr{t^2} \\
\hline
\gh_{21} & E_1 \oplus E_2 \oplus E_2 & \lr{t} & \lr{I,t^2} \\
\hline
\gh_{22} & \F_2 \oplus D \oplus E_1 \oplus E_2 &  \lr{t} & \lr{t^2} \\
\hline
\end{array}
$$
\end{table}

\noindent {\em Verification of Table} \ref{hStructure}.
Let $[f]$ be the matrix representation of $f $ in $\gh_{ij}$ with respect to the special bases for $E_1$ and $E_2$.  An element $\delta$ of $\Delta$ acts on $f$ by $(\delta f)(x) = \delta(f(\delta^{-1} x))$.  In terms of matrices, $[f] \mapsto \rho_{E_j}(\delta) \,  [f] \rho_{E_i}(\delta)^{-1}$.  We have $t = \rho_{E_1}(\tau) = \rho_{E_2}(\tau)$ and $s_i = \rho_{E_i}(\sigma)$.  Hence $\tau$ acts on $[f]$ via conjugation by $t$ in all cases.  Matrices in the corner spaces commute with $t$ and thus are polynomials in $t$.  The action of $\sigma$ on $\gh_{ij}$ is given by $[f] \mapsto s_j [f] s_i^{-1}$.  For each power $t^n$, one checks whether it is annihilated by the characteristic polynomial $c_1$ or $c_2$ of Definition \ref{Cors} to determine which corner space $t^n$ is in.   For $\gh_{11}$ and $\gh_{22}$, after taking the span of the corners, there remains a 3-dimensional space  that commutes with $s$, so consists of  polynomials in $s$ and therefore is isomorphic to $\F_2  \oplus D$.  \qed

\vspace{5 pt}

Thanks to $\rR$-linearity, a Honda character $\chi_\cW$ is determined by its values on the $\rR$-module generators $g_0$, $g_1$, $g_2$, $g_3$ for $\Gamma$ used in Proposition \ref{Max2}.  

\begin{lem}  \label{LocalChi}
If $\chi_\cW$ is a  Honda character, then $\chi_\cW(g_2)$ has the form $\left[ \begin{smallmatrix} * & 0 \\ * & *  \end{smallmatrix} \right]$  and  $\chi_\cW(g_3) =0$ or  $\left[ \begin{smallmatrix} 0 & 0 \\ t^2 & 0  \end{smallmatrix} \right]$.
\end{lem}

\begin{proof}
Refer to the conductor bounds in \S \ref{ptsEEp} using $p=2$.  By Proposition \ref{E2E1Field}, the field of points of an extension of $\cE_2$ by $\cE_1$ annihilated by 2 has ray class conductor exponent at most 2 over $F$.  Since $g_3$ and $g_2$ act trivially on that field, $\chi_{21}(g_3)$ and $\chi_{21}(g_2)$ vanish.  This explains $\chi_\cW(g_2)$.   By Propositions \ref{E1E1Field} and \ref{E2E2Field} extensions of $\cE_1$ by $\cE_1$ or $\cE_2$ by $\cE_2$ give rise to subfields of $T$ of conductor exponent at most 4 over $F$ and so $\chi_{11}(g_3) = \chi_{22}(g_3) = 0$.  In Table \ref{hStructure}, $t^2$ is the unique corner of $\gh_{12}$ generating an $\rR$-submodule isomorphic to $E_2$.  By Proposition \ref{Max2},  $\rR g_3 \simeq E_2$.  Since $\chi_{12}$ is an $\rR$-module map, $\chi_{12}(g_3)$ is trivial or equals $t^2$.
\end{proof}

\section{Arithmetic of favorable fields} \label{FavFieldsSection} 
\numberwithin{equation}{section}

Throughout this section, $F$ denotes the Galois closure of a favorable field  $F_0$ of discriminant $\pm 2^{2g} N$, with $N$ odd and squarefree (see Definition \ref{FavField}).   Write $\cI_\gp = \cI_\gp(F/\Q)$ for the inertia group and $\cD_\gp = \cD_\gp(F/\Q)$ for the decomposition group at a prime $\gp$ over $2$ in $F$.  We discuss the number theory of $F$ and some applications to favorable abelian varieties.

\begin{prop} \label{Fproperties}
The field $F$ enjoys the following properties.
\begin{enumerate}[{\rm i)}]
\item The permutation action of $\Gal(F/\Q)$  on the left cosets of $\Gal(F/F_0)$  is  isomorphic to $\cS_{2g+1}$. 

\vspace{2 pt}

\item The completion $F_\gp$ is isomorphic to $\Q_2(\Mu_{2g+1},\pi)$ with $\pi^{2g+1}=2$ and the inertia group $\cI_\gp$ is generated by a $(2g+1)$-cycle.  Moreover, $\cD_\gp$ is the metacyclic group $\cI_\gp \rtimes \lr{\Frob_\gp}$, where $\Frob_\gp$ is any Frobenius at $\gp$.  The order of $\Frob_\gp$ is the order of $2$ in $(\Z/(2g+1)/\Z)^\times$.  

\vspace{2 pt}

\item Write $N^* = \pm N \equiv 1+4 a \pmod{8}$.  Then $\sqrt{N^*}$ is in $F$ and $a= 0$ if $\Frob_\gp$ is an even permutation, but $a = 1$ otherwise.

 \vspace{2 pt}

\item The inertia group $\cI_\gq = \cI_\gq(F/\Q)$ at each prime $\gq \vert N$ is generated by a transposition $\sigma_\gq$.
\end{enumerate} 
\end{prop}

\begin{proof}
\noindent  We first verify items (ii)--(iv) and then use them to prove (i).

\vspace{2 pt}

\noindent ii) Since every unit of $\Z_2$ is a $(2g+1)$-power, there is a prime element $\pi$ over 2 in the completion $\wtF_0$ of $F_0$ at the prime over 2, such that $\pi^{2g+1} = 2$ and we have $\wtF_0 = \Q_2(\pi)$.   Lemma \ref{CompletionClosure} below shows that $F_\gp = \Q_2(\Mu_{2g+1}, \pi)$.  Thus $\cI_\gp = \lr{\sigma}$ is cyclic of order $2g+1$.  We have $\sigma(\pi) = \zeta \pi$ for a generator $\zeta$ of $\Mu_{2g+1}$ and $\sigma(\zeta) = \zeta$.  Since $\sigma$ acts transitively on the roots of $x^{2g+1}-2$, it is represented by a $(2g+1)$-cycle.   Choose a Frobenius $\tau$ at $\gp$ to satisfy $\tau(\zeta) = \zeta^2$ and $\tau(\pi) = \pi$, so $\tau \sigma \tau^{-1} = \sigma^2$.  Any other choice of Frobenius lies in $\tau \cI_\gp$ and so is conjugate to $\tau$.

\vspace{2 pt}

\noindent iii) Since the quadratic field $\Q(\sqrt{N^*})$ is contained in $F$, it is unramified over 2 and so $N^* \equiv 1 \pmod{4}$.  In fact, $N^*$ is a square in $\Q_2$ precisely when $\tau$ acts trivially on $\sqrt{N^*}$, or equivalently, $\tau$ is an even permutation.  

\vspace{2 pt}

\noindent iv) Since $q$ exactly divides $N$, \cite[Satz II]{vdW} shows that $\sigma_\gq$ is a transposition. 

\vspace{2 pt}

\noindent i) By items (ii) and (iv) of the Proposition, $\Gal(F/\Q)$ contains a $2g+1$-cycle and a transposition.  If the action of $\Gal(F/\Q)$ on the roots of a minimal polynomial for $F_0/\Q$ is primitive, then $\Gal(F/\Q) \simeq \cS_{2g+1}$.  Primitivity is obvious when $2g+1$ is prime, but holds in general.  If not, there is an intermediate field $K$ such that $F_0 \supsetneq K \supsetneq \Q$, so $n = \fdeg{K}{\Q}$ and $\fdeg{F_0}{K}$ are odd and at least 3.  Because $N$ is odd and squarefree and 
$$
2^{2g} N^* = \disc_{F_0/\Q} = N_{K/\Q}(\disc_{F_0/K}) \, \disc_{K/\Q}^{[F_0:K]},
$$
we find that $\pm \disc_{K/\Q}$ is a power of 2.  But the unique prime over 2 in $K$ is totally ramified and tame, so $\disc_{K/\Q} = \pm 2^{n-1}$.  Hence the root discriminant of $K/\Q$ is $2^{1-\frac{1}{n}}$, contradicting Odlyzko's lower bound \cite{dyd}, since $K$ has at least one real place.
\end{proof}

\begin{Not}  \label{PairResDef}
For the rest of this section, $K$ denotes a subfield of $F$ such that $\Gal(F/K)$ is the centralizer of a transposition in $\Gal(F/\Q) \simeq \cS_{2g+1}$ and $K$ is called a {\em pair-resolvent} for $F$.
\end{Not}

Since the symmetric group is doubly transitive, $K$ is well-defined up to isomorphism.  Moreover $F$ is the Galois closure of $K/\Q$.  Let $f$ be an irreducible polynomial over $\Q$ of degree $2g+1$ with splitting field $F$.  Given distinct roots  $a_1, a_2$ of $f$, the field $\Q(a_1+a_2)$ is a pair-resolvent for $F$ and has degree $\binom{2g+1}{2}$ over $\Q$.  More generally, if $k'$ is the splitting field of an separable, irreducible polynomial over a field $k$ and $\Gal(k'/k)$ is doubly transitive, one can obtain a well-defined pair-resolvent for $k'$ over $k$ from the roots of $f$ in this way.

The next several Lemmas provide arithmetic information about $K$.  For convenience, take $H = \Gal(F/K)$ to be the subfield of $F$ fixed by the centralizer $C$ of the transposition (12).

\begin{lem} \label{ramK}
The ramification index $e_\gp(K/\Q) = 2g+1$ for primes $\gp \vert 2$ in $K$.  
\end{lem}

\begin{proof}
Let the prime $\gP$ lie over $\gp$ in $F$.  By Proposition \ref{Fproperties}(ii), the inertia group $\cI_\gP = \cI_\gP(F/\Q)$ is generated by a $(2g+1)$-cycle $\sigma$.  If $d = \gcd\{2g+1,j\}$, then $h = \sigma^j$ is a product of $d$ disjoint cycles of length $(2g+1)/d$.  Thus, $h$  preserves the set $\{1,2\}$ only if $h= 1$ and so 
$\cI_\gP \cap C = \{1\}$.  We conclude that $\gP$ is unramified in $F/K$ and the ramification index $e_\gp(K/\Q) = e_\gP(F/\Q) = 2g+1$.
\end{proof}

\begin{Rem} \label{OnePrime}
Let $r$ be the multiplicative order of 2 in $(\Z/(2g+1)\Z)^\times$.  One can show that there is one prime over $2$ in $K$ exactly when $2g+1$ is prime and one of the following holds: \, i) $r = g$, with $g$ odd; \, or \,  ii) $r = 2g$.  We need only $g = 3$, which can be checked directly.
\end{Rem}

\begin{lem}  \label{TooReal}
Denote the ideal class group of $K$ by $\cC_K$ and let $\Omega_K$ be the maximal elementary $2$-extension of $K$ of ray class conductor $\gp^{2g} \infty$, where $\gp$ is a fixed prime over $2$ in $K$ and $f$ is its residue degree in $K/\Q$.  If $F_0$ admits exactly $r_1$ real embeddings, then 
$$
\dim_{\F_2} \Gal(\Omega_K/K) \ge \dim_{\F_2} \cC_K[2] + (g-f)(g+1) +
\textstyle{\frac{1}{4}}\left(r_1-1\right)^2.
$$
\end{lem}

\proof
Let
$
\cU = \{\alpha \in K^\times \, | \, \text{the ideal }(\alpha) = \ga^2 \text{ is the square of an ideal}\}.
$
If the ideal $\ga$ represents a class in $C_K[2]$, then a generator $\alpha$ for the principal ideal $\ga^2 = (\alpha)$ is well-defined up to multiplication by an element of the group $U_K$ of units of $K$.  We therefore have an isomorphism
$$
\cC_K[2] \xrightarrow{\sim} \cU/U_K K^{\times 2} \quad \text{induced by} \quad \ga \mapsto \alpha U_K K^{\times 2}.
$$
Counting $-1$, the Dirichlet unit theorem implies that $U_K K^{\times 2}/K^{\times 2} \simeq U_K/U_K^2$ has dimension $r_1' + r_2'$ over $\F_2$, where $r_1'$ is the number of distinct real embeddings of $K$ and $r_1'+2r_2' = \fdeg{K}{\Q} = \binom{2g+1}{2} = g(2g+1)$.  Hence
$$
\dim_{\F_2} \cU/K^{\times 2} = \dim_{\F_2} \cU/U_KK^{\times 2} + \dim_{\F_2} U_K/U_K^{\times 2} = \dim_{\F_2} \cC_K[2] + r_1'+r_2'.
$$

Let $\Kappa$ be the {\em Kummer group} for $\Omega_K/K$; that is, the subgroup of $K^\times/K^{\times 2}$ in the perfect pairing of Kummer theory
$$
\Gal(\Omega_K/K) \times \Kappa \to \Mu_2 \quad \text{given by} \quad (\gamma, \alpha K^{\times 2}) \mapsto \gamma(\sqrt{\alpha})/\sqrt{\alpha}.
$$
The coset $\alpha K^{\times 2}$ belongs to $\Kappa$ if and only if $\alpha$ is in $\cU$ and the ray class conductor condition $\gf_\gp(K(\sqrt{\alpha})/K) \le 2g$ is satisfied.  Write $U_\gp$ for the unit group in the ring of integers $\cO_\gp$ of $K_\gp$ and let $U_\gp^{(n)}= 1+ \gp^n \cO_\gp$.  Since the ramification index of $\gp$ in $K/\Q$ is $2g+1$, we find that $\alpha K^{\times 2}$ is in $\Kappa$ if and only if $\alpha$ belongs to $U_\gp^{(2g+3)} K_\gp^{\times 2}$ by \cite[Lemma C.6]{BK5}.  Thus $\Kappa$ is the kernel of the map 
$$
\cU/K^{\times 2} \, \longrightarrow \, U_\gp K_\gp^{\times 2}/U_\gp^{(2g+3)}K_\gp^{\times 2} \simeq U_\gp/U_\gp^{(2g+3)} U_\gp^2.
$$  
By Lemma \ref{r1r2}, $r_1' = \binom{r_1}{2} + r_2$ and we have $\dim_{\F_2} U_\gp/U_\gp^{(2g+3)} U_\gp^2 = f(g+1)$, so
\begin{eqnarray*}
\dim \Kappa &\ge& \dim_{\F_2} \cU/K^{\times 2} - \dim_{\F_2}  U_\gp/U_\gp^{(2g+3)} U_\gp^2 \\
&\ge& \dim_{\F_2}  C_K[2] + r_1'+r_2' - f(g+1) \\
&\ge&  \dim_{\F_2}  C_K[2]  +  (g-f)(g+1) +
\textstyle{\frac{1}{4}}\left(r_1-1\right)^2.  \hspace{10 pt}  \qed
\end{eqnarray*}

\begin{lem} \label{RamOverN}
Let $K_1$ be a subfield of $F$ quadratic over $K$ and let $q$ be a rational prime dividing $N$.  Then some prime over $q$ ramifies in $K_1/K$.
\end{lem}

\begin{proof}
Fix $H = \Gal(F/K)$ to be the centralizer of the transposition $(12)$.  Since $H$ is generated by its transpositions, any subgroup $H_1$ of index 2 in $H$ omits some transposition, say $t$.  The action of $\Gal(F/\Q) \simeq \cS_{2g+1}$ on the transpositions is transitive and Proposition \ref{Fproperties} implies that the inertia group for each prime of $\cO_F$ dividing $N$ is generated by a transposition.  Hence there is prime $\gq \vert q$ in $\cO_F$ whose inertia group is generated by $t$.  Since $t$ fixes $K$ but not $K_1 = F^{H_1}$, the prime below $\gq$ is unramified in $K/\Q$, but ramifies in $K_1/K$.  
\end{proof}

See Appendix \ref{ArtinSchreier} for a review of the notion of conductor exponent for field extensions used in the next Proposition.

\begin{prop}  \label{LProperties}
Assume that $N$ is prime.  Let $L$ be an elementary $2$-extension of $F$, Galois over $\Q$, with $L/F$ unramified outside $\{2,\infty\}$ and conductor exponent $\gf_\gp(L/F) \le 2g$ for all primes $\gp \vert 2$ in $L$.  Let $L_0$ be the maximal subfield of $L$ abelian over $\Q$.   Then the following properties hold.
\begin{enumerate}[{\rm i)}]
\item  $L_0 = \Q(\sqrt{N^*})$, with $N^* = \pm N \equiv 1 \pmod{4}$. 
\vspace{2 pt}

\item At each prime $\gq \vert N$, the inertia group $\cI_\gq(L/\Q) = \lr{\sigma_\gq}$ is generated by an involution $\sigma_\gq$ whose projection to $\Gal(F/\Q) \simeq \cS_{2g+1}$  is a transposition.

\vspace{2 pt}

\item Fix a prime $\gq$ dividing $N$.  The conjugates of $\sigma_\gq$ generate $\Gal(L/\Q)$.  
\end{enumerate}
\end{prop}

\begin{proof}  
Since $\sqrt{N^*}$ is in $F$, the field $\Q(\sqrt{N^*})$ is contained in $L_0$.  By our assumptions, $L_0/\Q$ is unramified outside $\{2,N,\infty\}$.   Since the conductor exponent of $L/F$ satisfies $\gf_\gp(L/F) \le 2g$, we have 
$$
\gf_\gp(L_0/\Q) \le \gf_\gp(L/\Q) \le (4g+1)/(2g+1) < 2
$$ 
by \eqref{condH} and so 2 is unramified in $L_0/\Q$.  Proposition \ref{Fproperties}(iv) shows that each inertia group over $N$ in $F$ is generated by a transposition.  Since $\gq$ is unramified in $L/F$, item (ii) holds.  Then by Kronecker-Weber, $L_0 \subseteq \Q(\sqrt{N^*})$, with equality by Proposition \ref{Fproperties}(iii).   The subfield of $L$ fixed by all the conjugates of $\sigma_\gq$ is unramified outside $\{2,\infty\}$ and is therefore contained in $\Q(i)$ by the Odlyzko bounds and our 2-adic conductor bound (see \cite[Prop.\! 3.1]{BK1}), so equals $\Q$.  This proves (iii).
\end{proof}

\begin{lem} \label{FactorN}
If $\cI_\gq(F/\Q) = \lr{\sigma_\gq}$ at a prime $\gq \vert N$ of residue characteristic $q$, then the action of $\sigma_\gq$ on left cosets of $\Gal(F/K)$ in $\Gal(F/\Q)$ is the product of $2g-1$ transpositions.  Equivalently, $(q)\cO_K = \ga \gb^2$ where $\ga$ and $\gb$ are squarefree, relatively prime ideals of $\cO_K$ with respective absolute norms $q^{2g^2-3g+2}$ and $q^{2g-1}$. 
\end{lem}

\begin{proof}
Let $f$ be the minimal polynomial of an element $\alpha$ satisfying $F_0 = \Q(\alpha)$.  Choose a field $k$ linearly disjoint from $F$, such that the residue degree of each prime over $q$ in $kF/k$ is 1.  By Proposition \ref{Fproperties}(iv),  $\sigma_\gq$ acts as a transposition on the roots of $f$, so there are $r_1 = 2g-1$ roots in $\wtk = F_0k$ and $r_2 = 1$ pair of roots in a ramified quadratic extension over $\wtk$.  The claims now follow from Lemma \ref{r1r2} below.
\end{proof}

Up to isomorphism, the smallest non-trivial irreducible representation of $\cS_{2g+1}$ is given by the hyperplane 
$$
H = \{ (a_0, \dots, a_{2g+1}) \in \F_2^{2g+1} \, \vert \, a_0 + \dots + a_{2g+1} = 0 \}
$$
with the permutation action \cite{Di}.  This representation preserves a symplectic form, so that $\rho_H\!: \, \cS_7 \hookrightarrow \SP_6(\F_2)$.  It takes transpositions to transvections and is absolutely irreducible.  If $r$ is the order of 2 in $(\Z/(2g+1)\Z)^\times$ and $\sigma$ is a $(2g+1)$-cycle in $\cS_{2g+1}$, then there is an element $\tau$ of order $r$ such that
$$
\Delta = \lr{\sigma, \tau\, \vert \, \sigma^{2g+1} = 1, \, \tau^r = 1, \, \tau \sigma \tau^{-1} = \sigma^2}
$$
is a metacyclic group $\lr{\sigma} \rtimes \lr{\tau}$ of order $r(2g+1)$.   The characteristic polynomial of $\sigma$ is the polynomial $\xi(x) = (x^{2g+1}-1)/(x-1)$.  To obtain the irreducible $r$-dimensional  $\F_2[\Delta]$-modules, let $\zeta_j$ range over a set of representatives for the orbits under the action of squaring on primitive $(2g+1)$-th roots of unity and let $E_j$ be induced from the representation of $\lr{\sigma}$ determined by $\sigma \mapsto \zeta_j$.  These representations are permuted by $\Aut(\Mu_{2g+1})$.  The other irreducible representations of $\Delta$ can be obtained similarly, as in \cite[\S47]{CR}.  

Proposition \ref{Fproperties}(ii) indicates that if $F$ is the Galois closure of a favorable field $F_0$, and $\gp \vert 2$, then the inertia group $\cI_\gp(F/\Q) = \lr{\sigma_p}$ is cyclic of order $2g+1$ and the decomposition group $\cD_\gp(F/\Q)$ is isomorphic to $\Delta$.  

\begin{prop} \label{biconnected}
Let  $A$ be a favorable abelian $g$-fold.  Then $\cE=A[2]$ is a simple finite flat group scheme over $\Z[1/N]$.  The restriction $\cE_{\vert \Z_2}$ is biconnected and the simple constituents of $\cE_{\vert \Z_2}$ correspond to the irreducible factors of $\xi(x)$ over $\F_2$.
\end{prop}

\begin{proof}
Since $E$ is a faithful $\F_2[\cS_{2g+1}]$-module of dimension $2g$, it is isomorphic to $H$.  As an $\F_2[\Delta]$-module, $E$ has no one-dimensional constituents, since $\xi(x)$ is the characteristic polynomial for the action of $\sigma$.  Thus, the \'{e}tale and multiplicative components of $\cE_{\vert \Z_2}$ are trivial and so $\cE_{\vert \Z_2}$ is biconnected.  Indeed, the structure of $E$ as an $\F_2[\Delta]$-module determines the constituents of $\cE_{\vert \Z_2}$.
\end{proof}

To obtain a family of favorable abelian $g$-folds, let  $J$ be the Jacobian of the hyperelliptic curve $C\!: \, y^2 + y = f(x)$, with $f$ monic in $\Z[x]$ and $\deg f = 2g+1$.  Then $C$ has good reduction at 2 and so $J[2]$ is a finite flat group scheme over $\Z_2$.  If $\theta$ is a root of $h(x) = 1 + 4f(x)$, then $F = \Q(J[2])$ is the Galois closure of $F_0 = \Q[\theta]$.  The Newton polygon of $h$ shows that there is unique prime over 2 in $F_0/\Q$, totally ramified of degree $2g+1$ and $\ord_2(\theta) = -\frac{2}{2g+1}$.  Hence $h$ is irreducible over $\Q_2$ and  so also over $\Q$. Since 2 is tamely ramified, $\ord_2(\disc_{F_0/\Q}) = 2g$.  If, in addition, the discriminant of $h$ is $\pm 2^{2g} N$, with $N$ odd and square-free, then $F_0$ is a favorable field and  $J$ is a favorable $g$-fold. 

When $N$ is prime, the global group theory associated to extensions of $A[2]$ by itself,  described in \S 6 for $g=3$ and in \cite{BK5} for $g=2$, can be generalized to arbitrary $g$.  In fact, the lattice of  $\F_2[\cS_{2g+1}]$-submodules of $\End(A[2]) \simeq H \otimes H^*$ depends only on the parity of $g$.  The local study over $\Z_2$ is more delicate and depends in a subtle way on $g$.  One needs to confront the issues in \S \ref{Honda} and \S \ref{LocalFieldofPoints} for all the irreducible $\F_2[\Delta]$-modules in $A[2]$ and their extensions.

\begin{prop} \label{Eexists}
Let $F$ be the Galois closure of a favorable heptic field $F_0$ of discriminant $\pm 2^6 N$.  For $g = 3$, let $E \simeq H$ afford the Galois representation 
$$
\rho_E\!: \, \Gal(F/\Q) \xrightarrow{\sim} \cS_7 \xrightarrow{\rho_H} \SP_6(\F_2).
$$
Then $E$ prolongs to a unique finite flat group scheme $\cE$ over $\Z[1/N]$ such that $F = \Q(E)$.  Moreover, $\cE_{\vert \Z_2} \simeq \cE_1 \oplus \cE_2$, where the $\cE_i$'s are the simple group scheme over $\Z_2$ corresponding to the Honda systems $(\rM_i,\rL_i)$ of {\em Proposition \ref{LocalSimple}}.
\end{prop}

\begin{proof}
The decomposition group $\cD_\gp(F/\Q)$ at any prime $\gp$ over $2$ is isomorphic $\Delta$.  Let $E_1$ and $E_2$ be the $\F_2[\Delta]$-modules associated to $\cE_1$ and $\cE_2$, respectively and described in \eqref {stEnd1} and \eqref {stEnd2}.  Then $E \simeq E_1 \oplus E_2$ as an  $\F_2[\Delta]$-module.  Since $\cE_1$ and $\cE_2$ are biconnected, it follows from Fontaine and the Mayer-Vietoris sequence of Schoof \cite[Cor. 2.4]{Sch1} that $E$ prolongs to a unique finite flat group scheme over $\Z[1/N]$.
\end{proof}

\vspace{5 pt}

Finally, we verify the two general number-theoretic lemma cited above.  The first is used in the proof of Proposition \ref{Fproperties}(ii), to interchange completion with Galois closure.   The second is used to understand ramification over $N$ and $\infty$ in the proofs of Lemma \ref{TooReal} and Lemma \ref{FactorN}.

\begin{lem} \label{CompletionClosure}
Let $L_0/k$ be a finite extension of number fields and let $L$ be the Galois closure of $L_0$ over $k$.  Fix a prime $\gp$ of $k$ and let $P$ be a prime over $\gp$ in $L$. Assume that $\gP=P\cap L_0$ is the unique prime over $\gp$ in $L_0$.  Denote the respective completions by $\wtL = L_P$, $\wtL_0 = (L_0)_\gP$ and $\wtk = k_\gp$.  Then $\wtL$ is the Galois closure of $\wtL_0$ over $\wtk$. 
\end{lem}

\begin{proof}
Let $D = D_P(L/k) \simeq \Gal(\wtL/\wtk)$ be the decomposition group of $P$ inside $G = \Gal(L/k)$ and let $H = \Gal(L/L_0)$.  Because $\wtL$ is Galois over $\wtk$ and contains $\wtL_0$, the Galois closure $M$ of $\wtL_0$ is contained in $\wtL$.  The decomposition group $D_P(L/L_0)$ of $P$ inside $H$ is given by $J = D \cap H \simeq \Gal(\wtL/\wtL_0)$.

Since there is only one prime over $\gp$ in $L_0$, the group $H$ acts transitively on the set of primes over $\gp$ in $L$.  Thus, for each $g$ in $G$, there is some $h$ in $H$ such that $g(P) = h(P)$ and so $G=HD = DH$.  Since $L$ is a minimal Galois extension of $k$ containing $L_0$, we conclude that
$$
\{1\} = \bigcap_{g \in G} gHg^{-1} = \bigcap_{d \in D} dHd^{-1} \supseteq \bigcap_{d \in D} d J d^{-1}.
$$
Hence, $M = \wtL$ as desired. 
\end{proof}

\begin{lem} \label{r1r2}
Suppose that $G = \Gal_k(f)$ acts doubly transitively on the roots of a separable polynomial $f$ over a field $k$.   Let $\wtk$ be an extension of $k$ containing exactly $r_1$ of the roots of $f$ and assume that the remaining $2r_2$ roots lie in a quadratic extension of $\wtk$.  For any two distinct roots $\gamma_1,\gamma_2$ of $f$, there are exactly $r_1' = \binom{r_1}{2} + r_2$ embeddings of $k(\gamma_1+\gamma_2)$ into $\wtk$.
\end{lem}

\begin{proof}
Let $\alpha_1, \dots, \alpha_{r_1}$ be the roots of $f$ in $\wtk$ and let $\beta_1,\ov{\beta}_1, \dots, \beta_{r_2},\ov{\beta}_{r_2}$ be the conjugate pairs of roots not in $\wtk$.  The conjugates of $\gamma_1+\gamma_2$ in $\wtk$ have the form $\alpha_i + \alpha_j$ for $1 \le i < j \le r_1$ or $\beta_i + \ov{\beta}_i$ for $1 \le i \le r_2$.  
\end{proof}

\section{Favorable abelian threefolds}  \label{3Folds}  \numberwithin{equation}{section}

Throughout this section, $F$ is the Galois closure of a favorable heptic field of discriminant $\pm 2^6 N$ with $N$ {\em prime}.  Also, $F = \Q(E)$, where $\cE$ is the group scheme over $R = \Z[1/N]$ introduced in Proposition \ref{Eexists}.  Thus $\cE$ is absolutely simple and $\cE_{\vert \Z_2}$ is biconnected. 

Let $\un{E}$ be the category of finite flat group schemes over $R$ defined in \S \ref{Overview}.  Recall from \cite[\S 3]{BK5} that $\un{E}$ is a full subcategory of the category of $2$-primary group schemes over $R$, closed under taking products, closed flat subgroup schemes and quotients by closed flat subgroup schemes.   In particular, $\Ext_{\un{E}}^1$ is defined.  If $\cV$ and $\cW$ are in $\un{E}$ and annihilated by $2$, let $\Ext_{[2], \un{E}}^1(\cV,\cW)$ denote the subgroup of $\Ext_{\un{E}}^1(\cV,\cW)$ whose classes are represented by extensions killed by $2$.  

When $A$ is a favorable threefold, our hypotheses on $\cE$ are satisfied by $\cE = A[2]$, according to Proposition \ref{biconnected}.  As a consequence of \cite[Theorem 3.7]{BK5}, the following uniqueness criterion holds.

\begin{crit} \label{Bcrit}
Assume that $\Ext^1_{[2],\un{E}}(\cE,\cE) = 0$.  Let $B$ be a semistable abelian variety of conductor $N^d$ such that $B[2]^{ss}\simeq d \cE$. Then $B$ is isogenous to $ A^d$.
\end{crit} 

The aim of this section is to obtain explicit conditions on the field $F$ to guarantee the vanishing of $\Ext^1_{[2],\un{E}}(\cE,\cE)$.  When there is an extension $V$ of $E$ by $E$ as Galois modules, it is essential for us to compare the restriction of $V$ to a decomposition group $\cD_\gp(\Q(V)/\Q)$ at a prime $\gp \vert 2$ with the Galois modules arising from Honda systems in \S \ref{LocalCorners}.  

\vspace{5 pt}

We use a fixed but non-standard basis for $E$ to define a map $\iota\!: \, \cS_7 \to \SP_6(\F_2)$, as follows.  Write the following $6 \times 6$ matrices as blocks of $3 \times 3$ matrices, using the notation for the local representations in \eqref{stEnd1} and \eqref{stEnd2}:  
\begin{equation} \label{bfst}
\begin{array}{l l l l}  
\bfs =  \left[\begin{matrix} s_1 & 0 \\ 0 & s_2 \end{matrix}\right]  \quad \text{and}  \quad \bft =  \left[\begin{matrix}t &0 \\ 0 & t \end{matrix}\right]. 
\end{array}
\end{equation}
Fix $\iota$ by demanding that $(1234567) \mapsto \bfs$ and $(12) \mapsto \bfr$, where $\bfr$ is the transvection
\begin{equation}  \label{bfr}
\begin{array}{l l l l}  
\bfr  = I + \left[\begin{smallmatrix}  0 & 1 & 1 & 0 & 1 & 1 \end{smallmatrix} \right]^T \left[\begin{smallmatrix} 0 & 0 & 1 & 0 & 1 & 0 \end{smallmatrix} \right]  
\end{array}
\end{equation}
Then $\iota\!: \, (124)(365) \mapsto \bft$.  There is a prime $\gp \vert 2$ in $F$ such that $\rho_E(\sigma_\gp) = \bfs$ and $\rho_E(\tau_\gp) = \bft$, where $\sigma_\gp$ generates the inertia group $\cI_\gp(F/\Q)$ and $\tau_\gp$ is a Frobenius.   Moreover, there is a prime $\gq$ over $ N$ in $F$ such that the generator $\sigma_\gq$ of the inertia group $\cI_\gq(F/\Q)$, satisfies $\rho_E(\sigma_\gq) = \bfr$.  From now on, $\rho_E$ denotes the representation 
\begin{equation} \label{Erep}
\rho_E\!: \, \Gal(F/\Q) \simeq \cS_7 \xrightarrow{\iota} \SP_6(\F_2),
\end{equation} 
with this normalization.

An extension $V$ of $E$ by $E$ of exponent 2 and field of points $L = \Q(V)$ provides a representation 
\begin{equation} \label{rhoV}
\rho_V\!:  G = \Gal(L/\Q) \to \cP = \left\{ \left[\begin{smallmatrix}  x & y \\ 0 & x \end{smallmatrix} \right] \, \vert \, x \in \iota(\cS_7), \, y \in M_6(\F_2) \right\}, 
\end{equation}
where $\cP = \cP_{E,E}$ is the parabolic group described in \eqref{defP}.  The field $L$ contains $F$ and  $\Gal(L/F)$ is an elementary 2-group.  Moreover, $L$ does not depend on the choice of representative for the extension class $[V]$ in $\Ext^1_{[2],\Q}(E,E)$.  We always arrange for $\pi \rho_V = \rho_E$, where $\pi\!: \, \cP \twoheadrightarrow \iota(\cS_7)$ is the projection map $ \left[\begin{smallmatrix}  x & y \\ 0 & x \end{smallmatrix} \right] \mapsto x$.  The radical of $\cP$ is $\cR = \ker \pi$.   If, in addition, $V$ is the Galois representation associated to a class $[\cV]$ in $\Ext^1_{[2], \un{E}}(\cE,\cE)$, we shall see in Proposition \ref{GaProp} that $\rho_V(G)$ is conjugate in $\cP$ to one of the groups that we now define. 

Set $\gamma_0 = 0$ in $M_6(\F_2)$ and define the following matrices $\gamma_a$ in $M_6(\F_2)$, writing them as blocks of $3 \times 3$ matrices in terms of the matrix $t$ from \eqref{stEnd1}:
\begin{equation} \label{CorGa}
\begin{array}{ c || c | c | c | c | c }
a & 6 & 14 & 20 & 29 & 35  \\
\hline
\gamma_a &   \left[\begin{smallmatrix}  t & t^2+1 \\  t^2 & t^2 \end{smallmatrix}\right]  &   
 \left[\begin{smallmatrix}  0 & t^2+1 \\  0 & 0 \end{smallmatrix}\right]& \gamma_6 + \gamma_{14}  &  \left[\begin{smallmatrix}   0 & 1 \\  t^2 & 0 \end{smallmatrix}\right]  & \gamma_6 + \gamma_{29}  
\end{array}
\end{equation}
The group ring $S = \F_2[\cS_7]$ acts on $M_6(\F_2)$ via conjugation by the matrices in $\iota(\cS_7)$.  

\begin{Not} \label{Ca} 
Let $c\!: \, M_6(\F_2) \to \cP$ by $c(m)=\left[\begin{smallmatrix}1&m\\0&1\end{smallmatrix}\right]$ and  $d\!: \, \SP_6(\F_2) \to \cP$ by $d(m)=\left[\begin{smallmatrix} m & 0 \\ 0 & m\end{smallmatrix}\right]$.   Let $G_0 \simeq \cS_7$ be the image under $d$ of $\iota(\cS_7)$. 
For $\gamma_a$ in \eqref{CorGa} define the following ${\rm S}$-submodules of $M_6(\F_2)$:
\begin{equation} \label{Gammas}
\Gamma_6 = {\rm S} \, \gamma_6 \quad \Gamma_{14} = {\rm S} \, \gamma_{14}, \quad \Gamma_{20} = {\rm S} \, \gamma_{20}, \quad \Gamma_{29} = {\rm S} \, \gamma_{29}, \quad \Gamma_{35} = {\rm S} \,\gamma_{35}
\end{equation} 
and let $G_a = \,  <G_0,c(\gamma_a)> \, =\,  c(\Gamma_a) \rtimes G_0$. 
\end{Not}

The radical of $G_a$ is $\cR_a = c(\Gamma_a)$ and has size $ 2^a$.  The abelianization of $G_a$ has order 2.  For $a$ in $\{29,35\}$, the center of $G_a$ is generated by $c(1)$, but the center is trivial for the other values of $a$.  Inclusions among the groups $G_a$  are indicated in the Hasse diagram below by ascending lines: 
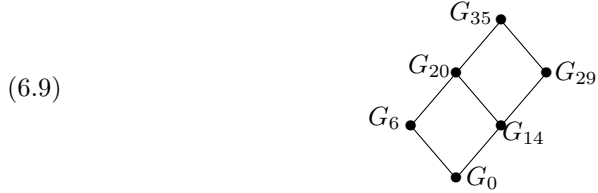
\begin{equation}  \label{GaDiagram}
\begin{tikzpicture}[scale=1, thin, baseline=(current  bounding  box.center)]	
       \draw (0,1.4) node {$\bullet$};  \draw (-.6,.7) node {$\bullet$}; \draw (.6,.7) node {$\bullet$};
       \draw (-1.2,0) node {$\bullet$}; \draw (0,0) node {$\bullet$};
       \draw (-.6,-.7) node {$\bullet$};
       \draw (-1.2,0) -- (-.6,.7)  -- (0,1.4);  \draw (.6,.7) -- (0,0)  -- (-.6,.7);  \draw (.6,.7)  -- (0,1.4);
       \draw  (-1.2,0) -- (-.6,-.7) -- (0,0); 
       \draw (-0.4,1.5) node {$G_{35}$};  \draw (-.95,.8) node {$G_{20}$}; \draw (1,.7) node {$G_{29}$}; \draw (0.3,-.1) node {$G_{14}$};  \draw (-1.55,.1) node {$G_6$};
       \draw (-.25,-.7) node {$G_0$};
\end{tikzpicture}
\end{equation}
Moreover, $G_{20}$ is isomorphic to the fiber product of $G_{14}$ and $G_6$ over $G_0$ and similarly for the other parallelograms.    The identity on $G_0$ extends to a surjection $f_{a,b}\!: G_a \twoheadrightarrow G_b$ sending $c(\gamma_a)$  to  $c(\gamma_b)$. 

\begin{Def} 
An involution $h$ is {\em good} in a group $H$ if the normal closure of $\lr{h}$ in $H$ equals $H$.  If $h$ is good in a subgroup $H$ of the parabolic group $\cP$ and $\rk \, (h-1) = 2$, then $h$ is {\em very good} in $H$.  
\end{Def}

\begin{Rem} \label{OneVG}
A Magma verification shows that  each $G_a$  has a unique conjugacy class of very good involutions, represented by $d(\bfr)$ with $\bfr$ given in \eqref{bfr}. 
\end{Rem}

Let $[V]$ be a class in $\Ext^1_{[2],\Q}(E,E)$, with $L = \Q(V)$ and $F = \Q(E)$.  We shall impose the following assumptions, which do not depend on the choice of representative for the class $[V]$. 
\begin{enumerate} [{\bf \hspace{3 pt} G1.}]
\item The extension $L/F$ is unramified outside $\{2,\infty\}$.   \vspace{2 pt}
\item  If $\cI_{\gq}(L/\Q) = \lr{\sigma_\gq}$ at a prime $\gq \vert N$ and $h = \rho_V(\sigma_\gq)$, then $\rk \, (h - 1) = 2$.   \vspace{2 pt}
\item At primes $\gp \vert 2$, the ray class conductor exponent $\gf_\gp(L/F) \le 6$.
\end{enumerate}

\begin{prop}
If $\cV$ represents a class in $\Ext^1_{[2], \un{E}}(\cE,\cE)$, then the associated Galois module $V$ satisfies {\bf G1}--{\bf G3}.  
\end{prop} 

\begin{proof}
By  {\bf E2},  $\sigma_\gq$ has order 2 in $\Gal(L/\Q)$, so all the ramification at primes over $N$ has already occurred in $F/\Q$ and thus {\bf G1} holds for $L = \Q(V)$.  Since $\sigma_\gq$ acts tamely on $V$, the conductor exponent of $V$ at $N$ is given by 
$$
\gf_N(V) = \dim_{\F_2} (V/V^{\lr{\sigma_\gq}}) = \rk \,( \rho_V(\sigma_\gq)-1).
$$ 
But $\gf_N(V) = 2$ by {\bf E3}, so {\bf G2} holds.  Finally, {\bf G3} is the conductor bound for $\cV_{\vert \Z_2}$ given by Theorem \ref{CondBd}.
\end{proof} 

\begin{prop}  \label{GaProp}
If $[V]$ satisfies {\bf G1}--{\bf G3}, then $\rho_V(\Gal(\Q(V)/\Q))$ is conjugate in $\cP$ to some $G_a$.  Given a prime $\gq \vert N$ there is a representative $V$ such that 

\vspace{2 pt} 

\centerline{{\rm (i)} \,  $\rho_V(\Gal(\Q(V)/\Q)) = G_a$ \quad and \quad {\rm (ii)} \, $\rho_V(\sigma_\gq) = d(\bfr)$.}

\vspace{2 pt}

\noindent In particular, $a = 0$ if and only if $V$ is a split extension.  If $V'$ is another representative for $[V]$ satisfying {\rm (i)} and {\rm (ii)}, then $\rho_{V'} = \alpha \rho_V\alpha^{-1}$ for some $\alpha$ in the centralizer of $d(\bfr)$ in $G_a$.  
\end{prop}

\begin{proof}
{\bf G1}--{\bf G3} translate to the following properties of $\ov{G} = \rho_V(\Gal(\Q(V)/\Q))$:
 \begin{enumerate} [\hspace{3 pt} a)]
\item the projection $\pi\!: \, \ov{G} \twoheadrightarrow \iota(\cS_7)$ is surjective;   \vspace{2 pt}
\item $\ov{G}$ is the normal closure of an involution $h$ with  $\rk \, (h-1)=2$;  \vspace{2 pt}
\item the abelianization of $\ov{G}$ has order 2.
\end{enumerate}
Item (a) holds because $\pi \rho_V = \rho_E$.  Item (b) holds for $h = \rho_V(\sigma_\gq)$ by ${\bf G2}$ and Proposition \ref{LProperties}, which also gives (c).   One verifies that any subgroup of $\cP$ satisfying (a)--(c) is conjugate by some $\alpha$ in the radical $\cR$ of $\cP$ to $G_a$ for some $a$.  By Lemma \ref{CocycleRep}, replacing $\rho_V$ by $\rho_W = \alpha \rho_V \alpha^{-1}$ simply changes $V$ to another representative $W$ for the same extension class, now satisfying $\rho_W(G) = G_a$.  In particular, if $a=0$, then $W$ is a split extension.  Since $\rho_W(\sigma_\gq)$ is a very good involution, it is conjugate in $G_a$ to $d(\bfr)$ by Remark \ref{OneVG}.  Applying this inner automorphism to $G_a$ gives  $\rho_W(\sigma_\gq) = d(\bfr)$. 
\end{proof}

\begin{Def}   \label{GaClass}
A $G_a$-class is a class $[V]$ in $\Ext^1_{[2],\Q}(E,E)$ satisfying {\bf G1}--{\bf G3}.
\end{Def}

Fix a pair-resolvent $K$ for $F = \Q(E)$ by  taking $\Gal(F/K)$ to be the centralizer of the transposition $(12)$.  For primes $\gp \vert 2$, the ramification index $e_\gp(K/\Q) = 7$ by Lemma \ref{ramK}.   In addition, any choice of Frobenius at $\gp$ is the product of two 3-cycles in $\cS_7$ so it is not in $\Gal(F/K)$.  Hence, as a special case of Remark \ref{OnePrime},  there is one prime $\gp_K$ over 2 in $K$ and it has residue degree 3.  Let $\Omega_K$ be the maximal elementary 2-extension of $K$ of modulus $\gp_K^6 \, \infty$, where $\infty$ means that ramification is allowed at each archimedean place.   Let $K'$ be a subfield of $\Omega_K$ quadratic over $K$.  Refer to the  diagram of fields and Galois groups below:

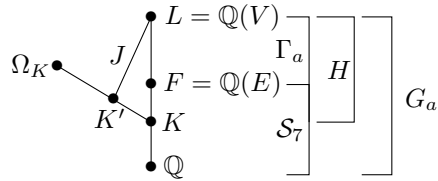
\begin{center}
\begin{tikzpicture}[scale=1, thin, baseline=(current  bounding  box.center)]	
    \draw (0,0) -- (0,2);
    \draw (0,0) node {$\bullet$};   \draw (.3,0) node {$\Q$};
    \draw (0,.6) node {$\bullet$};     \draw (.3,.6) node {$K$};
    \draw (0,1.1) node {$\bullet$};   \draw (.95,1.1) node {$F = \Q(E)$};
    \draw (0,2) node {$\bullet$};   \draw (.95,2) node {$L = \Q(V)$};
    \draw (0,.6) -- (-.5,.9) -- (0,2);    \draw (-.5,.9) node {$\bullet$};     \draw (-.55,.65) node {$K'$}; 
    \draw (-.5,.9) -- (-1.25,1.35); \draw (-1.25,1.35) node {$\bullet$};   \draw (-1.6,1.35) node {$\Omega_K$};
    \draw (2.2,2) -- (2.7,2) -- (2.7,.6) -- (2.2,.6);   \draw (2.5,1.3) node {$H$};
    \draw (1.8,2) -- (2.1,2) -- (2.1,.6);   \draw (1.85,1.55) node {$\Gamma_a$}; 
    \draw (1.8,1.1) -- (2.1,1.1) -- (2.1,-.1) -- (1.8,-.1);   \draw (1.85,.5) node {$\cS_7$};
     \draw (2.8,2) -- (3.2,2) -- (3.2,-.1) -- (2.8,-.1);   \draw (3.6,.9) node {$G_a$}; 
    \draw (-.45,1.5) node {$J$};      
\end{tikzpicture}
\end{center}

\begin{prop}   \label{MakeGaFields} 
Let $K$ be the pair-resolvent of $F$.  There is a bijection 
\begin{equation}  \label{corresp}
\{G_a\text{-}{\rm classes} \hspace{3 pt}  [V] \hspace{5 pt} \vert \hspace{4 pt} a \in \{6, 14, 20\}  \} \longleftrightarrow \{{\rm subfields} \hspace{3 pt} K' \subseteq \Omega_K \hspace{2 pt} \vert \hspace{5 pt} \fdeg{K'}{K}=2\}  
\end{equation}
such that $\Q(V)$ is the Galois closure of $K'/\Q$.
\end{prop}

\begin{proof}
Fix $a$ in $\{6,14,20\}$ and let $H = \pi^{-1}(C)$ be the inverse image under the projection map $\pi\!: \, G_a \twoheadrightarrow \iota(\cS_7)$ of the centralizer $C$ of $\bfr = \iota((12))$.  Among the four subgroups of index 2 in $H$, exactly one, say $J$, enjoys both of the following properties: (i) the action of $G_a$ on $G_a/J$ is faithful and (ii) $d(\bfr)$ has exactly 22 fixed points in this action. 

Given a $G_a$-class $[V]$, fix a prime $\gq \vert N$ in $L = \Q(V)$ and choose a representative $V$ satisfying $\rho_V(\Gal(L/\Q)) = G_a$ and $\rho_V(\sigma_\gq) = d(\bfr)$, as in Proposition \ref{GaProp}.  The subfield of of $L$ fixed by $\rho_V^{-1}(H)$ is the pair-resolvent field $K$ for $F/\Q$ as in Notation \ref{PairResDef} independent of the choice of representative $V$.  The fixed field $K'$ of $\rho_V^{-1}(J)$ is quadratic over $K$ and also is independent of the choice of $V$.  Indeed, if $[V'] = [V]$, then $\rho_{V'}= \alpha \rho_V \alpha^{-1}$ for some $\alpha$ in $G_a$ centralizing $d(\bfr)$ by Proposition \ref{GaProp}.  But then $\alpha$ is in $H$, so $\alpha$ normalizes $J$.  Hence $\rho_{V'}^{-1}(J) = \rho_V^{-1}(J)$.

By (ii) and the factorization of the ideal $(N)\cO_K$ given by Lemma \ref{FactorN} all primes over $N$ in $K$ are unramified in $K'/K$. Since the primes $\gp$ over $2$ are unramified in $F/K$ by Lemma \ref{ramK}, we have $\gf_\gp(K'/K) = \gf_\gp(L/F)$, for example by \cite[Lemma C.11]{BK5}.  Hence {\bf G3} implies that $\gf_\gp(K'/K) \le 6$, so $K'$ is contained in $\Omega_K$ and we have a well-defined map $[V] \mapsto K'$ in \eqref{corresp}.

Conversely, let $K'$ be a subfield of $\Omega_K$ quadratic over $K$, $L$ the Galois closure of $K'/\Q$, $G = \Gal(L/\Q)$, $H = \Gal(L/K)$ and $J = \Gal(L/K')$.  Then $L$ properly contains $F$, since each quadratic extension of $K$ in $F$ does ramify at some prime over $N$ by Lemma \ref{RamOverN}.  By Proposition \ref{LProperties}, the generator $\sigma_\gq$ of inertia $\cI_\gq(L/\Q)$ at each prime $\gq \vert N$ is a good involution in $G$ and $\sigma_\gq$ can be chosen in $H$.  Since no prime over $N$ ramifies in $K'/K$, the action of $\sigma_\gq$ on $G/J$ has 22 fixed points.   The following group-theoretic properties of $G$ have been established:
\begin{enumerate}[\hspace{3 pt} i)]
\item There is  a surjection $\pi\!: G \twoheadrightarrow \cS_7$ whose kernel  has exponent 2 and is the radical of $G$.  \vspace{2 pt}

\item The abelianization of $G$ has order 2.  \vspace{2 pt}

\item If $H$ is the inverse image under $\pi$ of the centralizer of a transposition in $\cS_7$, then there is a subgroup $J$ of index 2 in $H$ such that the action of $G$ on $G/J$ is faithful, so $G$ is a transitive permutation group on 42 letters.    \vspace{2 pt} 
\item There is a good involution $g$ in $G$ whose action on $G/J$ has 22 fixed points. 
\end{enumerate}
We have (i), since the radical of $\Gal(L/\Q)$ is $\Gal(L/F)$ and (ii) by Lemma \ref{LProperties}(i).  For (iii), note that  $\fdeg{G}{J} = 42$ and $G$ acts faithfully on $G/J$ because $L$ is the Galois closure of $K'/\Q$.  In addition, $H = \pi^{-1}(\Gal(F/K))$, where $\Gal(F/K)$ is the centralizer of $(12)$ in $\cS_7$.

In the Magma database of 9491 transitive groups of degree 42, items (i)--(iv) hold for exactly three permutation groups up to conjugation in $\cS_{42}$.   They are isomorphic to $G_6$, $G_{14}$ and $G_{20}$, respectively, via an isomorphism sending $g$ in (iv) to a very good involution in $G_a$, conjugate to $d(\bfr)$ by Remark \ref{OneVG}.  All automorphisms of $G_6$ are inner.  For $a = 14$ or 20, the outer automorphism group of $G_a$ has order 2 and non-trivial outer automorphisms move $d(\bfr)$ to good involutions $j$ such that $\rk (j-1) = 6$.

Let $\rho_W\!: \, G \to G_a$ be the representation given by an isomorphism that preserves the conjugacy class of $d(\bfr)$.  Then $\pi \rho_W$ and $\rho_E$ are conjugate in $\iota(\cS_7)$ because the automorphisms of $\cS_7$ are inner.  By surjectivity of $\pi$, there is a representation $\rho_V\!: \, G \to G_a$, conjugate to $\rho_W$, such that $\pi \rho_V = \rho_E$.  In addition, $\rho_V(\sigma_\gq)$ is conjugate to $d(\bfr)$ in $G_a$, so {\bf G2} holds.   Since $K'$ is contained in $\Omega_K$, we have {\bf G1} and {\bf G3}, making $[V]$ a $G_a$-class.

Suppose that $\rho_{V'}\!: \, G \to G_a$ is another representation such that $\pi \rho_{V'} = \rho_E$ and $\rho_{V'}(\sigma_\gq)$ is conjugate to $d(\bfr)$ in $G_a$.    Since the only automorphisms of $G_a$ preserving the conjugacy class of $d(\bfr)$ are inner, we have $\rho_{V'} = \alpha \rho_V \alpha^{-1}$ for some $\alpha$ in $G_a$.  But then $\pi(\alpha) = 1$ because the center of $\cS_7$ is trivial, so $\alpha$ is in the radical $\cR_a$ of $G_a$.  Hence $[V'] = [V]$ and we therefore have a well-defined inverse map from $K'$ to the extension class $[V]$.  
\end{proof}

From now on, $[V]$ is a $G_a$-class and $L = \Q(V)$. In particular, $V$ represents a class in $\Ext^1_{R'}(\cE,\cE)$, where $R' = \Z[\frac{1}{2N}]$ and the conductor bound {\bf G3} holds.  The Mayer-Vietoris sequence of Schoof \cite[Cor. 2.4]{Sch1} with $R = \Z[1/N]$ implies that
$$
\Ext^1_R(\cE,\cE) \to  \Ext^1_{R'}(E,E) \times \Ext^1_{\Z_2}(\cE,\cE) \to \Ext^1_{\Q_2}(E,E)
$$
is exact.  Thus $V$ prolongs to a group scheme $\cV$ over $R$ if and only if there is some class $[\cW]$ in $\Ext^1_{\Z_2}(\cE,\cE)$, as determined in \S\ref{LocalCorners}, such that the images of $[V]$ and $[\cW]$ in $\Ext^1_{\Q_2}(E,E)$ agree.  If so, the other properties of a $G_a$-class guarantee that $[\cV]$ is in $\Ext^1_{[2],\un{E}}(\cE,\cE)$.   Next, we sketch our strategy for examining this patching problem and review the notation required.

Recall that $\gh = \Hom_{\F_2}(E,E)$ and let the cocycle $\psi\!: \, G_\Q \to \gh$ represent the cohomology class in $H^1(G_\Q,\gh)$ associated to the $G_a$-class $[V]$, as in \eqref{RepFromPsi}.  Let $L_\gp$ and $F_\gp$ be the completions of $L$ and $F$ at a fixed prime $\gp \vert 2$.  The decomposition group $\cD_\gp = \Gal(L_\gp/\Q_2)$ admits the quotient $\Delta = \Gal(F_\gp/\Q_2) \simeq \lr{\bfs,\bft}$, where $\bfs$ and $\bft$ are defined in \eqref{bfst}.  By {\bf G3}, $L_\gp$ is contained in the maximal elementary 2-extension $T$ of $F_\gp$ with ray class conductor exponent at most 6.  Let $\res_{\cD_\gp} [\psi]$ denote the restriction to the decomposition group.

In the local theory of \S \ref{LocalCorners}, a class $[\cW]$ in $\Ext^1_{[2],\Z_2}(\cE,\cE)$ gives rise to a  {\em Honda character} $\chi_\cW$ in $\Hom_\rR(\Gamma,\gh)$, where $\Gamma = \Gal(T/F_\gp)$ and $\rR = \F_2[\Delta]$.   As noted before Definition \ref{HondaChar}, $\res_\Gamma$ and $\inf_{L_\gp}^T$ are injective.   Let 
$$
\psi_\gp = \res_\Gamma \, {\rm inf}_{L_\gp}^T \, \res_{\cD_\gp}([\psi]).
$$
It follows that $V$ prolongs to a group scheme over $\Z[1/N]$ if and only if $\psi_\gp = \chi_\cW$ in $\Hom_\rR(\Gamma,\gh)$.  To test whether this last equality holds, it suffices to compare $\psi_\gp$ and $\chi_\cW$ on the generators $g_0, \dots, g_3$ for $\Gamma$ defined before Proposition \ref{Max2}. 

See Definition \ref{Cors} for corner spaces.  Since $g_1$ is in $C_1(\Gamma)$, but $g_2$ and $g_3$ are in $C_2(\Gamma)$, the values of $\chi_\cW$ on these generators lie in the corresponding corner spaces of $\gh$.   Let $E_2$, $E_1$ and $D$ be the irreducible $\F_2[\cD_\gp]$-modules from \S \ref{LocalCorners}, with respective $\F_2$-dimensions $3,3,2$. 

\begin{lem} \label{CornerGroups} 
The $\F_2[\cD_\gp]$-module structure of the the groups $\Gamma_a$ and generators for their corner spaces are as follows.
\begin{enumerate}[$\bullet$]
\item $C_2(\Gamma_{6}) = \lr{\gamma_6}$, $C_1(\Gamma_{6}) = \lr{\ov{\gamma}_6}$ and $\Gamma_6 \simeq E_1 \oplus E_2$, where 
\begin{equation*} 
 \gamma_6 = \left[\begin{matrix}  t & t^2+1 \\  t^2 & t^2 \end{matrix}\right] \quad \text{and} \quad  \ov{\gamma}_6 =  \left[\begin{matrix} t^2 &t \\  t+1 & t \end{matrix}\right]. 
\end{equation*}

\vspace{2 pt}

\item $C_2(\Gamma_{14}) = \lr{\gamma_{14}, \gamma_{14}'}$, $C_1(\Gamma_{14}) = \lr{\ov{\gamma}_{14}, \ov{\gamma}_{14}'}$ and $\Gamma_{14} \simeq E_1^2 \oplus E_2^2 \oplus D$, where
\begin{equation*} 
\gamma_{14} = \left[\begin{matrix} 0 & t^2+1 \\  0 & 0 \end{matrix}\right], \quad
\gamma_{14}' = \left[\begin{matrix} t & 0 \\  0 & t ^2\end{matrix}\right], \quad
\ov{\gamma}_{14} = \left[\begin{matrix} 0 & 0 \\  t+1 & 0 \end{matrix}\right], \quad 
\ov{\gamma}_{14}' = \left[\begin{matrix} t^2  & 0 \\  0 & t \end{matrix}\right].
\end{equation*}

\vspace{2 pt}

\item  $C_2(\Gamma_{29}) = \lr{\gamma_{29}, \gamma_{29}', \gamma_{14}, \gamma_{14}'}$ and $C_1(\Gamma_{29}) = \lr{\ov{\gamma}_{29}, \ov{\gamma}_{29}', \ov{\gamma}_{14}, \ov{\gamma}_{14}'}$, where
\begin{equation*} 
\gamma_{29} = \left[\begin{matrix} 0 & 1 \\  t^2 & 0 \end{matrix}\right],  \quad 
\gamma_{29}' = \left[\begin{matrix} t  & 0 \\  t^2 & 0 \end{matrix}\right], \quad 
\ov{\gamma}_{29} = \left[\begin{matrix} 0 & t \\ 1  & 0 \end{matrix}\right],  \quad 
\ov{\gamma}_{29}' = \left[\begin{matrix} t^2 & t \\ 0  & 0 \end{matrix}\right]
\end{equation*}
and $\Gamma_{29} \simeq E_1^4 \oplus E_2^4 \oplus D^2 \oplus \F_2$.  
\end{enumerate}
Moreover, $\Gamma_{20} = \Gamma_6 \oplus \Gamma_{14}$ and $\Gamma_{35} = \Gamma_6 \oplus \Gamma_{29}$, with corresponding sums of corners.
\end{lem}
The proof consists of a Magma verification.\qed

\begin{prop}  \label{Cond4}
Let $V$ represent a $G_a$-class and $L_a = \Q(V)$.  Let $\gf_\gp(L_a/F)$ be the ray class conductor exponent of the extension $\wtL/\wtF$, where $\wtL$ and $\wtF$ are the completions of $L_a$ and $F$ at a place $\gp$ over $2$.  If $V$ prolongs to a group scheme over $\Z[1/N]$, then $\gf_\gp(L_a/F) \le 6$.  In addition, the following stricter bounds must hold:
$$
\gf_\gp(L_6/F)  \le 2, \hspace{40 pt}  \gf_\gp(L_{14}/F)  \le 4, \hspace{40 pt}  \gf_\gp(L_{29}/F)  \le 4.
$$
\end{prop}

\begin{proof}
We use some of the notation recalled above.  If $\psi_\gp(g_3)$ is not trivial and matches $\chi_\cW(g_3)$ for some Honda cocycle $\chi_\cW$, then $\psi_\gp(g_3) = \left[ \begin{smallmatrix} 0 & 0 \\ t^2 & 0 \end{smallmatrix} \right]$ by the local theory in Lemma \ref{LocalChi}.  But this element is not in any of the global corner groups $C_2(\Gamma_a)$ for $a = 6$, 14 or 29, according to Lemma \ref{CornerGroups}.  Hence $\gf_\gp(L_a/F) \le 4$ in those cases.  Suppose further that $a=6$.  If $\psi_\gp(g_2)$ matches $\chi_\cW(g_2)$ for some Honda cocycle $\chi_\cW$, then $\psi_\gp(g_2)$ has the form $\left[ \begin{smallmatrix} * & 0 \\ * & * \end{smallmatrix} \right]$, also by Lemma \ref{LocalChi}.  But $C_2(\Gamma_6) = \lr{\gamma_2}$ by Lemma \ref{CornerGroups}, forcing $\psi_\gp(g_2) = 0$.  Hence $\gf_\gp(L_6/F) \le 2$.
\end{proof}

We next recall the concept of an amiable field from \S1.  As is standard for this section, $F$ is the Galois closure of a favorable heptic field $F_0$ with $\disc_{F_0/\Q}=\pm 2^{6} N$ and $N$ is prime.  Also, $K$ is a pair-resolvent for $F$ and $\gp$ is the prime over 2 in $K$.

\begin{Def}\label{amiable}
Write $\Omega_K$ for the maximal elementary 2-extension of $K$ unramified outside $\{2,\infty\}$, such that the ray class conductor exponent $\gf_\gp(\Omega_K/K) \le 6$.  We say that $F$ is {\em amiable} if either $\Omega_K = K$ or one of the following conditions holds:  
\begin{enumerate} [{\rm i)}] 
\item $\fdeg{\Omega_K}{K}=2$ and $\gf_\gp(\Omega_K/K) =6$, or \vspace{2 pt}
\item $\fdeg{\Omega_K}{K}=2$, $\gf_\gp(\Omega_K/K) =4$ and the Galois closure of $\Omega_K/\Q$ has group $G_6$.
\end{enumerate}
\end{Def}

\begin{prop}  \label{AmiNec}
For $F$ to be amiable, both of the following conditions must hold:
\begin{enumerate}[{\rm i)}]
\item the narrow class number of $K$ is odd and  \vspace{2 pt}

\item the number of real embeddings of $F_0$ satisfies $r_1 \le 3$.  \end{enumerate}
\end{prop}

\begin{proof}
If (i) fails, then $\Omega_K$ contains a quadratic extension of $K$ whose conductor exponent is 0, so $F$ is not amiable.  If (ii) fails, then $r_1 \ge 5$, so $\dim_{\F_2} \Gal(\Omega_K/K) \ge 4$ by Lemma \ref{TooReal}, also contradicting amiability.
\end{proof}

\begin{lem} \label{ab}
Let $[V]$ be a $G_a$-class and let $[V']$ be a $G_{a'}$-class.  
\begin{enumerate}[{\rm i)}]
\item The Baer sum  $[V''] = [V] + [V']$ is a $G_b$-class for some $b$.   \vspace{2 pt}

\item If there is a surjection $f_{a,b} \!: G_a \twoheadrightarrow G_b$ in diagram \eqref{GaDiagram}, then the Galois module for $f_{a,b} \, \rho_V$ represents a $G_b$-class. 
\end{enumerate}
\end{lem}

\begin{proof} 
If $L = \Q(V)$ and $L' = \Q(V')$, then $L'' = \Q(V'')$ is a subfield of the compositum $LL'$.   Conditions {\bf G1} and {\bf G3} hold for ramification in $L''/F$ because they are valid for $L$ and $L'$.  By Proposition \ref{GaProp}, we can choose the representative $V$ and a prime $\gq$ over $N$ in $LL'$ such that $\rho_V(\sigma_\gq) = d(\bfr)$.  Let the cocycles $\psi$ and $\psi'$ belong to $V$ and $V'$, as in \eqref{RepFromPsi}.  Then $\psi'' = \psi + \psi'$ represents the  cocycle class for $V''$.   Since $\psi(\sigma_\gq) = 0$, we have $\psi''(\sigma_\gq) = \psi'(\sigma_\gq)$.  Conclude that {\bf G2} holds for $[V'']$ by using $\rho_{V''}(\sigma_\gq)$. 

For part (ii), let $L'$ be the subfield of $L$ fixed by $\rho_V^{-1}(\ker f_{a,b})$.  Then $L'$ contains $F$ and $f_{a,b} \, \rho_V$ induces an isomorphism $\rho_{V'}\!: \, \Gal(L'/\Q) \to G_b$, where $V'$ is the corresponding Galois module.  The ramification conditions in {\bf G1} and {\bf G3} hold for subfields of $L$ containing $F$.    Since $f_{a,b}$ is the identity on $G_0$ we find that $\rho_{V'}(\sigma_\gq) = \rho_V(\sigma_\gq)$.  Hence {\bf G2} also holds for $V'$.  
\end{proof}

\begin{prop}\label{ext2} 
If $F=\Q(E)$ is amiable, then $\Ext^1_{[2],\un{E}}(\cE,\cE)=0$. 
\end{prop}

\begin{proof} 
If $K$ is a pair-resolvent field of $F$, then $\fdeg{\Omega_K}{K} \le 2$ by definition of an amiable field.  Suppose that there is a non-trivial class $[\cV]$ in $\Ext^1_{[2],\un{E}}(\cE,\cE)$.  By Proposition  \ref{GaProp}, $[V]$ is $G_a$-class with $a \ne 0$ and $L = \Q(V)$ properly contains $F$.

If $a=20$ or 35, we have both a $G_6$-class $[V']$ and a $G_{14}$-class $[V'']$, according to diagram \eqref{GaDiagram} and Lemma \ref{ab}(ii).  Then Proposition \ref{MakeGaFields} shows that there are two independent quadratic extensions of $K$ in $\Omega_K$ corresponding to $V'$ and $V''$, which contradicts $\fdeg{\Omega_K}{K} \le 2$.

If $a = 6$ or 14, Proposition \ref{MakeGaFields} provides a quadratic extension $K'$ of $K$ in $\Omega_K$ corresponding to $V$ and so $K' = \Omega_K$.   By Proposition \ref{Cond4}, $\gf_\gp(K'/K) \le 2$ if $a = 6$ and $\gf_\gp(K'/K) \le 4$ if $a = 14$.  Either case violates amiability.  

If $a = 29$, then $\gf_\gp(L/F) \le 4$ by Proposition \ref{Cond4}.  Use \eqref{GaDiagram} and Lemma \ref{ab}(ii) to obtain a $G_{14}$ class $[V']$ with $L' = \Q(V')$ contained in $L$.  Hence $\gf_\gp(L'/F) \le 4$, but no such $G_{14}$-class is permitted by amiability.
\end{proof}

\begin{theo}  \label{main}
Let $A$ be a favorable abelian threefold of prime conductor $N$ such that $F=\Q(A[2])$ is amiable. If $B$ is a semistable abelian variety of dimension $3d$ and conductor $N^d$, with $B[2]$  filtered by $\cE = A[2]$, then $B$ is isogenous to $A^d$. 
\end{theo} 

\begin{proof} 
To apply Criterion \ref{Bcrit}, it suffices to show that $\Ext^1_{[2],\un{E}}(\cE,\cE)=0$.  If $K$ is a pair-resolvent  of $F$, then $\fdeg{\Omega_K}{K} \le 2$ since $F$ is  amiable.  Suppose that there is a non-trivial class $[\cV]$ in $\Ext^1_{[2],\un{E}}(\cE,\cE)$.  By Proposition  \ref{GaProp}, $[V]$ is $G_a$-class with $a \ne 0$ and $L = \Q(V)$ properly contains $F$. We next eliminate the other values of $a$.

If $a=20$ or 35, we have both a $G_6$-class $[V']$ and a $G_{14}$-class $[V'']$, according to diagram \eqref{GaDiagram} and Lemma \ref{ab}(ii).  Then Proposition \ref{MakeGaFields} shows that there are two independent quadratic extensions of $K$ in $\Omega_K$ corresponding to $V'$ and $V''$, thus contradicting $\fdeg{\Omega_K}{K} \le 2$.

If $a = 6$ or 14, Proposition \ref{MakeGaFields} provides a quadratic extension $K'$ of $K$ in $\Omega_K$ corresponding to $V$ and so $K' = \Omega_K$.   By Proposition \ref{Cond4}, $\gf_\gp(K'/K) \le 2$ if $a = 6$ and $\gf_\gp(K'/K) \le 4$ if $a = 14$.  Either case violates amiability.  

If $a = 29$, then $\gf_\gp(L/F) \le 4$ by Proposition \ref{Cond4}.  Use \eqref{GaDiagram} and Lemma \ref{ab}(ii) to obtain a $G_{14}$ class $[V']$ with $L' = \Q(V')$ contained in $L$.  Hence $\gf_\gp(L'/F) \le 4$, but no such $G_{14}$-class is permitted by amiability.
\end{proof}

\section{Some data for heptic favorable curves}\label{data} 
Recall Definition \ref{amiable} and Proposition \ref{AmiNec} about amiability.  Consider the subsets of favorable heptic fields $F_0$ enjoying, respectively, the following properties.

\vspace{5 pt}

\noindent\hspace{3pt} $R$:  the number of real places of $F_0$ satisfies $r_1 \le 3$. \\
\hspace*{3pt} $Z$:  the ray class field $\Omega_K = K$, where $K$ is the corresponding pair-resolvent.  \\
\hspace*{3pt} $\gf 4$:   $\fdeg{\Omega_K}{K}=2$, $\gf_\gp(\Omega_K/K)=4$ and the Galois closure of $\Omega_K/\Q$ has group $G_6$.  \\
\hspace*{3pt} $\gf 6$:   $\fdeg{\Omega_K}{K}=2$ and $\gf_\gp(\Omega_K/K)=6$.  \\
\hspace*{3pt} $A$:  $F_0$ is amiable, that is $A = Z \cup \gf 4 \cup \gf 6$.

\vspace{5 pt}

The simplest favorable curves of prime conductor $N$ have the shape  
$$
y^2+y=g(x)=x^7+a_0x^6+a_1x^5+a_2x^4+a_3x^3+a_4x^2+a_5x+a_6
$$
with the  discriminant of $f(x)=1+4g(x)$ equal to $\pm2^{12} N$. The favorable number field $F_0$ defined by a root of $f(x)$ has discriminant $\pm2^6 N$. A small search gave us  38759 favorable curves of prime conductors at most $10^{10}$. Of those, 8171 have at most 3 real places and $N\le 10^9$. We describe the distribution of the amiable curves among them in Table \ref{data_curved}. 

In Column $j$ of Tables  \ref{data_curved} and \ref{data_fields}, we consider primes $N$ between $10^8 j$ and $10^8(j+1)$.  In the respective rows, we count fields $F_0$ in the corresponding sets.  

\begin{table}[h]
\begin{caption} {Amiable  favorable curves with  $r_1\le 3$  and  $N\le 10^9$} \label{data_curved}\end{caption}
\centering{
{\small \begin{tabular}{|c||c|c|c|c|c|c|c|c|c|c||c|}
\hline
$j$ &0&1& 2&3&4& 5&6& 7&8&9& Total\\
\hline
R& 2069& 1041&815&763&706&622&574&538&534&509&8171\cr
\hline
Z&35& 8&13&8& 10& 13& 10&9&  9 &11&126\cr
$ \gf 4$&163& 53& 41& 40& 31& 27& 42& 21& 22& 27&467\cr
$ \gf 6$&256&103& 73& 62& 67& 59& 50& 39& 54& 42&805\cr
\hline
A&454&164&127&110&108&99&102&69&85&80  &1398\cr
\hline
\end{tabular}}}
\end{table}

In addition to the table, we summarize what was found about the even sparser collection of 24132 favorable curves  with $10^9\le N\le 10^{10}$ and at most 3 real places. Of those, 472 satisfy Z, 1110 satisfy $\gf 4$ and 2177 satisfy $\gf 6$ for a total of 3759 amiable curves. The ten amiable  curves of largest conductors found  are given by $y^2=f(x)$ with $f(x)$ in Table \ref{large_curves} and  the last column indicates the aspect of amiability that applies.

\begin{table}[h]
\begin{caption} {Amiable curves with conductor near   $10^{10}$} \label{large_curves}\end{caption}
\centering{
{\small \begin{tabular}{|l|l| c |}
\hline
\multicolumn{1}{|c|}{$f(x)$} &Conductor & \\
\hline
$4x^7 - 12x^6 - 4x^5 - 4x^4 - 12x^3 - 8x^2 + 1$&9936420433 &  $\gf 6$\\
$4x^7 - 8x^5 - 20x^4 + 16x^3 + 48x^2 + 28x + 5$&9942023503 &  $\gf 4$ \\
$4x^7 + 8x^6 - 12x^5 - 28x^4 - 8x^3 + 28x^2 + 24x + 5$&9946958657 &  $\gf 6$\\
$4x^7 - 16x^6 + 28x^5 - 16x^4 - 12x^3 + 16x^2 + 4x - 3$&9950037247  & Z\\
$4x^7 + 12x^5 - 4x^4 - 20x^3 + 20x^2 - 8x + 1$&9951742121 & $\gf 6$ \\
$4x^7 + 12x^6 + 28x^5 + 36x^4 + 32x^3 + 16x^2 - 4x + 1$&9959376079  & $\gf 4$ \\
$4x^7 + 4x^6 - 8x^5 + 4x^4 + 4x^3 - 24x^2 + 24x - 7$&9979868599 & $\gf 6$\\
$4x^7 - 8x^6 + 20x^5 - 36x^4 - 16x^3 + 52x^2 + 4x - 15$&9984318889 & $\gf 4$ \\
$4x^7 + 8x^6 - 4x^5 - 16x^4 + 8x^3 - 4x^2 + 1$&9988954033 & $\gf 4$ \\
$4x^7 + 20x^5 + 12x^4 + 16x^3 + 32x^2 - 16x - 7$&9988955521 & $\gf 4$\\
\hline
\end{tabular}}}
\end{table}
Andrew Sutherland \cite{Suth} has a large database of   hyperelliptic curves of genus 3. Of those,  660 are favorable, including  171 amiable curves.  In fact, the favorable curves in that data base with the eight smallest conductors, namely  
$$
[ 10487, 13399, 18839, 36383, 37321, 38567, 42239, 42473]
$$
are amiable.

\vspace{5 pt}

With no claim of completeness, we gathered 54270 non-isomorphic favorable fields of prime conductor $N\le10^9$.  Some are fields whose Galois closure is the two-division field of a favorable curve.  Others are favorable fields defined by polynomials of the shape 
$$
f=x^7+ 2\,(a_0x^6+a_1x^5+a_2x^4+a_3x^3+a_4x^2+a_5x+2a_6+1),
$$ 
with $a_i\le 20$. We  analyzed these heptic fields as to amiability in Table \ref{data_fields}.

\begin{table}[h]
\begin{caption} {Amiable  favorable fields with  $r_1\le 3$  and  $N\le 10^9$} \label{data_fields}\end{caption}
\centering{
{\small \begin{tabular}{|c||c|c|c|c|c|c|c|c|c|c||c|}
\hline
$j$ &0&1& 2&3&4& 5&6& 7&8&9& Total\\
\hline
R&13889&7265&5572&5343&4688&4319&3994&3812&3658&3328&55868\\
\hline
Z&1746&664&562&490&446&370&367&349&301&291&5586\\ 
$ \gf 4$&1317&591&462&402&316&300&316&250&262&245&4461\\
$ \gf 6$&3308&1467&1122&992&905&809&713&725&655&618&11314\\
\hline
A& 6371&2722&2146&1884&1667&1479&1396&1324&1218&1154&21361\\
\hline
\end{tabular}}}
\end{table}

For comparison, we had found only 1302 out of the  2416 favorable fields with $N\le 31 250 000$ and at most 3 real embeddings in the complete table of heptic fields with discriminant of absolute value at most $2\!\cdot\!10^9$ determined by Driver and Jones in \cite{DJ}.  Of the latter,  530 have property Z, 691 have property $\gf 6$ and 243  have property $\gf 4$  for a total of  $1464$ amiable fields.  

\vspace{10 pt}

\noindent {\em Acknowledgements}.   We thank the anonymous referee for suggestions that helped to improve the exposition.  Research of the second author was partially supported by PSC-CUNY Awards, cycles 50 and 52.  On behalf of all authors, the corresponding author states that there is no conflict of interest.  All accessible data appears in the manuscript.

\appendix \numberwithin{equation}{section}

\section{Conductors for Artin-Schreier-like extensions} \label{ArtinSchreier}

We use the conventions of Serre \cite[IV]{Ser1} for conductors.  Let $G = \Gal(L/K)$ be the Galois group of a finite extension of local fields of residue characteristic $p$.  Write $G_j$ for the $j$-th ramification subgroup of $G$ and $c_{L/K} = \max\{j \, \vert \, G_j \ne \{1\} \}$.  The {\em conductor exponent} of $L/K$ is given by 
$$
\gf(L/K) = 1 + \frac{1}{\vv{G_0}} \left( \vv{G_1} + \dots \vv{G_{c_{L/K}}} \right).
$$
If $H$ is a subgroup of $G$, then $H_j = H \cap G_j$.  In particular, if $G_1 \subseteq H$ and $F$ is the fixed field of $H$, then $H_j = G_j$ for $j \ge 1$ and $c_{L/F} = c_{L/K}$, so
\begin{equation} \label{condH}
\gf(L/K) = 1 +\frac{1}{\fdeg{G_0}{H_0}} \left( \gf(L/F)-1 \right).
\end{equation}

Assume from now on that $K$ is an unramified extension of $\Q_p$ with ring of integers $\cO_K$ and fixed algebraic closure $\ov{K}$.  Finding the field of points of a group scheme $\cV$ over $\cO_K$ from the linear algebra of its associated Honda system often leads to rather complicated congruences.  In \S \ref{ptsEEp}, such congruences were simplified by a suitable change of variables, to arrive at a characteristic 0 version of Artin-Schreier theory.  A key property is that distinct solutions to our Artin-Schreier-like equations differ by units, although the solutions themselves are not integral.

Let $q = p^n$ with $n \ge 1$ and let $F$ be a finite extension of $K(\Mu_{q-1})$ in $\ov{K}$ with ring of integers $\cO_F$, maximal ideal $m_F$ and valuation $v_F$, normalized so that $v_F(p) = e_F$ is the ramification index of $F/\Q_p$.  

\begin{prop} \label{AS}
Let $f(x) = x^q-x+C$, where $C = u w^{-p^m}$ with: \vspace{2 pt}

\centerline{\begin{tabular}{l l l l}
$\bullet$ $u$ a unit in $\cO_{K}$, \hspace{15 pt} &$\bullet$ $w$ in $\gm_F$, \hspace{15 pt} &$\bullet$ $p \nmid v_F(w)$, \hspace{15 pt} &$\bullet$  $0 \le m < n$.  \vspace{2 pt}
\end{tabular}}

\noindent   Let $L$ be the splitting field of $f$.  If $\ord_p(C) > q/(1-q)$, then $\Gal(L/F) \simeq \F_q$ is an elementary abelian $p$-group and $L/F$ is totally ramified, with conductor exponent $\gf(L/F) = v_F(w)+1$.
\end{prop}

\begin{proof}
For $\alpha$ and $\beta$ in $L$ and an ideal $\ga$ of $\cO_L$, write $\alpha=\beta + \rO(\ga)$ if $\alpha-\beta$ is in $\ga$.  To estimate interior terms of a binomial expansion, recall that if $1 \le j \le q-1$, then $\ord_p \binom{q}{j} = n - \ord_p{j} \ge 1$. 

Fix a root $\theta$ of $f$ in $L$, let $g(x) = f(x+\theta)$ and $M = F(\theta)$.   By assumption, $\ord_p(C) < 0$, so $\ord_p(\theta) =  \frac{1}{q} \ord_p(C) > \frac{1}{1-q}$.  For $j = 1, \dots, q-1$, the coefficient of $x^j$ in $(x+\theta)^q$ is in the maximal ideal $\gm_M$.  Indeed, 
$$ 
\textstyle{\ord_p\left(\binom{q}{j} \theta^j\right)  \ge 1 + j \ord_p(\theta) \ge 1 + (q-1)\ord_p(\theta)  > 0}.
$$
Thus, $g(x) = x^q - x + h(x)$, where $h(x)$ is in $\gm_M[x]$ and $h(0) = 0$.

By Hensel's Lemma, every element of $\Mu_{q-1}$ can be refined to a root of $g$ in $M$, so $L = M$.  Furthermore, for $\sigma$ in $\Gal(L/F)$, the root $\sigma(\theta)$ of $f$ has the form $\sigma(\theta) = \theta + \zeta_\sigma + \rO(\gm_L)$, with $\zeta_1 = 0$ if $\sigma = 1$ is the identity and $\zeta_\sigma$ in $\Mu_{q-1}$ otherwise.  Pass to the residue field, to obtain an injective homomorphism $\Gal(L/F) \to \F_q$ given by 
$
\sigma \mapsto \zeta_\sigma \pmod{\gm_L}.
$  
Hence $\Gal(L/F)$ is an elementary abelian $p$-group.

Since $u$ is a unit in the unramified ring $\cO_K$ with perfect residue field, there is a unit $u_1$ in $\cO_K$ satisfying $u_1^{p^m} \equiv u \pmod{p}$.  Set $t = \theta^{p^{n-m}} + u_1 w^{-1}$.   We claim that $t^{p^m} = \epsilon \theta$ for some unit $\epsilon$ in $\cO_L$, thanks to the following estimates.
Note first that  
$$
\ord_p\left(\theta^{p^{n-m}}\right) = p^{-m} \ord_p(C)= -\ord_p(w).
$$
If $m >0$ and $1 \le j \le p^m-1$, we have
$$
\textstyle{\ord_p\left(\binom{p^m}{j \,} (\theta^{p^{n-m}})^j w^{-(p^m-j)}\right) \ge 1- p^m \ord_p(w) = 1 + \ord_p(C)} > \ord_p(\theta), 
$$
where last inequality follows from our lower bound on $\ord_p(C)$.  Hence
$$
t^{p^m} = \theta^{q} + (u_1 w^{-1})^{p^m} + \lambda \theta =  \theta - C + (u_1 w^{-1})^{p^m} + \lambda \theta = \theta + \frac{u_1^{p^m} - u}{w^{p^m}} + \lambda \theta   
$$
for some $\lambda \in \gm_L$.  In addition, 
$$
\textstyle{\ord_p\left((u_1^{p^m} - u)w^{-p^m}\right) \ge \ord_p\left(p w^{-p^m} \right) = 1 +  \ord_p(C) >  \ord_p(\theta)}, 
$$
so $t^{p^m} =  \epsilon \theta$ for some unit $\epsilon$ in $L$, as claimed.  If $m = 0$ take $u_1 = u$, so $t = \theta$. Thus, 
$$
\ord_p(t) = p^{-m} \ord_p(\theta) = - q^{-1} \ord_p(w) \, \text{ for all } m.
$$
Since $p \nmid \ord_p(w)$, the ramification index of $L/F$ is a multiple of $q$.  But $\fdeg{L}{F} \le q$, so $L/F$ is totally ramified of degree $q$ and $L = F(t)$.   Similar estimates of interior terms give 
$$
\sigma(\theta^{p^{n-m}}) = (\theta + \zeta_\sigma+ \rO(\gm_L))^{p^{n-m}} = \theta^{p^{n-m}} + \zeta_\sigma^{p^{n-m}} + \rO(\gm_L).
$$
Then $\sigma(t) - t  = \zeta_\sigma^{p^{n-m}} + \rO(\gm_L)$, since $u_1w^{-1}$ is in $F$.  If $e_L$ is the ramification index of $L/\Q_p$ and $v_L$ is the valuation on $L$ satisfying $v_L(p) = e_L$, we find that
$$
v_L(t) = e_L \ord_p(t) = q e_F \ord_p(t) = -e_F \ord_p(w) = -v_F(w)
$$ 
is prime to $p$.  Hence $\gf = v_F(w)+1$ by \cite[Prop. C.5]{BK5}. 
\end{proof}

\begin{Rem}
If $f(x) = x^q -x+C$, with $C$ in $\cO_F$, then Artin-Schreier theory over $k_F$ implies that the splitting field $L$ of $f(x)$ is unramified over $F$, of degree dividing $p$.  Indeed, as in the proof above, the map $\Gal(L/F) \to \F_q$ is injective, so $\Gal(L/F)$ is an elementary $p$-group and it has order at most $p$, since unramified extensions are cyclic.  In particular, if $\ord_p(C) > 0$, then the elements of $\Mu_{q-1}$ lead to roots of $f$ by Hensel's Lemma, so $L = F$.
\end{Rem}


\begin{thebibliography}{LMF}   
\bibitem[Ab]{Abr} V. A. Abrashkin, Galois modules of group schemes of period $p$ over the ring of Witt vectors, Math. USSR-Izv. {\bf 31} (1988), no. 1, 1--46. 
\bibitem[BC]{BrCo} O. Brinon and B. Conrad, CMI Summer School Notes on p-adic Hodge Theory,  available at \url{http://math.stanford.edu/~conrad/papers/notes.pdf}.
\bibitem[BK1]{BK1} A. Brumer and K. Kramer,  Non-existence of certain semistable abelian varieties, Manus. Math. {\bf 106} (2001), 291--304.
\bibitem[BK2]{BK5} A. Brumer and K. Kramer, Certain abelian varieties bad at only one prime. Algebra and Number Theory, {\bf 12(5)} (2018), 1027--1071.
\bibitem[Co]{Con2} B. Conrad, Finite Group Schemes over Bases with Low Ramification, Compos. Math.  {\bf 119},   (1999), 239--320.
\bibitem[CR]{CR} C.W. Curtis and I. Reiner, Representation Theory of Finite Groups and Associative Algebras, Wiley, 1962.
\bibitem [Di]{Di} L. E. Dickson, Representations of the general symmetric group as linear groups in finite and infinite fields, Trans. Amer. Math. Soc. {\bf 9} (1908), 121-148.
\bibitem[DJ]{DJ} E. Driver and J. Jones, Computing septic number fields, Jl. Number Theory {\bf 202} (2019) 426--429.
\bibitem[DyD]{dyd} F. Diaz y Diaz, Tables minorant la racine $n$-i\`{e}me du discriminant d'un corps de degr\'{e} $n$, Ph.D. Thesis, Publ. Math. Orsay, 1980.
\bibitem[Fo]{Fon3} J.-M. Fontaine, Groupes $p$-divisibles sur les corps locaux, Ast\'{e}risque, {\bf 47-48}, Soc. Math. France, 1977.
\bibitem[Mag]{MAG} W. Bosma, J. Cannon  and C. Playoust. The Magma algebra system. I. The user language. J. Symb. Comp. {\bf 24} (1997), 235--265.
\bibitem [Od]{Odl} A. Odlyzko, Bounds for discriminants and related estimates for class numbers, regulators and zeros of zeta functions:  a survey of recent results, J. de Th\'{e}orie des Nombres de Bordeaux, {\bf 2} (1990), 119--141. 
\bibitem[Sc1]{Sch1} R. Schoof,  Abelian varieties over cyclotomic fields   with everywhere good reduction, Math. Ann. {\bf 325} (2003), 413--448. 
\bibitem[Sc2]{Sch2} R. Schoof,  Abelian varieties over $\Q$ with bad reduction in one prime only, Compos. Math. {\bf 141} (2005), 847--868.
\bibitem[Sc3]{Sch3} R. Schoof, Semistable abelian varieties with good reduction outside 15,  Manus. Math. {\bf 139} (2012), 49--70.
\bibitem[Sch4]{Sch4} R. Schoof, On the modular curve $X_0(23)$, Geometry and Arithmetic, EMS Publishing House, Z\"urich 2012, 317--345.
\bibitem[Se]{Ser1} J.-P. Serre, Local Fields, Lecture Notes in Math. {\bf 67}, Springer-Verlag, 1979.
\bibitem [Su]{Suth} A.  Sutherland, various databases,   \url{https://math.mit.edu/~drew/}.
\bibitem[vdW]{vdW} B. L. van der Waerden, Die Zerlegungs- und Tr\"{a}gheitsgruppe als Permutationsgruppen, Math. Annalen {\bf 111} (1935), 731--733.
\end{thebibliography}
\end{document}